\pdfoutput=1
\documentclass[11pt]{article}
\usepackage{amsmath,amssymb,amsthm,mathtools}
\usepackage{bbm}

\usepackage[svgnames]{xcolor}
\usepackage[
backend=biber,
natbib=true,
style=alphabetic,
maxbibnames=15,
maxalphanames=10,
minalphanames=6,
doi=false,
isbn=false,
url=false,
eprint=false,
backref=true,
]{biblatex}
\usepackage[colorlinks=true,
            citecolor=blue,
            linkcolor=magenta,
            urlcolor=magenta]{hyperref}
\usepackage[capitalize,noabbrev]{cleveref}
\usepackage[a4paper,margin=1in]{geometry}
\usepackage{booktabs}
\usepackage{microtype}
\usepackage{bm}
\usepackage{url}
\usepackage{enumitem}
\usepackage{thmtools}

\newtheorem{theorem}{Theorem}[section]
\newtheorem*{theorem*}{Theorem}
\newtheorem{lemma}[theorem]{Lemma}

\newtheorem{proposition}[theorem]{Proposition}
\newtheorem{corollary}[theorem]{Corollary}
\newtheorem{fact}[theorem]{Fact}
\newtheorem{definition}[theorem]{Definition}

\newtheorem{question}{Question}

\newtheorem{remark}[theorem]{Remark}

\newcommand{\eps}{\varepsilon}
\newcommand{\dtv}{d_\mathrm{TV}} 
\newcommand{\dhel}{d_\mathrm{hel}}
\newcommand{\norm}[1]{\left\|#1\right\|}

\newcommand{\TV}{\mathrm{TV}}
\newcommand{\Hub}{\mathrm{Hub}}
\newcommand{\Adv}{\mathrm{Adv}}
\newcommand{\Sub}{\mathrm{Sub}}
\newcommand{\prob}{\mathbb P}

\newcommand{\ptv}{p^{*}_{\TV}}
\newcommand{\qtv}{q^{*}_{\TV}}
\newcommand{\phub}{p^{*}_{\Hub}}
\newcommand{\qhub}{q^{*}_{\Hub}}
\newcommand{\psub}{p^{*}_{\Sub}}
\newcommand{\qsub}{q^{*}_{\Sub}}

\newcommand{\cl}{c'}                     %
\newcommand{\cu}{c''}                    %

\newcommand{\clHub}{\cl_{\Hub}}
\newcommand{\cuHub}{\cu_{\Hub}}
\newcommand{\clTV}{\cl_{\TV}}
\newcommand{\cuTV}{\cu_{\TV}}

\newcommand{\epsTVup}[1]{\eps_{\TV}^{\mathrm{up}}\!\left(#1\right)}
\newcommand{\epsTVlow}[1]{\eps_{\TV}^{\mathrm{low}}\!\left(#1\right)}
\newcommand{\epsHubup}[1]{\eps_{\Hub}^{\mathrm{up}}\!\left(#1\right)}
\newcommand{\epsHublow}[1]{\eps_{\Hub}^{\mathrm{low}}\!\left(#1\right)}

\newcommand{\nsHub}{n^*_{\Hub}}
\newcommand{\nsTV}{n^*_{\TV}}

\newcommand{\nsSub}{n^*_{\Sub}}

\newcommand{\ns}{n^*}

\newcommand{\nsAHub}{n^*_{\mathrm{A\!-\!\Hub}}}
\newcommand{\nsASub}{n^*_{\mathrm{A}\!-\!\Sub}}
\newcommand{\nsATV}{n^*_{\mathrm{A}\!-\!\TV}}

\newcommand{\ptvpr}{\ptv{}'}   %
\newcommand{\qtvpr}{\qtv{}'}   %
\newcommand{\Ber}{\mathrm{Ber}}

\newcommand{\Hsq}{H^{2}}
\newcommand{\E}{\mathbb{E}}
\newcommand{\Var}{\operatorname{Var}}
\newcommand{\hstat}{h}
\newcommand{\olh}{\overline{h}_n}

\newcommand{\cA}{\mathcal{A}}

\newcommand{\cP}{\mathcal{P}}
\newcommand{\cQ}{\mathcal{Q}}

\newcommand{\cX}{\mathcal{X}}
\newcommand{\cY}{\mathcal{Y}}

\title{On the Sample Complexity of Robust Binary Hypothesis Testing}

\author{
Shankar Vallinayagam$^{1}$ \and
Ankit Pensia$^{2}$ \and
Varun Jog$^{1}$
}

\date{
$^{1}$Department of Pure Mathematics and Mathematical Statistics, University of Cambridge \\
$^{2}$Department of Statistics, Carnegie Mellon University %
}

\begin{document}
\maketitle

\begin{abstract}
We study the sample complexity of robust binary hypothesis testing under three standard contamination models: $\eps$-additive (Huber), $\eps$-subtractive, and $\eps$-total variation (TV), denoted by $\nsHub(\eps)$, $\nsSub(\eps)$, and $\nsTV(\eps)$, respectively. For subtractive contamination, we show that least favourable distributions exist and provide explicit formulas for the same, bringing this model in line with the classical Huber and TV models.  Next we show that in all three models, sample complexity may be highly unstable in the contamination parameter $\eps$, increasing by polynomial factors even for $o(\eps)$ perturbations. Similarly, there may be polynomial factor gaps between the sample complexities when $\eps$ is known exactly versus when it is known up to $o(\eps)$ error. 
Despite the instability of the sample complexity in all models, we show that the sample complexities across models are comparable up to constant-factor rescaling of $\eps$. Specifically, for any fixed $\delta_0>0$, the following hold for all distributions $p$ and $q$:
(i) $\nsHub(\eps) \lesssim \nsTV(\eps) \lesssim \nsHub(2\eps)$,
(ii) $\nsSub(\eps) \lesssim \nsTV(\eps) \lesssim \nsSub((2+\delta_0)\eps)$, and
(iii) $\nsSub(\eps) \lesssim \nsHub(\eps) \lesssim \nsSub((1+\delta_0)\eps)$,
and the scaling constants are tight. Finally, we extend our results to adaptive versions of the contamination models.
\end{abstract}

\tableofcontents

\section{Introduction}\label{sec: intro}

A central problem in statistics is simple binary hypothesis testing. Given i.i.d.\ samples drawn from either $p$ (hypothesis 0) or $q$ (hypothesis 1), the goal is to decide which distribution generated the data while optimising a chosen performance criterion, such as the error probability, asymptotic error rates, or sample complexity. In practice, however, the assumption that samples are drawn exactly from $p$ or $q$ is often unrealistic: the underlying distributions may only be approximately known, and the observed data may be corrupted or contain missing entries. This motivates the study of \emph{robust binary hypothesis testing}, a classical topic in robust statistics, where the i.i.d.\ assumption is replaced by models that explicitly account for data contamination.

The earliest work in this area, due to Huber~\citep{Hub65}, considers a setting where (under hypothesis 0) samples are not drawn exactly from $p$, but instead from an unknown distribution $p'$ lying in a neighbourhood of $p$, referred to as an uncertainty set. \citet{Hub65} introduced two canonical contamination models, depending on the structure of this uncertainty set: \emph{Huber contamination} (or additive contamination) and \emph{total variation (TV) contamination} (or general contamination).

The classical likelihood ratio test between $p$ and $q$, which is optimal when there is no data contamination, may be grossly suboptimal when samples are contaminated. Intuitively, a robust test should limit the impact of ``outlier'' samples; i.e., samples where the likelihood ratio tends to the extremes. One way to do so is to ``clip'' the likelihood ratio so that values above an upper threshold are rounded down to the threshold, and values below a lower threshold are rounded up to the threshold. For Huber and TV contamination, \citet{Hub65} proved a remarkable result that if the upper and lower thresholds are chosen carefully depending on $p$, $q$, and the contamination level $\eps$, the resulting ``clipped likelihood ratio'' test is \emph{minimax optimal} for the usual testing criteria. Moreover, the clipped likelihood ratio test corresponds to the usual likelihood ratio test between a pair of \emph{least-favourable distributions} (LFDs) $p^*$ and $q^*$ in their respective uncertainty sets. In effect, the results in \citet{Hub65} reduce the robust hypothesis testing problem to the \emph{simple} hypothesis testing between $p^*$ and $q^*$; i.e., the optimal achievable errors or sample complexities for the robust problem are the same as those for the simple binary hypothesis testing problem.

In this paper, our goal is to study the \emph{sample complexity} of robust binary hypothesis testing. Sample complexity is the smallest number of samples needed to make the sum of the (worst-case) type-I and (worst-case) type-II error below a constant, say 1/10. The preceding discussion might suggest that this is essentially well-understood, at least for Huber and TV contamination, as the sample complexity of the simple binary hypothesis testing between $p^*$ and $q^*$ is well known to be $\Theta(1/\dhel^2(p^*, q^*))$. 

Unfortunately, this expression is difficult to interpret in terms of $p$, $q$, and $\eps$ because the LFD-formulas in \cite{Hub65} are expressed in terms of the lower and upper clipping thresholds, which themselves are defined implicitly as solutions to fixed-point equations. As a result, the LFD perspective does not yield satisfactory answers to some natural questions of interest, such as:

\begin{question}
\label{ques:1}
How does the sample complexity depend on $\eps$?
\end{question}

The point of interest here is whether the sample complexity, denoted by $\ns(\eps)$, depends smoothly on $\eps$, or whether small changes in $\eps$ can cause large jumps. There is no known formula for the sample complexity of robust hypothesis testing directly in terms of $p$, $q$, and $\eps$. One reason for studying \Cref{ques:1} is to gain insight into what such a formula might look like, if it were to exist. Intuitively, if the sample complexity is highly unstable in $\eps$, then it seems unlikely that there is a simple expression for it in terms of an easy to interpret divergence between $p$ and $q$. Another reason to study \Cref{ques:1} is the possibility of ``black-box'' robustness amplification mechanisms. Could a test that is robust to $\eps$-contamination be modified into one that is robust to $1.01\eps$ contamination at the cost of a constant factor enlargement in the sample complexity? If the sample complexity turns out to be highly unstable in $\eps$, we can refute the existence of any such amplification mechanism for robust hypothesis testing. %

\Cref{ques:1} is also closely related to problem settings where the exact value of contamination is unknown, but is known to lie in an interval $[\eps_1, \eps_2]$. If the contamination was known to be exactly $\eps_1$, we would only need to draw $\ns(\eps_1)$ samples. When the contamination lies in a range, it is intuitive that we should use the clipped likelihood ratio test calibrated for $\eps_2$ contamination and draw $\ns(\eps_2)$ samples. For all the models considered in this paper, this intuition is correct. If $\ns(\eps_2)$ is significantly larger than $\ns(\eps_1)$ despite $\eps_2 = \eps_1 + o(\eps_1)$, this indicates the cost (in terms of increased sample complexity) of not knowing the contamination exactly can be very high. One may also consider the reverse situation where we underestimate contamination; i.e., the contamination is thought be at most $\eps_1$ but is, in fact, $\eps_2 > \eps_1$. In this case, we may ask whether the clipped likelihood ratio test calibrated to $\eps_1$ continues to work when $\eps_2 = \eps_1 + o(\eps_1)$. 

\Cref{ques:1} pertains to fixing a contamination model and analysing the sample complexity as the level of contamination changes. In a different direction, we may also seek to analyse relationships between the different contamination models themselves. Specifically, we consider the following question:

\begin{question}
\label{ques:2}
How do different contamination models compare in terms of sample complexity?
\end{question}
For the same contamination level $\eps$, TV contamination is strictly stronger than Huber contamination, and hence the corresponding sample complexity can only be larger. But how much larger? If the Huber contamination is slightly more, say $1.01\eps$, would it now be stronger than $\eps$-TV contamination? Such comparisons make sense for any pair of contamination models. An important example is \emph{subtractive contamination}, which is a natural third model that complements the additive (Huber) and general (TV) contamination models. A related comparison asks whether allowing the adversary to act adaptively, after seeing the full dataset, changes the sample complexity. Adaptive contamination is typically stronger than non-adaptive (or oblivious) contamination, but how much larger is the sample complexity with adaptive contamination compared to oblivious contamination? 

The final question we study concerns connections between the sample complexity of robust hypothesis testing and differentially private hypothesis testing. Recent results show that, up to constant factor changes in the privacy and robustness parameters, and the sample complexity, differentially private algorithms and robust algorithms are equivalent for binary hypothesis testing~\cite{DwoLei09,AsiUZ23,HopKMN23}. Thus, the sample complexity of robust hypothesis testing may be studied through the lens of privacy as well. As noted in contribution (ii) below, the sample complexity of robust hypothesis testing may be highly unstable in the contamination level. Interestingly, the sample complexity of private hypothesis testing is always stable in the privacy parameter. This leads to an apparent paradox: despite robustness and privacy being equivalent, the sample complexity has polynomial jumps in one setting but not the other. How to resolve this paradox?

\subsection{Our contributions} We summarise our main contributions below.\\

(i) \textbf{Least favourable distributions for subtractive contamination:} Our first contribution is Theorem~\ref{prop: lfd_sub_formulas}, where we bring subtractive contamination in line with the classical Huber and TV contamination models. Specifically, we show that least favourable distributions exist for the hypothesis testing problem with subtractive contamination and we find explicit formulas for these. Because subtractive contamination is a natural counterpart to additive and general contamination, it is surprising that the corresponding LFD theory was not developed earlier, especially given that the theory for the latter two models dates back more than 60 years~\citep{Hub65}.\\

(ii) \textbf{Sample complexity jumps:} Next, we study \Cref{ques:1} in detail for the oblivious (non-adaptive) Huber, subtractive, and TV contamination models. Our main observation is that in all three models, the sample complexity can exhibit big jumps with small changes in $\eps$. Informally, we show the following:

\begin{theorem*}[Theorem~\ref{thm:poly-jumps}, informal]
For each contamination model $A \in \{\Hub,\TV,\Sub\}$, the robust sample complexity
$\ns_A(\eps)$ can be highly unstable in $\eps$.
In particular, there exist distributions $p,q$
and contamination levels $\eps_2=\eps_1 + o(\eps_1)$ such that
\begin{align*}
\ns_A(\eps_1) \ll \ns_A(\eps_2).
\end{align*}
In fact, the same distribution pair $(p,q)$ works in all three models.
\end{theorem*}

On the positive side, we show that sufficiently small contamination, $\eps\lesssim \dhel^2(p,q)$, does not change the uncontaminated sample complexity, and that $O(\eps^2)$ perturbations of the contamination level can change the robust sample complexity by at most constant factors.

We also consider settings where the model is misspecified; i.e., contamination assumed to be $\eps$ but it is actually either $\eps - o(\eps)$ or $\eps + o(\eps)$. In all three models, we show that $o(\eps)$ overestimation of the true contamination may lead to a polynomial increase in the sample size. This is a direct consequence of Theorem~\ref{thm:poly-jumps} and monotonicity of the contamination models. In the opposite direction, $o(\eps)$ underestimation of the true contamination may lead to a complete breakdown of the test. This is shown in the following theorem, stated informally here:
\begin{theorem*}[Theorem~\ref{thm:underestimate_breakdown} (informal)]
For each contamination model $A\in\{\Hub,\TV,\Sub\}$, there
exist distributions $p,q$ and contamination levels $\eps_1<\eps_2$ with $\eps_1=\eps_2-o(\eps_2)$ such that the likelihood ratio test calibrated to
$\eps_1$-contamination fails under $\eps_2$-contamination. More precisely, the sum of the type-I and type-II errors of the $\eps_1$-calibrated test tends to $1$ as $n\to\infty$. In fact, the same distribution pair $(p,q)$ works in all three models.

\end{theorem*}

We prove Theorem~\ref{thm:poly-jumps} and Theorem~\ref{thm:underestimate_breakdown} by analysing the sample complexity of robust hypothesis testing under each model for a carefully chosen pair of distributions on the support $\{1, 2, 3\}$. Despite its simplicity, this example captures several key phenomena and serves as a useful guide for understanding the behaviour of sample complexity across contamination models.\\

(iii) \textbf{Comparisons across contamination models:} We consider \Cref{ques:2} which asks about relationships between the sample complexities across the models. Our main result, summarised below, shows that up to scaling of $\eps$ by constants, the three types of contamination are essentially equally powerful.

\begin{theorem*}[Theorems~\ref{theorem: tv_hub},~\ref{theorem: sub_tv}, and~\ref{theorem: hub_sub} (informal)]
Denote the sample complexities under Huber, TV, and subtractive contamination at level $\eps$ by
$\nsHub(\eps)$, $\nsTV(\eps)$, and $\nsSub(\eps)$, respectively. Then the three contamination
models are equivalent up to constant-factor rescalings of $\eps$. Specifically, for any fixed
$\delta_0>0$, the following hold for all distributions $p$ and $q$:
\begin{itemize}
    \item[(a)] $\nsHub(\eps) \lesssim \nsTV(\eps) \lesssim \nsHub(2\eps)$,

    \item[(b)] $\nsSub(\eps) \lesssim \nsTV(\eps) \lesssim \nsSub((2+\delta_0)\eps)$,

    \item[(c)] $\nsSub(\eps) \lesssim \nsHub(\eps) \lesssim \nsSub((1+\delta_0)\eps)$.
\end{itemize}
The implied constants are universal in part~(a), and may depend on $\delta_0$ in parts~(b)
and~(c), but never on $p$, $q$, or $\eps$. Moreover, these rescalings of $\eps$ are essentially optimal: the factor $2$ in part~(a)
cannot be improved in general, and the slack $\delta_0>0$ in parts~(b) and~(c) cannot
in general be removed.
\end{theorem*}

The main technical tool used here is analysing how, in each model, the upper and lower clips that characterise the LFDs behave with respect to $\eps$, and how the clips for one model at $\eps$ relate to the clips for a different model after scaling $\eps$.\\ %

(iv) \textbf{Adaptive contamination:} We also consider adaptive variants of the three contamination models. Our main result, summarised below, shows that adaptive and oblivious contamination are essentially equally powerful.

\begin{theorem*}[Theorems~\ref{thm: adaptive_simulates} and~\ref{thm: adaptive_coupling} (informal)]
Fix any oblivious contamination model $O \in\{\Hub,\TV,\Sub\}$, and let $A$ denote
its adaptive counterpart. Then the sample complexity under adaptive contamination at
level $\eps$ is sandwiched between the corresponding oblivious sample complexities at
levels $(1-\delta_0)\eps$ and $(1+\delta_0)\eps$. Specifically, for any fixed
$\delta_0>0$,
\begin{align*}
    \ns_O((1-\delta_0)\eps)
    \lesssim
    \ns_{A}(\eps)
    \lesssim
    \ns_O((1+\delta_0)\eps).
\end{align*}
The constants implicit in $\lesssim$ may depend on $\delta_0$, but are independent of
$p$, $q$, and $\eps$. 
\end{theorem*}
As a direct consequence of this theorem, we note that comparison results between Huber, TV, and subtractive contamination as in Theorems~\ref{theorem: tv_hub},~\ref{theorem: sub_tv}, and~\ref{theorem: hub_sub} continue to hold for adaptive contamination after the same constant-factor rescalings of $\eps$. Moreover, polynomial sample complexity jumps as in Theorem~\ref{thm:poly-jumps} occur in the adaptive settings as well. These extensions to adaptive settings are proved using a coupling interpretation of contamination followed by standard concentration inequalities.\\ 

(v) \textbf{Connections to privacy:} Our main contribution here is to resolve the apparent paradox between private and robust hypothesis testing. We show the impact on the sample complexity from transformations from private to robust algorithms or from robust to private algorithms is quite subtle, and once correctly identified, there is no paradox.  As a consequence, we show that it is possible to conclude quadratic jumps in the sample complexity for $O(\eps)$ perturbations in the adaptive-TV contamination model (also called strong contamination model) by appealing to known results in the privacy literature.

\subsection{Paper structure}
In Section~\ref{sec: classical_lfd}, we review the classical least favourable distribution
theory for Huber and TV contamination, and recall how the existence of LFDs reduces robust
hypothesis testing to ordinary binary hypothesis testing between the LFDs. In
Section~\ref{sec: pseudo_lfd}, we introduce subtractive contamination and prove that LFDs
exist for this model, giving explicit formulas analogous to the classical Huber and TV formulas.
In Section~\ref{sec: example}, we study how the robust sample complexity depends on the
contamination parameter $\eps$, and show that small perturbations of $\eps$ can lead to
polynomial-factor jumps in the sample complexity. In Section~\ref{sec: mismatch}, we consider
model misspecification in the contamination level, showing in particular that underestimating
$\eps$ may cause the calibrated clipped likelihood ratio test to fail completely. In
Section~\ref{sec: sandwich}, we prove our main comparison results between Huber, TV, and
subtractive contamination, showing that the corresponding sample complexities are comparable
after constant-factor rescalings of $\eps$. In Section~\ref{sec: adaptive}, we extend
these comparisons to adaptive contamination models and show that adaptive and oblivious
contamination have comparable sample complexities up to an arbitrarily small constant-factor
slack in $\eps$. Finally, in Section~\ref{sec: privacy}, we explore connections between private binary hypothesis testing and its robust counterpart.

\subsection{Related work} 
Hypothesis testing is a fundamental task in statistics with a rich history; see, for example, \cite{NeyPea33,LeCam86,DevLug01,PolWu25}. 
Robust binary hypothesis testing converts a simple binary hypothesis testing into a composite binary hypothesis testing problem.
When the two candidate hypotheses sets are convex (which is the case for TV, Huber, and Subtractive contamination models in our case), the resulting sample complexity (up to constants) is characterized by the minimum Hellinger divergence between the two sets (see, for example, \cite{LeCam86}).
When these composite sets are given by the Huber contamination model or the TV contamination model, 
the classical works of \cite{Hub65} and \citet{HubStra73} developed minimax optimal tests for this problem (see~\cite{LehRC86} for more details). 
In addition to these contamination models, robust hypothesis testing with Hellinger uncertainty sets has been studied in \cite{Birge83,LeCam86,Baraud10}.

Our paper is closely related to a recent line studying the sample complexity of hypothesis testing under \emph{resource constraints}. In particular, several works have analysed the sample complexity of simple binary hypothesis testing under communication constraints and under local and central differential privacy constraints~\citep{CanEtal19, PenEtal23, PenEtal24, PenEtal24b, KazEtal25}. Robustness can also be viewed as arising from a constraint; namely, a (statistical) \emph{sampling constraint} that limits the extent to which clean data can be observed~\citep{Pen23}. Our contribution can thus be interpreted as a sampling-constrained counterpart to the resource-constrained literature on simple binary hypothesis testing. 

Another motivation for studying the robust sample complexity comes from a recent paper \citep{AsiEtal24} that, among other results, proved that the sample complexity under central differential privacy (first established in \cite{CanEtal19}) may be equivalently characterised in terms of a divergence-like quantity between the two distributions when the privacy parameter satisfies $\eps \leq 1$ (high privacy regime). There are many similarities between robustness and privacy, since both objectives drive algorithms towards some version of stability (see \Cref{sec: privacy} for the precise connections). In fact, the expression for the sample complexity in \cite{CanEtal19} is in terms of a Hellinger divergence between two distributions that look remarkably similar to the least favourable distributions from \citet{Hub65}. These results suggested that perhaps the sample complexity of robust hypothesis testing also had a neat expression, and finding such an expression was our original motivation for this work.

The relationship between different types of contamination has received recent attention~\citep{DiaKS17,CanHLLN23, BlaVal25}. For the same level of contamination $\eps$ and a particular estimation problem (Gaussian mean estimation in high dimensions), \cite{DiaKS17} showed that the computational sample complexity (within general, restricted families of efficient algorithms) for Huber and TV contamination adversaries could be super-polynomially different.
Turning to statistical rates, \cite{CanHLLN23} show that for the same contamination $\eps$ and for a particular non-convex testing problem (Gaussian mean testing in high dimensions), the sample complexity with adaptive contamination is (at least) polynomially more than its oblivious counterpart. On the other hand, \cite{BlaVal25} prove a reverse relation showing that the sample complexity with adaptive contamination is at most polynomially more than that with oblivious contamination. The bound in \cite{BlaVal25} is  general, but is limited to the discrete setting.
\cite{CanHLLN23} and \cite{BlaVal25} compare oblivious and adaptive adversaries of the same type for the same level of contamination (for example, oblivious $\eps$-Huber versus adaptive $\eps$-Huber, or oblivious $\eps$-TV versus adaptive $\eps$-TV) so their focus is substantially different from ours. In our setting, establishing a polynomial (in fact, a quadratic) relationship between the oblivious and adaptive sample complexities for the same $\eps$ is relatively straightforward.

Finally, our results imply that the cost (in terms of increased sample complexity) of not knowing $\eps$ exactly may be very high. Results of a similar nature have appeared in recent work~\citep{LuoGao24}, although the problem considered there---providing tight confidence intervals that adapt to the unknown contamination $\eps$---is quite different from the one considered here.

\section{Preliminaries and Classical LFD theory}\label{sec: classical_lfd}

We start by describing the simple binary hypothesis testing setup. Let $p$ and $q$ be two probability distributions on a finite discrete sample space $\cX$\footnote{Throughout the paper, we shall assume the sample space $\cX$ is discrete and finite.}. In simple binary hypothesis testing, given $n$ i.i.d.\ samples $X^n = (X_1,\dots,X_n)$ drawn from either $p$ or $q$, the goal is to decide which of the two distributions generated the samples. A (possibly randomised) test based on $n$ samples is a measurable function $\phi : \cX^n \to \{0,1\},$ where $\phi(X^n)=0$ corresponds to deciding in favour of $p$, and $\phi(X^n)=1$ corresponds to deciding in favour of $q$. The performance of a test $\phi$ is measured by the sum of the worst-case type-I and type-II errors over the uncertainty sets:
\begin{align*}
e_n(\phi;p,q) :=  \prob_{X^n \sim p^{\otimes n}}(\phi(X^n)=1) +  \prob_{X^n \sim q^{\otimes n}}(\phi(X^n)=0).
\end{align*}
The optimal test to minimise $e_n(\phi; p,q)$ is simply the likelihood ratio test, and corresponding error is simply 
\begin{align*}
\min_{\phi} e_n(\phi; p, q) =: e_n^*(p,q) = \left(1 - \dtv(p^{\otimes n}, q^{\otimes n})\right).
\end{align*}
The sample complexity of simple binary hypothesis testing, denoted by $\ns(p,q)$, is the smallest $n$ such that the optimal error is smaller than $1/10$. Equivalently, the sample complexity is
\begin{align*}
\ns(p,q) = \min\{n : \dtv(p^{\otimes n}, q^{\otimes n}) \geq 9/10 \}.
\end{align*}
As the total variation distance does not tensorise, the following fact is used to consider the Hellinger divergence instead:
\begin{fact}\label{fact: hel_tv}
The Hellinger divergence between $p$ and $q$, defined as $\dhel^2(p,q) = \sum_{x \in \cX} \left(\sqrt p(x) - \sqrt q(x) \right)^2$ satisfies
\begin{align*}
\dtv^2(p,q) \leq \dhel^2(p,q) \leq 2\dtv(p,q).
\end{align*}
\end{fact}
The sample complexity of simple binary hypothesis testing is characterised by the Hellinger divergence between the two distributions~\citep{LeCam86}. We state the following fact:
\begin{fact}\label{fact: sc_simple}
The sample complexity of simple binary hypothesis testing satisfies $\ns(p,q) \asymp \frac{1}{\dhel^2(p,q)}.$
\end{fact}

The robust binary hypothesis problem was introduced in~\citet{Hub65}. Let $p$ and $q$ be two probability distributions on a discrete sample space $\cX$, and let $\cP,\cQ \subseteq \Delta(\cX)$ be uncertainty sets around $p$ and $q$, respectively. Given $n$ i.i.d.\ samples $X^n = (X_1,\dots,X_n)$ drawn from an unknown distribution that belongs either to $\cP$ or to $\cQ$, the goal is to decide which of the two uncertainty sets generated the data. A (possibly randomised) test based on $n$ samples is a measurable function $\phi : \cX^n \to \{0,1\},$ where $\phi(X^n)=0$ corresponds to deciding in favour of $\cP$, and $\phi(X^n)=1$ corresponds to deciding in favour of $\cQ$. The performance of a test $\phi$ is measured by the sum of the worst-case type-I and type-II errors over the uncertainty sets:
\begin{align*}
e_n(\phi;\cP,\cQ) := \sup_{P \in \cP} \prob_{X^n \sim P^{\otimes n}}(\phi(X^n)=1) + \sup_{Q \in \cQ} \prob_{X^n \sim Q^{\otimes n}}(\phi(X^n)=0).
\end{align*}
The minimax error $e_n^*(\cP,\cQ)$ is given by minimising the above for over all tests $\phi$; i.e.,
\begin{align*}
e_n^*(\cP,\cQ) := \inf_{\phi}  e_n(\phi;\cP,\cQ).
\end{align*}

Analogous to the simple hypothesis testing setting, the sample complexity of robust hypothesis testing is defined as
\begin{align*}
n^*(\cP,\cQ) := \inf\left\{n \in \mathbb{N} : e_n^*(\cP,\cQ) \le \frac{1}{10}\right\}.
\end{align*}
That is, $n^*(\cP,\cQ)$ is the minimum number of samples such that there exists a test with the sum of its type-I and type-II errors at most $1/10$.\footnote{The choice of $1/10$ is arbitrary. If the target error is $\delta$, the LFD-based
sample complexities acquire the usual multiplicative factor $\log(1/\delta)$, and the
comparisons in this paper continue to hold with the same constant-factor rescalings of
$\eps$ after the standard adjustment of constants.}

A key concept in robust hypothesis testing is that of \emph{least favourable distributions} (LFDs), defined below.

\begin{definition}[Least favourable distributions]\label{def: lfd}
Consider the robust binary hypothesis testing problem between $p$ and $q$ with corresponding uncertainty sets $\cP$ and $\cQ$. A sample $X$ is drawn from some $P \in \cP$ under hypothesis 0 or some $Q \in \cQ$ under hypothesis 1.  Distributions $p^* \in \cP$ and $q^* \in \cQ$ are said to be least favourable distributions for the hypothesis testing problem if for every likelihood ratio test $\phi^*: \cX \to \{0,1\}$ between $p^*$ and $q^*$, the following holds:
\begin{align*}
\sup_{P \in \cP} \prob_{X \sim P}(\phi^*(X) = 1) &\leq \prob_{X \sim p^*}(\phi^*(X) = 1) \quad \text{ and }\\
\sup_{Q \in \cQ} \prob_{X \sim Q}(\phi^*(X) = 0) &\leq \prob_{X \sim q^*}(\phi^*(X) = 0).
\end{align*}
Here, a likelihood ratio test $\phi^*$ is a test parametrised by $T \in \mathbb R$ and $\kappa \in [0,1]$ of the form
\begin{align*}
\prob(\phi^*(x) = 0) =
\begin{cases}
1 &\quad \text{ if } p^*(x)/q^*(x) > T,\\
\kappa &\quad \text{ if } p^*(x)/q^*(x) = T, \text{ and }\\
0 &\quad \text{ if } p^*(x)/q^*(x) < T.
\end{cases}
\end{align*}
\end{definition}

In words, given a single sample, the type-I and type-II errors for $\phi^*$ are maximised by $p^*$ and $q^*$ among all distributions in $\cP$ and $\cQ$, respectively. Consider the following minimax problem that is relevant to this paper:
\begin{align*}
\min_{\phi} \sup_{P \in \cP, Q \in \cQ} \left[ \prob_{X \sim P}(\phi(X)=1) + \prob_{X \sim Q}(\phi(X)=0) \right].
\end{align*}
When least favourable distributions exist, this problem is solved at $(p^*, q^*)$ and the likelihood ratio test $\phi^*$ that minimises  
\begin{align*}
\prob_{X \sim p^*}(\phi(X)=1) + \prob_{X \sim q^*}(\phi(X)=0).
\end{align*}
Using a stochastic domination argument (\citet[Lemma 2]{Hub65} or a more general statement in \citet[Corollary 4.2]{HubStra73}), it can be shown that the optimal test given $n \geq 1$ samples continues to be the likelihood ratio test between $p^*$ and $q^*$. In effect, the existence of least favourable distributions converts the robust hypothesis testing problem to a simple hypothesis testing problem between $p^*$ and $q^*$. Using Fact~\ref{fact: sc_simple}, we may state the following fact for the sample complexity of robust hypothesis testing: 
\begin{fact}[Sample complexity when LFDs exist]
Suppose LFDs $p^*$ and $q^*$ exist for a robust testing problem between $\cP$ and $\cQ$. Then the sample complexity $n^*(\cP, \cQ)$ satisfies the following: 
\begin{align*}
\ns(\cP, \cQ) \asymp \frac{1}{\dhel^2(p^*, q^*)}.
\end{align*}
\end{fact}

Least favourable distributions need not always exist for all robust testing problems. For Huber and TV contamination models~\citep{Hub65} and more generally for uncertainty sets defined via \emph{two alternating capacities}~\citep{HubStra73}, LFDs are guaranteed to exist. We recall the definitions of Huber and TV contamination below.

\begin{definition}[$\eps$-Huber uncertainty set]
Let $p \in \Delta(\cX)$ and $\eps < 1$. The $\eps$-Huber uncertainty set around $p$ is defined as 
$$\cP_\Hub(p, \eps) := \{p' : p' = (1-\eps)p + \eps h, h \in \Delta(\cX)\}.$$
Equivalently,
$$\cP_\Hub(p, \eps) := \{p' : p' \in \Delta(\cX) \text{ and } p'(x) \geq (1-\eps)p(x) \text{ for all } x \in \cX \}.$$
\end{definition}

\begin{definition}[$\eps$-TV uncertainty set]
Let $p \in \Delta(\cX)$ and $\eps < 1$. The $\eps$-TV uncertainty set around $p$ is defined as 
$$\cP_\TV(p, \eps) := \{p' : \dtv(p',p) \leq \eps\}.$$
\end{definition}

The main contribution of \citet{Hub65} is finding explicit formulas for the LFDs for Huber and TV contamination. These formulas are expressed in terms of two clips $\cl$ and $\cu$ which are guaranteed to exist. For $\cl < 1 < \cu$, partition the support into the following three sets depending on whether the likelihood ratio is low, medium, or high:
\begin{align*}
L &\coloneqq \{i : p(i)/q(i) < \cl \},\\
M &\coloneqq \{i : p(i)/q(i) \in [\cl, \cu]\},\\
H &\coloneqq \{i : p(i)/q(i) > \cu \}.
\end{align*}
We recall the formulas for Huber and TV-LFDs below. We note that at the boundaries $\cl$ and $\cu$, the LFD-formulas are continuous, so the closed and open intervals in the definitions of $L$, $M$, and $H$ may be chosen as per our convenience.

\begin{fact}[Huber-LFDs]
For $\eps$-Huber contamination, the LFDs are given by
\begin{align}
\phub(i)=
\begin{cases}
(1-\eps)\cl q(i), & i\in L,\\
(1-\eps)p(i), & i\in M,\\
(1-\eps)p(i), & i\in H,
\end{cases}
\qquad
\qhub(i)=
\begin{cases}
(1-\eps)q(i), & i\in L,\\
(1-\eps)q(i), & i\in M,\\
\frac{(1-\eps)p(i)}{\cu}, & i\in H.
\end{cases}
\label{eq: lfd_hub}
\end{align}
where the clips $\cl$ and $\cu$ are chosen to ensure that $\phub$ and $\qhub$ integrate out to 1; i.e.,
\begin{align}
\cl q(L) - p(L)  = \frac{\eps}{(1-\eps)} \quad \text{ and } \quad 
\frac{p(H)}{\cu} - q(H)  = \frac{\eps}{(1-\eps)}.
\label{eq: clips_hub}
\end{align}
\end{fact}

\begin{fact}[TV-LFDs]
For $\eps$-TV contamination, the LFDs are given by
\begin{align}
\ptv(i)=
\begin{cases}
\frac{\cl\big(p(i)+q(i)\big)}{1+\cl}, & i\in L,\\
p(i), & i\in M,\\
\frac{\cu\big(p(i)+q(i)\big)}{1+\cu}, & i\in H,
\end{cases}
\qquad
\qtv(i)=
\begin{cases}
\frac{p(i)+q(i)}{1+\cl}, & i\in L,\\
q(i), & i\in M,\\
\frac{p(i)+q(i)}{1+\cu}, & i\in H.
\end{cases}
\label{eq: lfd_tv}
\end{align}
where the clips $\cl$ and $\cu$ are chosen to ensure $\dtv(\ptv, p) = \eps$ and $\dtv(\qtv, q) = \eps$:
\begin{align}
\frac{p(H) - \cu q(H)}{1+\cu} = \eps, \quad \text{ and } \quad
\frac{\cl q(L) - p(L) }{1+\cl} = \eps.
\label{eq: clips_tv}
\end{align}
\end{fact}

As a sanity check, we see that in both Huber and TV contamination, the likelihood ratio for the LFD-pair always in the interval $[\cl, \cu]$. Thus, the likelihood ratio test for the LFDs is indeed a clipped-version of the likelihood ratio test for the uncontaminated distributions.

\section{Subtractive contamination}\label{sec: pseudo_lfd}

Subtractive contamination is a natural analogue to additive and general contamination. We show the surprising fact that binary hypothesis testing with subtractive contamination is just as amenable to analyses as Huber and TV contamination. Specifically, we show that least favourable distributions exist for subtractive contamination and we find exact expressions for these.

\subsection{Subtractive contamination models}

In this section, we first define what is meant by subtractive contamination. Our definitions are identical to the established convention~\citep[Chapter~1]{DiaKan23} up to reparametrisation.

\begin{definition}[$\eps$-subtractive uncertainty set]
Let $p \in \Delta(\cX)$ and $\eps < 1$. The $\eps$-subtractive uncertainty set around $p$ is defined as 
$$\cP_\Sub(p, \eps) := \{p' : p' \in \Delta(\cX) \text{ and } p'(x) \leq (1+\eps)p(x) \text{ for all } x \in \cX \}.$$
\end{definition}

With this definition, the sample complexity with $\eps$-subtractive contamination is defined similarly to that with Huber and TV contamination. Specifically, the sample complexity $\nsSub(\eps)$ is the smallest number of samples such that there exists some test under which the sum of worst-case type-I and type-II errors is at most 1/10.

The ``subtractive'' aspect of this contamination becomes clearer via an alternate and  equivalent interpretation of  $p' \in \cP_\Sub(p, \eps)$. Consider the distribution of $X \sim p$ conditioned on some event $E$ with $\prob(E) \geq 1/(1+\eps)$, and call this conditional distribution $p'$. For each $x \in \cX$, define $a(x) = \prob(E|X=x) \leq 1$. Then
\begin{align*}
p'(x) = \prob(X=x|E) =  \frac{a(x) p(x)}{\prob(E)} \leq \frac{p(x)}{\prob(E)} \leq p(x)(1+\eps).
\end{align*}
Hence, $p' \in \cP_\Sub(p, \eps)$. Conversely, given any $p' \in \cP_{\Sub}(p, \eps)$, define $a(x) = \frac{p'(x)}{p(x)(1+\eps)} \leq 1$. Consider $U \sim \mathrm{Unif}[0,1]$ independent of $X$ and define the event 
\begin{align*}
E := \{U \leq a(X)\}.
\end{align*}
Then $\prob(E) = \E a(X) = 1/(1+\eps)$. This shows that subtractive contamination is equivalent to conditioning on an event with probability at least $1/(1+\eps)$. This motivates an alternate definition of subtractive contamination given below:
\begin{definition}[Subtractive contamination as selective censoring]
Let $p \in \Delta(\cX)$ and let $\eps < 1$. Let $p' \in \cP_\Sub(p, \eps)$ and set \(a(x)=p'(x)/((1+\eps)p(x))\) when \(p(x)>0\), and set \(a(x)=0\) when \(p(x)=0\). For $X \sim p$, define $Y = X$ with probability $a(X)$, and $Y \,=\, \perp$ with probability $1-a(X)$. Then the distribution of $Y$ conditioned on $Y \,\neq\, \perp$ is $p'$, and $Y$ is thought of as the result of $X$ subjected to subtractive contamination. %
\end{definition}

Given $N$ i.i.d.\ samples $X_1,\dots, X_N$ from $p$, the contaminated dataset replaces each $X_i$ by $Y_i$, independently over all samples. Removing the $\perp$ symbols from the contaminated dataset, the result is a random-sized dataset of size $n_R \sim \mathrm{Bin}(N, 1/(1+\eps))$ wherein samples are generated i.i.d.\ from $p'$. Thus, the only difference is that instead of having exactly $N$ samples from $p'$ as we had earlier, we have $N_R$ samples from $p'$. Let us call these two models as the fixed-size model and the random-sized model. Note that $\E [N_R] = N/(1 + \eps)$. By a multiplicative Chernoff bound, with probability at least $1 - \exp\left(- \frac{N}{8(1+\eps)} \right)$,
$$\frac{N}{2(1+\eps)} \leq N_R \leq N.$$
Indeed, any fixed-size test \(T_N\) can be lifted by drawing \(M=CN\) pre-censoring
samples and applying \(T_N\) to the first \(N\) retained samples, declaring arbitrarily if
fewer than \(N\) remain. Conditional on retaining at least \(N\) samples, these inputs are
i.i.d. from the corresponding subtractive contamination set; the failure event has
arbitrarily small constant probability by Chernoff for large enough universal \(C\). Conversely, any test for the random-sized model can be applied directly in the fixed-size model since the latter provides exactly $N$ samples from $p'$. Therefore, the sample complexities of the two models match up to constant factors.

\subsection{LFDs for subtractive contamination} 

We first argue that it is natural to expect LFDs to exist for subtractive contamination. 

\begin{lemma}[Existence of LFDs for subtractive contamination]\label{lemma: exists_lfd}
Least favourable distributions $(p^*, q^*)$ exist for robust binary hypothesis testing with subtractive contamination.
\end{lemma}

This lemma is proved by showing that subtractive uncertainty set can be associated with the set-valued function $v(A) = \min\{(1+\eps)p(A), 1\}$, which we show is a two-alternating capacity. The existence of LFDs follows from the results of \citet{HubStra73}. The proof of the lemma is deferred to Appendix~\ref{app: exists_lfd}.

Lemma~\ref{lemma: exists_lfd} only guarantees the existence of LFDs without giving explicit formulas. In \citet{Hub65}, such explicit formulas were presented for Huber and TV contamination. We shall now do the same for subtractive contamination. 

\begin{theorem}[LFDs for subtractive contamination]\label{prop: lfd_sub_formulas}
For clips $\cl < 1 < \cu$, define the sets $L, M,$ and $H$ as before based on thresholding the likelihood ratio. Let $\bar H = \{x : p(x)/q(x) = \infty\}$ and $\bar L = \{x : p(x)/q(x) = 0\}$. 
Consider the fixed point equations:
\begin{align}
p(H) - \cu q(H) &= \frac{\eps}{1+\eps}, \quad \text{ and } \quad q(L) - \frac{p(L)}{\cl} = \frac{\eps}{1+\eps}. \label{eq: clips_sub}
\end{align}
\begin{itemize}
\item[(a)] If solutions $\cl$ and $\cu$ can be found to the fixed point equations~\eqref{eq: clips_sub}, then LFDs are given by
\begin{align}
\psub(i)=
\begin{cases}
(1+\eps)p(i), & i \in L,\\
(1+\eps)p(i), & i \in M,\\
\cu(1+\eps)q(i), & i \in H,
\end{cases}
\qquad
\qsub(i)=
\begin{cases}
\frac{(1+\eps)}{\cl}\,p(i), & i \in L,\\
(1+\eps)q(i), & i \in M,\\
(1+\eps)q(i), & i \in H,
\end{cases}
\label{eq:lfd_sub}
\end{align}

\item[(b)] A solution does not exist for $\cu$ only when $p(\bar H) > \eps/(1+\eps)$. When this happens, the LFD $\psub$ is
\begin{align}
\psub(i)=
\begin{cases}
(1+\eps)p(i), & i \notin \bar H,\\
(1+\eps)p(i) \left( 1 - \frac{\eps}{(1+\eps)p(\bar H)}\right) & i \in \bar H.
\end{cases}
\end{align}
In this case, the formula for \(\qsub\) is still given by the lower-clip formula in \eqref{eq:lfd_sub} when the lower clip exists, and by part (c) when the lower clip also fails to exist.

\item[(c)] A solution does not exist for $\cl$ only when $q(\bar L) > \eps/(1+\eps)$. When this happens, the LFD $\qsub$ is
\begin{align}
\qsub(i)=
\begin{cases}
(1+\eps)q(i), & i \notin \bar L,\\
(1+\eps)q(i) \left( 1 - \frac{\eps}{(1+\eps)q(\bar L)}\right) & i \in \bar L.
\end{cases}
\end{align}
In this case, the formula for \(\psub\) is still given by the upper-clip formula in \eqref{eq:lfd_sub} when the upper clip exists, and by part (b) when the upper clip also fails to exist.
\end{itemize}
\end{theorem}

\begin{remark}
The upper clip doesn't exist only when there is significant mass for $p$ at likelihood ratios of $+\infty$, and the lower clip doesn't exist only when $q$ has significant mass at likelihood ratios of $0$. Note that without loss of generality, we may take $\bar H$ and $\bar L$ to be singletons when they are non-empty. In case (b) above, if we denote the singleton point as $x = \bar H$ then the LFD is  easier to state: $\psub(x) = (1+\eps)p(x)$ for $x \notin \bar H$, and $\psub(\bar H) = (1+\eps)p(\bar H) - \eps$. Similar simplification may be done in formula (c) as well. Both clips are guaranteed to exist when $p$ and $q$ are mutually absolutely continuous, which can be ensured by perturbing them slightly. The settings when clips don't exist are often easier to analyse.
\end{remark}

\begin{proof}[Proof of Theorem~\ref{prop: lfd_sub_formulas}]
We may assume without loss of generality that $\dtv(p,q) > \frac{\eps}{1+\eps}.$ Indeed, if \(\dtv(p,q) \leq \eps/(1+\eps)\), then the two subtractive uncertainty sets intersect. To see this, note that
\begin{align*}
    (1+\eps)\sum_{x\in\cX}\min\{p(x),q(x)\}
    =
    (1+\eps)(1-\dtv(p,q))
    \geq 1.
\end{align*}
Hence there exists a distribution \(r\in\Delta(\cX)\) such that
\begin{align*}
    r(x) \leq (1+\eps)\min\{p(x),q(x)\}
\end{align*}
for all \(x\in\cX\). Therefore $r \in \cP_\Sub(p,\eps)\cap \cP_\Sub(q,\eps).$ In this case the robust testing problem is degenerate, and we may take $\psub=\qsub=r$ as least favourable distributions. Thus, in the nontrivial case, we may assume
\begin{align*}
    \dtv(p,q) > \frac{\eps}{1+\eps}.
\end{align*}

We now address the existence of $\cu$. First, observe that the left hand in~\eqref{eq: clips_sub} may be expressed as
\begin{align*}
p(H) - \cu q(H) = \sum_{x \in \cX} (p(x) - \cu q(x))_+,
\end{align*}
which shows it is monotonically decreasing in $\cu$. As $\cu = 1$, this quantity equals $\dtv(p,q) > \frac{\eps}{1+\eps}$. As $\cu \to \infty$, the left hand side reaches its minimum possible value which is $p(\bar H)$. Thus, a solution exists for $\cu$ if and only if $p(\bar H) \leq \eps/(1+\eps)$. A similar argument shows that a solution exists for $\cl$ iff $q(\bar L) \leq \eps/(1+\eps)$. 

We now proceed with the proof that the stated formulas are indeed LFDs. The proof is along similar lines of the Huber and TV proofs in \citet{Hub65}. Let us first consider case (a) when both clips exist. Let $r(x)=\frac{p(x)}{q(x)}$ and let the clips be $\cl < 1 <\cu$. Let the clipped likelihood ratio be defined as
\begin{align*}
s(x)
:=
\min \left\{\cu,\max\left\{\cl, r(x)\right\}\right\}.
\end{align*}
Given samples $(x_1,\dots, x_n)$, let
\begin{align*}
\gamma(x_1, \dots, x_n) := \sum_{i=1}^n \log s(x_i).
\end{align*}
Consider any likelihood ratio test $\phi$ between $\psub$ and $\qsub$ given by
\begin{align*}
\prob(\phi(x_1,\dots,x_n) = 0)
= 
\begin{cases}
1 &\text{ if } \gamma(x_1,\dots,x_n) > T,\\
\kappa &\text{ if } \gamma(x_1,\dots,x_n) = T,\\
0 &\text{ if } \gamma(x_1,\dots,x_n) < T.\\
\end{cases}
\end{align*}
For any $P \in \cP_\Sub(p,\eps)$, the type-I error for this test is
\begin{align*}
\prob_{X \sim P}(\phi = 1) &= \prob_{X \sim P}(\gamma(X_1,\dots,X_n) < T) + (1-\kappa) \prob_{X \sim P}(\gamma(X_1,\dots,X_n) = T)\\
&= (1-\kappa)\prob_{X \sim P}(\gamma(X_1,\dots,X_n) \leq T) + \kappa \prob_{X \sim P}(\gamma(X_1,\dots,X_n) < T).
\end{align*}
The statistic $\gamma$ is a sum of $n$ i.i.d.\ terms. We show that the distribution of each term satisfies a stochastic domination property. Specifically, we claim that the following inequality holds for all $t$:
\begin{align}\label{eq: sub_less}
\prob_{X \sim P}(s(X) < t) \leq \prob_{X \sim \psub(\eps)}(s(X) < t),
\end{align}
and similarly
\begin{align}\label{eq: sub_leq}
\prob_{X \sim P}(s(X) \le t) \leq \prob_{X \sim \psub(\eps)}(s(X) \le t).
\end{align} 
These bounds are trivially true if $t \notin [\cl, \cu]$, so we may consider $t \in [\cl(\eps), \cu(\eps)]$. In this case, we have
\begin{align*} 
\prob_{X \sim P}(s(X) < t) &\stackrel{(a)}\leq (1+\eps)\prob_{X \sim p}(s(X) < t)\\
&\stackrel{(b)}= \prob_{X \sim \psub(\eps)}(s(X) < t).
\end{align*}
Here, in $(a)$ we used that $P(x) \leq (1+\eps)p(x)$, and in $(b)$ we used the LFD formula~\eqref{eq:lfd_sub}. A similar proof goes through for $\leq t$ as well, and we conclude that
\begin{align*}
\prob_{X \sim P} (\phi = 1) \leq \prob_{X \sim \psub(\eps)} (\phi = 1).
\end{align*}
By a standard coupling, stochastic domination of $s(X)$ on each marginal implies stochastic domination of $\gamma(X^n)=\sum_i\log
s(X_i)$, from which the bound on the type-I error follows. This is precisely the property LFDs need to satisfy as noted in Definition~\ref{def: lfd}.

If $\cu_{\Sub}(\eps)$ does not exist, then $\psub(\eps) = (1+\eps)p$ for all $x \notin \{x: p(x)/q(x) = \infty\}$. Thus, for all finite $t$, inequalities~\eqref{eq: sub_less} and~\eqref{eq: sub_leq} continue to hold. A similar argument works for bounding the type-II error as well, which completes the proof.
\end{proof}

\begin{remark}
The proof of Theorem~\ref{prop: lfd_sub_formulas} also gives the LFDs when the contamination
levels around $p$ and $q$ are different. Specifically, consider testing between
$\cP_\Sub(p,\eps_1)$ and $\cP_\Sub(q,\eps_2)$. The fixed-point equations become
\[
    p(H)-\cu q(H)=\frac{\eps_1}{1+\eps_1},
    \qquad
    q(L)-\frac{p(L)}{\cl}=\frac{\eps_2}{1+\eps_2},
\]
and, whenever the corresponding clips exist, the LFDs are
\[
\psub(i)=
\begin{cases}
(1+\eps_1)p(i), & i\in L,\\
(1+\eps_1)p(i), & i\in M,\\
\cu(1+\eps_1)q(i), & i\in H,
\end{cases}
\qquad
\qsub(i)=
\begin{cases}
\dfrac{1+\eps_2}{\cl}p(i), & i\in L,\\
(1+\eps_2)q(i), & i\in M,\\
(1+\eps_2)q(i), & i\in H.
\end{cases}
\]
The cases in which the clips do
not exist are handled exactly as in Theorem~\ref{prop: lfd_sub_formulas}, with
$\eps_1/(1+\eps_1)$ replacing $\eps/(1+\eps)$ for the upper clip and
$\eps_2/(1+\eps_2)$ replacing $\eps/(1+\eps)$ for the lower clip.
\end{remark}

\section{Sample complexity dependency on $\eps$}\label{sec: example}

In this section, we explore \Cref{ques:1} in detail. We start by stating some baseline bounds on the sample complexity that are well-known in the literature. Then we construct an explicit example for which the sample complexity is highly unstable with respect to the contamination parameter $\eps$. Finally, we organise our observations from this example which guide the research questions explored in the rest of the paper.
 
\subsection{Baseline and small-$\eps$ bounds}\label{sec: prelim_baseline}
Our first result is folklore, but we repeat it here for completeness.

\begin{proposition}[Folklore]\label{prop:baseline}
Let $p$ and $q$ be probability distributions and let $\eps \le \frac{\dtv(p,q)}{4}$. Then  
\begin{align*}
    \frac{1}{\dtv(p,q)} \, \lesssim \, \nsHub(\eps),\, \nsTV(\eps),\, \nsSub(\eps)\, \lesssim \,\frac{1}{\dtv^2(p,q)}.
\end{align*}
\end{proposition}
An immediate consequence is that any change in the value of $\eps$ or in the contamination model can cause at most a quadratic jump in the sample complexity. The lower bound can be shown even for the uncontaminated setting (i.e.\ $\eps=0$) whereas the upper bound is proved by analysing Scheffe's test for TV-contamination, which also implies that the same bound holds for Huber and subtractive contamination. The proofs are standard and are deferred to Appendix~\ref{app: baseline}.

Next, we consider what happens when $\eps$ is very small. It is natural to expect that as $\eps$ becomes sufficiently small, the sample complexity with $\eps$-contamination will be almost the same as that without any contamination. In \citet{PenEtal23}, such a bound was established for $\eps \lesssim \dtv^2(p,q)$. Here, we prove such a result for $\eps \lesssim \dhel^2(p,q)$, which improves upon the range of $\eps$ addressed in \citet{PenEtal23} .

\begin{proposition}[Sample complexities for small $\eps$]\label{prop:small_eps}
Suppose $\eps \le \frac{\dhel^2(p,q)}{9}$. Then 
$\nsTV(\eps) \asymp \nsHub(\eps) \asymp \nsSub(\eps) \asymp \frac{1}{\dhel^2(p,q)}$. 
\end{proposition}
The proof follows by a simple application of the triangle inequality for Hellinger distance. Specifically, for any $p'$ and $q'$ in the respective uncertainty sets, we can show that $\dhel(p', q') \asymp \dhel(p,q)$. Thus, the smallest Hellinger distance between the uncertainty sets remains to be $\Theta(\dhel(p,q))$, and so the sample complexity remains $\asymp 1/\dhel^2(p,q)$.
The proof details are deferred to Appendix~\ref{app: small_eps}.

\subsection{Distribution family exhibiting instability}\label{subsection: family}

When $\eps \lesssim \dhel^2(p,q)$ the robust sample complexity coincides with the classical one and is essentially constant in $\eps$. We now show that beyond this regime, the behaviour may change drastically. We do this by constructing an example such that a small change in the contamination parameter $\eps$ causes a polynomial jump in the sample complexity. Our construction ensures that the contamination parameters are always at most $\dtv(p,q)/4$ to highlight that these jumps do not happen because of the uncertainty  sets intersecting, but because of the inherent mismatch between the Hellinger metric---which characterises the sample complexity---and the shapes of uncertainty sets. In particular, the sample complexity jumps do not occur if the uncertainty sets around $p$ and $q$ are Hellinger balls instead of the TV, Huber, or subtractive uncertainty sets.

Our example showing the instability of sample complexity in $\eps$ is very simple, and yet quite illustrative. Consider the following distributions supported on the three point set $\{1, 2, 3\}$:
\begin{align*}
p(1) &= \frac{1}{2} - 10\eps, &\quad q(1) &= \frac{1}{2},\\
p(2) &= \frac{1}{2} + 8\eps, &\quad q(2) &= \frac{1}{2},\\
p(3) &= 2\eps           &\quad q(3) &= 0.
\end{align*}
The total variation distance is $\dtv(p,q) = 10\eps$. In all our calculations, we'll assume contamination levels of $\eps$ or lower, which ensures that the uncertainty sets are strictly separated for TV, Huber, and subtractive contamination. 

The key intuition underlying the sample complexity jumps in all three models is the following. Observe that the most informative symbol is symbol $\{3\}$: its presence is irrefutable evidence for $p$, and it is the largest contributor to the Hellinger divergence between $p$ and $q$ (contributing $\Theta(\eps)$ compared to $\Theta(\eps^2)$ from $\{1\}$ and $\{2\}$). We show that if the contamination level is large enough to completely cancel out the impact of symbol $\{3\}$, then the sample complexity is very high. But if the contamination level is slightly smaller and cannot completely cancel out the impact of $\{3\}$, the resulting sample complexity drops down significantly. 

\paragraph{TV contamination:} Consider two possible contamination parameters: $\eps_2 = \eps$ and $\eps_1 = \eps - \delta$, where $\delta \le \eps/2$ is thought of as a small perturbation to the contamination parameter. Observe that under $\eps$-contamination, the LFDs are given by
\begin{align*}
p^*_{\eps_2}(1) &= \frac{1}{2} - 9\eps, &\quad q^*_{\eps_2}(1) &= \frac{1}{2} - \eps,\\
p^*_{\eps_2}(2) &= \frac{1}{2} + 8\eps, &\quad q^*_{\eps_2}(2) &= \frac{1}{2},\\
p^*_{\eps_2}(3) &= \eps           &\quad q^*_{\eps_2}(3) &= \eps.
\end{align*}
It is easy to see that
\[
\dhel^2(p^*_{\eps_2},q^*_{\eps_2})
=\left(\sqrt{\tfrac12-9\eps}-\sqrt{\tfrac12-\eps}\right)^2
 +\left(\sqrt{\tfrac12+8\eps}-\sqrt{\tfrac12}\right)^2
=\Theta(\eps^2).
\]
So the sample complexity must satisfy $\nsTV(\eps_2) \asymp \frac{1}{\eps^2}$. Now set $\eps_1 = \eps-\delta$. We may calculate the LFDs in a similar manner and directly evaluate the Hellinger divergence. Observe that the contribution to the Hellinger divergence from $\{1\}$ and $\{2\}$ is $\Theta(\eps^2)$. We shall show that the contribution from $\{3\}$ may be much larger. We have 
\begin{align*}
p^*_{\eps_1}(3) &= p(3) - \eps_1 = \eps+\delta \quad \text{ and }\\
q^*_{\eps_1}(3) &= q(3) + \eps_1 = \eps - \delta.
\end{align*}
Hence, 
\begin{align*}
\dhel^2(p^*_{\eps_1}, q^*_{\eps_1}) &\asymp \eps^2+ \left(\sqrt{p_{\eps_1}(3)} - \sqrt{q_{\eps_1}(3)}\right)^2 \\
&= \eps^2 + \left(\sqrt{\eps+\delta} - \sqrt{\eps-\delta}\right)^2\\
&= \eps^2 + \eps\left(\sqrt{1 + \frac{\delta}{\eps}} - \sqrt{1 - \frac{\delta}{\eps}} \right)^2\\
&\asymp \eps^2 + \eps \left( \frac{\delta^2}{\eps^2}\right) \tag{$\sqrt{1+x} = 1+ x/2 + o(x)$}\\
&\asymp \eps^2 + \frac{\delta^2}{\eps}.
\end{align*}
Choosing $\delta = \eps^{1+t}$ for some $0 < t < 1/2$, we have shown
\begin{align*}
\nsTV(\eps_2) \asymp \frac{1}{\eps^2} \quad \text{and} \quad \nsTV(\eps_1) \asymp \frac{1}{\eps^{1+2t}}
\end{align*}
In particular, for $t = 0$, we see that perturbations on the same order of $\eps$ lead to a quadratic jump in the sample complexity: from $1/\eps$ to $1/\eps^2$. In general, for $0< t < 1/2$, perturbations on the scale of $\eps^{1+t} = o(\eps)$ lead to polynomial jumps from $1/\eps^{1+2t}$ to $1/\eps^2$. This shows that even if we know $\eps$ very well, up to errors that are as small as $\eps^{3/2-\eta}$ for $\eta>0$, the sample complexity at $\eps$ and its perturbed value could be polynomially far apart. 

\paragraph{Huber contamination:} We continue using the same construction of $p$ and $q$ as used in the TV-contamination setting. Consider $\eps_2 = 2\eps/(1+2\eps) = 2\eps + O(\eps^2)$. The LFDs are given by
\begin{align*}
p^*_{\eps_2}(1) &= p(1)+(p(2) + p(3))\eps_2 &\quad q^*_{\eps_2}(1) &= q(1)(1-\eps_2)\\
p^*_{\eps_2}(2) &= p(2)(1-\eps_2) &\quad q^*_{\eps_2}(2) &= q(2)(1-\eps_2)\\
p^*_{\eps_2}(3) &= p(3)(1-\eps_2) &\quad q^*_{\eps_2}(3) &= (q(1) + q(2))\eps_2.
\end{align*}
Crucially, we see that 
\begin{align*}
p^*_{\eps_2}(3) = (1-\eps_2)2\eps = 2\eps/(1+2\eps) = \eps_2 = q^*_{\eps_2}(3).
\end{align*}
Thus, the Hellinger divergence between $p^*_{\eps_2}$ and $q^*_{\eps_2}$ is determined by the values at 1 and 2, and is easily seen to be $\Theta(\eps^2)$. This gives
\begin{align*}
\nsHub(\eps_2) \asymp \frac{1}{\eps^2}.
\end{align*}
Now, just like in the TV case, we choose a value of $\eps_1$ slightly smaller than $\eps_2$ such that it is impossible to cancel out the contribution to the Hellinger divergence from $\{3\}$. In particular, choose $\eps_1 = \eps_2-\eps^{1+t}$ for $0<t<1/2$. This means $\eps_1 = 2\eps - \Theta(\eps^{1+t})$. Carrying out a similar calculation as above, we have
\[
\left(\sqrt{2\eps(1-\eps_1)}-\sqrt{\eps_1}\right)^2
=\Theta(\eps^{1+2t}),
\]
while the first two coordinates contribute \(O(\eps^2)\).
Hence, $\dhel^2(p^*_{\eps_1}, q^*_{\eps_1}) \asymp \eps^{1+2t}$ and we observe the same phenomenon as before, that is,
\begin{align*}
\nsHub(\eps_2) \asymp \frac{1}{\eps^2} \quad \text{and} \quad \nsHub(\eps_1) \asymp \frac{1}{\eps^{1+2t}}
\end{align*}
and $\eps_1 = \eps_2 - o(\eps_1)$. The same $3/2$ threshold appears in this calculation as well.

\paragraph{Subtractive contamination:} We continue using the same construction of $p$ and $q$ as used in the TV-contamination setting. Choose $\eps_2 = 2\eps/(1-2\eps) = 2\eps + O(\eps^2)$. The LFDs are given by 
\begin{align*}
p^*_{\eps_2}(1) &= p(1)(1+\eps_2) &\quad q^*_{\eps_2}(1) &= q(1)(1-\eps_2)\\
p^*_{\eps_2}(2) &= p(2)(1+\eps_2) &\quad q^*_{\eps_2}(2) &= q(2)(1+\eps_2)\\
p^*_{\eps_2}(3) &= 0 &\quad q^*_{\eps_2}(3) &= 0.
\end{align*}
Note that $p^*_{\eps_2}$ is a valid distribution as 
\begin{align*}
p^*_{\eps_2}(1) + p^*_{\eps_2}(2) = (1+\eps_2)(p(1) + p(2)) = (1+\eps_2)(1-2\eps) = 1.
\end{align*}
We can check that $\dhel^2(p^*_{\eps_2}, q^*_{\eps_2}) \asymp \eps^2$, and so 
\begin{align*}
\nsSub(\eps_2) \gtrsim \frac{1}{\eps^2}.
\end{align*}
Pick $0<t<1$ and set $\eps_1 = \eps_2 - \eps^{1+t} = 2\eps - \eps^{1+t} + O(\eps^2)$. Calculating the Hellinger divergence between the LFDs $p^*_{\eps_1}$ and $q^*_{\eps_1}$ as before, we have
\[
p^*_{\eps_1}(3)=2\eps(1+\eps_1)-\eps_1=\Theta(\eps^{1+t}),\qquad
q^*_{\eps_1}(3)=0,
\]
so the coordinate \(3\) contribution is \(\Theta(\eps^{1+t})\), and the other coordinates contribute \(O(\eps^2)\). %
Thus, $\dhel^2(p^*_{\eps_1}, q^*_{\eps_1}) \asymp  \eps^{1+t}$ and we observe a similar polynomial jump in the sample complexity, that is
\begin{align*}
\nsSub(\eps_2) \asymp \frac{1}{\eps^2} \quad \text{and} \quad \nsSub(\eps_1) \asymp \frac{1}{\eps^{1+t}}
\end{align*}
where $\eps_1 = \eps_2 - o(\eps_1)$.

We summarise our findings in the following theorem:

\begin{theorem}[Polynomial jumps in sample complexity]\label{thm:poly-jumps}
Let $p$ and $q$ be defined by
\begin{align*}
p(1) &= \frac{1}{2}-10\eps,
&
q(1) &= \frac{1}{2},\\
p(2) &= \frac{1}{2}+8\eps,
&
q(2) &= \frac{1}{2},\\
p(3) &= 2\eps,
&
q(3) &= 0.
\end{align*}
Then $\dtv(p,q)=10\eps$. Moreover, the following hold for all $\eps$ sufficiently small (depending only on the constants in $p,q$): %

\begin{enumerate}[label=(\roman*)]
\item For TV contamination, for every $0<t<1/2$, if $\eps_2 = \eps$ and $\eps_1 = \eps-\eps^{1+t}$,
then $0<\eps_1<\eps_2\leq \dtv(p,q)/4$, and $\nsTV(\eps_2)
\asymp
\frac{1}{\eps^2}$ and $\nsTV(\eps_1)
\asymp
\frac{1}{\eps^{1+2t}}.$

\item For Huber contamination, for every $0<t<1/2$, if $\eps_2 = \frac{2\eps}{1+2\eps}$ and 
$\eps_1 = \eps_2-\eps^{1+t},$ then $0<\eps_1<\eps_2\leq \dtv(p,q)/4$, and
$\nsHub(\eps_2) \asymp \frac{1}{\eps^2}$
and
$\nsHub(\eps_1) \asymp \frac{1}{\eps^{1+2t}}.$

\item For subtractive contamination, for every $0<t<1$, if
$\eps_2 = \frac{2\eps}{1-2\eps},$ and $\eps_1 = \eps_2-\eps^{1+t},$ then $0<\eps_1<\eps_2\leq \dtv(p,q)/4$, and $\nsSub(\eps_2) \asymp \frac{1}{\eps^2}$ and $\nsSub(\eps_1) \asymp \frac{1}{\eps^{1+t}}.$
\end{enumerate}
\end{theorem}

\begin{remark}
The fact that $p$ and $q$ have different supports is not critical to the above result. We may allow $p(3) = 3\eps$ and $q(3) = \eps$ (and adjust the other values suitably) and still get polynomial jumps in the sample complexity. Here, the $\eps$-TV contamination will be able to cancel the impact of $\{3\}$ but $(\eps-\delta)$-TV contamination cannot. From this observation (and analogous observations in the Huber and subtractive settings) we may derive identical conclusions about the instability of the sample complexity for such an example. Thus, even under a bounded likelihood ratio condition on $p$ and $q$, such polynomial jumps in the sample complexity cannot be ruled out.
\end{remark}

\subsection{Observations and consequences} 

We make several observations.

\begin{itemize}
\item[(i)] \emph{Instability in all models:} Our primary observation from Theorem~\ref{thm:poly-jumps} is that in all three contamination models, the sample complexity can change drastically for small changes in $\eps$. It is interesting to note that for subtractive contamination, the sample complexity with a $\eps^{1+t}$-sized perturbation ($\asymp 1/\eps^{1+t}$) is smaller than the sample complexity under the Huber and TV-contamination settings ($\asymp 1/\eps^{1+2t}$). This difference is because subtractive contamination is, in some sense, weaker than Huber or TV contamination. To be precise, any $q'$ within the subtractive uncertainty set has to satisfy $q'(3) = 0$ (as  $q(3)=0$). In contrast, the Huber and TV-contamination settings allow $q(3)$ to be non-zero.

\item[(ii)] \emph{Different degrees of instability:} In the Huber and TV settings, our examples show that perturbations larger than $\sim \eps^{3/2}$ may lead to polynomial jumps in the sample complexity. For the subtractive adversary, the situation is even worse: perturbations larger than $\sim \eps^2$ may cause polynomial jumps. 

\item[(iii)] \emph{TV-contamination is strictly worse:} In Lemma~\ref{lemma: containment} we establish some properties of TV, Huber, and subtractive uncertainty sets. In particular, for the same contamination $\eps$, the TV uncertainty set contains both the Huber and subtractive uncertainty sets. Thus, TV-contamination is stronger than Huber or subtractive contamination. Our example shows that it can be \emph{much} stronger. Specifically, in our example we have $\nsTV(\eps) \asymp 1/\eps^2$, which is much larger than $\nsHub(\eps) \asymp \nsSub(\eps) \asymp \frac{1}{\eps}$. This is  because TV contamination can completely cancel out the impact of symbol $\{3\}$, which is the largest contributor to the Hellinger divergence, but Huber and subtractive contamination cannot.

\item[(iv)] \emph{Sandwich bounds:} In our example, both Huber and subtractive contamination \emph{can} cancel the impact of symbol $\{3\}$ at $2\eps + O(\eps^2)$ contamination. This suggests a possible positive direction, that although sample complexities cannot be compared across models for the same $\eps$, they can be compared by constant-factor scalings of $\eps$. That is, for any $A, B \in \{\TV, \Hub, \Sub\}$, there may be universal constants $c, C >0$ such that
\begin{align*}
\ns_A(c\eps) \lesssim \ns_B(\eps) \lesssim \ns_A(C\eps).
\end{align*}
\end{itemize}

Observation (i) suggests that there is unlikely to be a simple formula for the sample complexity of robust hypothesis testing in terms of, for example, a divergence between $p$ and $q$. Observation (ii) raises the question of whether $\eps^{3/2}$ and $\eps^2$ are true thresholds for the TV/Huber and subtractive settings, respectively; that is, whether perturbations much smaller than these scales may only cause constant-factor changes in the sample complexity. Observation (iii) suggests that comparing contamination models at the same $\eps$ may not be the right approach as the sample complexity may span the entire range from $1/\eps$ to $1/\eps^2$ stated in Proposition~\ref{prop:baseline}.
  
\subsection{Stability under $O(\eps^2)$ perturbations}\label{subsection: no_jumps}
We shall partially address the question raised by observation (ii) here. Specifically, we establish Proposition~\ref{prop: no_jumps} that shows that perturbations $\sim \eps^2$ or smaller may only cause constant factor changes in the sample complexity in all three models. In particular, this shows that the subtractive adversary has a threshold precisely at $\sim \eps^2$: perturbations of order $\eps^{2-\eta}$ can cause polynomial jumps in the sample complexity, whereas perturbations of the order $\eps^{2+\eta}$ can only cause constant factor changes. 

\begin{proposition}[No jumps for $O(\eps^2)$ perturbations]\label{prop: no_jumps}
Let $p$, $q$ be arbitrary probability distributions on a discrete space and let $\eps_0 \le \dtv(p,q)/4$. There exists a constant $c>0$ such that for all $\eps$ satisfying $0 \le \eps - \eps_0 \le c\eps_0^2$, the sample complexities at $\eps$ and $\eps_0$ are comparable; i.e., $\nsTV(\eps_0) \asymp \nsTV(\eps)$, $\nsHub(\eps_0) \asymp \nsHub(\eps)$, and $\nsSub(\eps_0) \asymp \nsSub(\eps)$.
\end{proposition}
\begin{proof}[Proof of Proposition~\ref{prop: no_jumps}]
Suppose $\eps = \eps_0 + c\eps_0^2$ for some $c$ to be determined later. Consider the TV-contamination case for now. Let the Hellinger minimising pairs at $\eps$ and $\eps_0$-contamination be $(p_\eps^*, q_\eps^*)$ and $(p_{\eps_0}^*, q_{\eps_0}^*)$, respectively. By monotonicity,
\begin{align*}
\dhel(p_{\eps_0}^*, q_{\eps_0}^*) \geq \dhel(p_\eps^*, q_\eps^*).
\end{align*}
We now prove a matching lower bound for $\dhel(p_\eps^*, q_\eps^*)$. We need the following technical lemma, proved in Appendix~\ref{app: tv_ball_at_p_eps}:
\begin{lemma}\label{lemma: tv_ball_at_p_eps}
Let $p$ be a probability distribution and let $0 < \eps_0 < \eps$. Suppose $p_\eps \in \cP_\eps$, where $\cP_\eps$ is the $\eps$-uncertainty set in either the TV, Huber, or subtractive contamination settings. Then there exists a $p_{\eps_0} \in \cP_{\eps_0}$ such that $\dtv(p_\eps, p_{\eps_0}) \leq \eps-\eps_0.$
\end{lemma}
By Lemma~\ref{lemma: tv_ball_at_p_eps}, for any $p_\eps, q_\eps$ in the $\eps$-TV uncertainty sets, there exist $p_{\eps_0}$ and $q_{\eps_0}$ in the $\eps_0$-TV uncertainty sets such that
\begin{align*}
\dtv(p_\eps, p_{\eps_0}) \leq c\eps_0^2, \quad \text{ and } \quad \dtv(q_\eps, q_{\eps_0}) \leq c\eps_0^2. 
\end{align*}
From Fact~\ref{fact: hel_tv}, we have the relation $\dhel \leq \sqrt{2\dtv}$, and applying the triangle inequality for the Hellinger distance, we have
\begin{align*}
\dhel(p_\eps, q_\eps) &\geq \dhel(p_{\eps_0}, q_{\eps_0}) - \dhel(p_{\eps_0}, p_\eps) - \dhel(q_{\eps_0}, q_\eps)\\
&\geq \dhel(p_{\eps_0}, q_{\eps_0}) - 2\eps_0\sqrt{2c}\\
&\geq \dhel(p_{\eps_0}^*, q_{\eps_0}^*) - 2\eps_0\sqrt{2c} \tag{$(p_{\eps_0}^*, q_{\eps_0}^*)$ is the Hellinger-minimising pair}
\end{align*}
However, observe also that
\begin{align*}
\dhel(p_{\eps_0}^*, q_{\eps_0}^*) &\geq \dtv(p_{\eps_0}^*, q_{\eps_0}^*) \tag{$\dhel \geq \dtv$}\\
&\geq \dtv(p,q) - 2\eps_0 \tag{triangle inequality}\\
&\geq 2\eps_0. \tag{$\eps_0 \leq \dtv(p,q)/4$}
\end{align*}
Choosing $c$ to be $1/8$, for instance, we can thus ensure
\begin{align*}
\dhel(p_\eps, q_\eps) &\geq \frac{\dhel(p_{\eps_0}^*, q_{\eps_0}^*)}{2}.
\end{align*}
As this is true for any choice of $p_\eps, q_\eps$, it also true for the Hellinger minimising pair. Hence,
\begin{align*}
\dhel(p_\eps^*, q_\eps^*) &\geq \frac{\dhel(p_{\eps_0}^*, q_{\eps_0}^*)}{2}.
\end{align*}
Combining the two bounds, we have shown
\begin{align*}
\dhel(p_\eps^*, q_\eps^*) \asymp \dhel(p_{\eps_0}^*, q_{\eps_0}^*),
\end{align*}
which immediately gives the desired sample complexity result. The proofs for the Huber and subtractive settings are identical once we replace $p_{\eps_0}$ and $q_{\eps_0}$ by the choices given in Lemma~\ref{lemma: tv_ball_at_p_eps}.
\end{proof}

\section{Sample complexity under model misspecification}\label{sec: mismatch}

In this section we consider the impact on the sample complexity when contamination is thought to be $\eps$, but the true value is different. We show that overestimating $\eps$ can lead to a significant increase in sample complexity, even when the overestimation is only up to a small $o(\eps)$ error. On the other hand, underestimating $\eps$ can be catastrophic: even a small $o(\eps)$ underestimation may cause the clipped likelihood ratio test calibrated to $\eps$ to completely break down. 

\subsection{Overestimating $\eps$}

We first discuss the setting where the statistician only knows that the contamination parameter
lies in an interval $[\eps_1,\eps_2]$, where $\eps_1<\eps_2$. For all three contamination models
considered in this paper, there is no difference between knowing $\eps=\eps_2$ and knowing only
that $\eps\in[\eps_1,\eps_2]$. Indeed, for each $A\in\{\Hub,\TV,\Sub\}$, the uncertainty sets are
monotone in $\eps$, and hence
\begin{align*}
\bigcup_{\eps\in[\eps_1,\eps_2]} \cP_A(p,\eps)
=
\cP_A(p,\eps_2).
\end{align*}
The same identity holds for the uncertainty set around $q$. Therefore, the minimax robust test
continues to be the clipped likelihood ratio test calibrated to $\eps_2$, and the relevant sample
complexity is $\ns_A(\eps_2)$.

However, if the true contamination level is actually $\eps_1$, this procedure may be highly
wasteful. The results of Section~\ref{subsection: family} show that the robust sample complexity
can be highly unstable in the contamination parameter. In particular, for each
$A\in\{\Hub,\TV,\Sub\}$, there exist distributions $p,q$ and contamination levels
$\eps_2=\eps_1+o(\eps_1)$ such that
\begin{align*}
\ns_A(\eps_1)
\ll
\ns_A(\eps_2).
\end{align*}
Thus, even an asymptotically negligible overestimate of the contamination level can force the
statistician to take polynomially more samples than would have been necessary if the true value
$\eps_1$ were known exactly. Equivalently, in the fixed-sample setting, the price of not knowing
$\eps$ exactly can be as large as the jumps exhibited in Theorem~\ref{thm:poly-jumps}.

\subsection{Underestimating $\eps$}
We now consider the opposite form of misspecification, where the statistician calibrates the
test to a contamination level $\eps_1$, but the true contamination level is some larger
$\eps_2>\eps_1$. In this case, the issue is not merely that the resulting test may be suboptimal:
it may fail completely. The
following theorem shows that this breakdown can occur in all three contamination models, even
when the underestimation error is only $o(\eps_2)$.

\begin{theorem}[Breakdown under underestimated contamination]\label{thm:underestimate_breakdown}
Consider the distributions on $\{1,2,3\}$ given by  $p = \left( \frac12 - 10\eps, \frac12 + 8\eps, 2\eps\right)$ and $q = \left( \frac12, \frac12, 0\right)$. For each model $A\in\{\TV,\Hub,\Sub\}$, there exist contamination levels
$\eps_{1}<\eps_{2}$ with $\eps_{1}=\eps_{2}-o(\eps_{2})$ such that the
clipped likelihood ratio test calibrated to $\eps_{1}$ breaks down under
$\eps_{2}$-contamination. More precisely, the following hold.
\begin{enumerate}[label=(\roman*)]
\item For TV contamination, fix $0<t<1/3$ and set
$\eps_{2} = \eps$ and $\eps_{1} = \eps-\eps^{1+t}.$

\item For Huber contamination, fix $0<t<1/2$ and set $\eps_{2} = \frac{2\eps}{1+2\eps}$ and $\eps_{1} = \eps_{2}-\eps^{1+t}.$

\item For subtractive contamination, fix $0<t<1$ and set $\eps_{2} = \frac{2\eps}{1-2\eps}$ and $\eps_{1} = \eps_{2}-\eps^{1+t}.$
\end{enumerate}

Let $\varphi_{A,\eps_{1},n}$ denote the $n$-sample likelihood ratio test between the
$A$-LFDs at contamination level $\eps_{1}$, with decision $0$ corresponding to $p$ and
decision $1$ corresponding to $q$. Then for all sufficiently small $\eps$, there exist
distributions $P_{2} \in \cP_A(p,\eps_{2})$ and $Q_{2} \in \cP_A(q,\eps_{2})$ such that
\begin{align*}
\liminf_{n\to\infty}\max\!\Bigl\{\prob_{P_2^{\otimes n}}(\varphi_{A,\eps_1,n}=1),\quad
\prob_{Q_2^{\otimes n}}(\varphi_{A,\eps_1,n}=0).\Bigr\}
\;=\;1.
\end{align*}
Thus, underestimating the contamination level can make the calibrated likelihood ratio test
fail by having at least one of the two errors tend to 1.
\end{theorem}

\begin{proof}
We prove the theorem separately for the three contamination models.

\paragraph{TV contamination.}
Let $\delta = \eps^{1+t}, \eps_{2} = \eps, \eps_{1} = \eps-\delta,$ and $\eta = \frac{\delta}{\eps}.$ The TV-LFDs at contamination level $\eps_{1}$ are $\ptv(\eps_{1}) = \left(\frac12 - 9\eps-\delta, \frac12+8\eps, \eps+\delta \right)$ and $\qtv(\eps_1) = \left(\frac12-\eps+\delta, \frac12, \eps-\delta  \right)$.
At contamination level $\eps_{2}=\eps$, the TV-LFDs are $\ptv(\eps_{2}) = \left(\frac12 - 9\eps, \frac12+8\eps, \eps \right)$ and $\qtv(\eps_2) = \left(\frac12-\eps, \frac12, \eps  \right)$. Set
$P_{2} = \ptv(\eps_{2})$ and $Q_{2} = \qtv(\eps_{2}).$
Define the $\eps_{1}$-calibrated log-likelihood statistic
\begin{align*}
Z_{\TV,\eps_{1}}(x)
=
\log \frac{\ptv(\eps_{1})(x)}{\qtv(\eps_{1})(x)}.
\end{align*}
Then under $Q_{2}$,
\begin{align*}
\E_{Q_{2}} Z_{\TV,\eps_{1}}(X)
&=
\left(\frac12-\eps\right)
\log
\frac{1-18\eps-2\delta}{1-2\eps+2\delta}
+
\frac12\log(1+16\eps)
+
\eps \log \frac{1+\eta}{1-\eta}.
\end{align*}
Using Taylor expansion around $(\eps,\eta)=(0,0)$ gives
\begin{align*}
\E_{Q_{2}} Z_{\TV,\eps_{1}}(X)
&=
\frac{2}{3}\eps\eta^3
-
128\eps^2
+
O(\eps\eta^5+\eps^2\eta+\eps^3).
\end{align*}
Since $0<t<1/3$, we have
\begin{align*}
\eps\eta^3
=
\eps^{1+3t}
\gg
\eps^2,
\qquad
\eps\eta^5+\eps^2\eta+\eps^3
=
o(\eps\eta^3).
\end{align*}
Therefore, for all sufficiently small $\eps$,
\begin{align*}
\E_{Q_{2}} Z_{\TV,\eps_{1}}(X)
>
0.
\end{align*}
Thus, if $S_n = \sum_{i=1}^n Z_{\TV,\eps_{1}}(X_i),$
then by the law of large numbers,
\begin{align*}
\prob_{X^n\sim Q_{2}^{\otimes n}}
\left(
S_n\geq 0
\right)
\to
1.
\end{align*}
Since the $\eps_{1}$-calibrated likelihood ratio test decides in favour of $p$ when
$S_n\geq 0$, we get
\begin{align*}
\prob_{X^n\sim Q_{2}^{\otimes n}}
\left(
\varphi_{\TV,\eps_{1},n}(X^n)=0
\right)
\to
1.
\end{align*}
This proves the desired breakdown for TV contamination.

\paragraph{Huber contamination.}
Let $\delta = \eps^{1+t}, \eps_{2} = \frac{2\eps}{1+2\eps},$ and
$\eps_{1} = \eps_2-\delta.$ The Huber-LFDs at contamination level $\eps_1$ are
\begin{align*}
\phub(\eps_1)
&=
\left(
(1-\eps_1)\left(\frac12-10\eps\right)+\eps_1,
(1-\eps_1)\left(\frac12+8\eps\right),
2\eps(1-\eps_1)
\right),\\
\qhub(\eps_1)
&=
\left(
\frac{1-\eps_1}{2},
\frac{1-\eps_1}{2},
\eps_1
\right).
\end{align*}
At contamination level $\eps_2$, the Huber-LFDs are given by the same formulas with
$\eps_1$ replaced by $\eps_2$. Since $\eps_2=2\eps/(1+2\eps)$, we have
\begin{align*}
2\eps(1-\eps_2)=\eps_2,
\end{align*}
so symbol $\{3\}$ is neutralised at level $\eps_2$. Set
$P_2=\phub(\eps_2)$ and $Q_2=\qhub(\eps_2).$
Define the $\eps_1$-calibrated log-likelihood statistic
\begin{align*}
Z_{\Hub,\eps_1}(x)
=
\log \frac{\phub(\eps_1)(x)}{\qhub(\eps_1)(x)}.
\end{align*}
Then under $Q_2$,
\begin{align*}
\E_{Q_2} Z_{\Hub,\eps_1}(X)
&=
\frac{1-\eps_2}{2}
\left[
\log
\frac{2\left((1-\eps_1)(\frac12-10\eps)+\eps_1\right)}
{1-\eps_1}
+
\log(1+16\eps)
\right]\\
&\qquad
+
\eps_2
\log
\frac{2\eps(1-\eps_1)}{\eps_1}.
\end{align*}
Using Taylor expansion around $(\eps,\delta)=(0,0)$ gives
\begin{align*}
\E_{Q_2} Z_{\Hub,\eps_1}(X)
&=
\frac{\delta^2}{4\eps}
-
128\eps^2
+
O\left(\eps\delta+\frac{\delta^3}{\eps^2}+\eps^3\right).
\end{align*}
Since $0<t<1/2$, we have
\begin{align*}
\frac{\delta^2}{\eps}
=
\eps^{1+2t}
\gg
\eps^2,
\qquad
\eps\delta+\frac{\delta^3}{\eps^2}+\eps^3
=
o\left(\frac{\delta^2}{\eps}\right).
\end{align*}
Therefore, for all sufficiently small $\eps$,
\begin{align*}
\E_{Q_2} Z_{\Hub,\eps_1}(X)
>
0.
\end{align*}
Thus, if $S_n = \sum_{i=1}^n Z_{\Hub,\eps_1}(X_i),$
then by the law of large numbers,
\begin{align*}
\prob_{X^n\sim Q_2^{\otimes n}}
\left(
S_n\geq 0
\right)
\to
1.
\end{align*}
Since the $\eps_1$-calibrated likelihood ratio test decides in favour of $p$ when
$S_n\geq 0$, we get
\begin{align*}
\prob_{X^n\sim Q_2^{\otimes n}}
\left(
\varphi_{\Hub,\eps_1,n}(X^n)=0
\right)
\to
1.
\end{align*}
This proves the desired breakdown for Huber contamination.

\paragraph{Subtractive contamination.}
Let $\delta = \eps^{1+t}, \eps_{2} = \frac{2\eps}{1-2\eps},$ and
$\eps_{1} = \eps_2-\delta.$ The subtractive LFDs at contamination level $\eps_1$ are
\begin{align*}
\psub(\eps_1)
&=
\left(
(1+\eps_1)\left(\frac12-10\eps\right),
(1+\eps_1)\left(\frac12+8\eps\right),
2\eps(1+\eps_1)-\eps_1
\right),\\
\qsub(\eps_1)
&=
\left(
\frac{1-\eps_1}{2},
\frac{1+\eps_1}{2},
0
\right).
\end{align*}
At contamination level $\eps_2$, the subtractive LFDs are given by the same formulas with
$\eps_1$ replaced by $\eps_2$. Since $\eps_2=2\eps/(1-2\eps)$, we have
\begin{align*}
2\eps(1+\eps_2)-\eps_2=0,
\end{align*}
so symbol $\{3\}$ is deleted at level $\eps_2$. Set
$P_2=\psub(\eps_2)$ and $Q_2=\qsub(\eps_2).$
Define the $\eps_1$-calibrated log-likelihood statistic
\begin{align*}
Z_{\Sub,\eps_1}(x)
=
\log \frac{\psub(\eps_1)(x)}{\qsub(\eps_1)(x)}.
\end{align*}
Although $Z_{\Sub,\eps_1}(3)=+\infty$, the point $\{3\}$ has zero mass under both $P_2$ and
$Q_2$. Under $P_2$,
\begin{align*}
\E_{P_2} Z_{\Sub,\eps_1}(X)
&=
\frac{\frac12-10\eps}{1-2\eps}
\log
\frac{(1+\eps_1)(1-20\eps)}{1-\eps_1}
+
\frac{\frac12+8\eps}{1-2\eps}
\log(1+16\eps).
\end{align*}
Using Taylor expansion around $(\eps,\delta)=(0,0)$ gives
\begin{align*}
\E_{P_2} Z_{\Sub,\eps_1}(X)
&=
-\delta
+
128\eps^2
+
O(\eps\delta+\delta^2+\eps^3).
\end{align*}
Since $0<t<1$, we have
\begin{align*}
\delta
=
\eps^{1+t}
\gg
\eps^2,
\qquad
\eps\delta+\delta^2+\eps^3
=
o(\delta).
\end{align*}
Therefore, for all sufficiently small $\eps$,
\begin{align*}
\E_{P_2} Z_{\Sub,\eps_1}(X)
<
0.
\end{align*}
Thus, if $S_n = \sum_{i=1}^n Z_{\Sub,\eps_1}(X_i),$
then by the law of large numbers,
\begin{align*}
\prob_{X^n\sim P_2^{\otimes n}}
\left(
S_n<0
\right)
\to
1.
\end{align*}
Since the $\eps_1$-calibrated likelihood ratio test decides in favour of $q$ when
$S_n<0$, we get
\begin{align*}
\prob_{X^n\sim P_2^{\otimes n}}
\left(
\varphi_{\Sub,\eps_1,n}(X^n)=1
\right)
\to
1.
\end{align*}
This proves the desired breakdown for subtractive contamination.

\end{proof}
\section{Sandwich bounds for sample complexity}\label{sec: sandwich}

In this section, we consider \Cref{ques:2} which asks for relationships between different adversaries. For two different contamination models $A$ and $B$, our goal is to show sandwich bounds of the form 
\begin{align*}
\ns_A(c\eps) \lesssim \ns_B(\eps) \lesssim \ns_A(C\eps),
\end{align*}
where $c$ and $C$ are universal constants; i.e., they do not depend on $p$, $q$, or $\eps$.

A natural approach for proving sandwich bounds would be to use a \emph{simulation-based} strategy. To be precise, we could prove sandwich bounds by showing that the $c\eps$- and $C\eps$-uncertainty sets for $A$ sandwich the $\eps$-uncertainty set for $B$. In our case, since LFDs exist in all cases, we may also try to show something narrower: that the LFDs for $\eps$-contamination with $B$ lie in the $C\eps$-uncertainty set for $A$ (i.e.\ $A$ can simulate $B$'s LFDs), and the LFDs for $c\eps$-contamination with $A$ lie within the $\eps$-uncertainty set with $B$ (i.e.\ $B$ can simulate $A$'s LFDs). 

Before embarking on our proofs, we note that simulation-based strategies are destined to fail for proving our desired sandwich bounds. In Appendix~\ref{app: lemma: containment}, we show that apart from the simple containment that $\eps$-Huber and $\eps$-subtractive uncertainty sets are contained in the $\eps$-TV uncertainty set, no other containment result holds in general. To be precise, for each pair of models, there exist distributions for which no containment of the form $\cP^A(c\eps)\subseteq\cP^B(\eps)$ or $\cP^B(\eps)\subseteq\cP^A(C\eps)$ holds, for any choice of constants $c,C>0$. The narrower approach of showing LFD-simulation also fails in a similar manner, no matter what constants are chosen. The proof proceeds by constructing explicit examples on a binary alphabet. 

\subsection{Sandwich bounds: Huber and TV}

We prove the following theorem:

\begin{theorem}[Comparing $\nsTV$ and $\nsHub$]\label{theorem: tv_hub}
Let $p$ and $q$ be probability distributions over a finite discrete space $\cX$ and let $\eps \leq \dtv(p,q)/4$. Then
\begin{align*}
\nsTV(\eps/2) \lesssim \nsHub(\eps) \lesssim \nsTV(\eps).
\end{align*}
The constants implicit in $\lesssim$ may depend on $\delta_0$, but are independent of
$p$, $q$, and $\eps$. 
\end{theorem}

\begin{proof}[Proof of Theorem~\ref{theorem: tv_hub}]	

It is enough to prove that
\[
\nsTV(\eps/2) \lesssim \nsHub(\eps).
\]
Recall that the sample complexity is determined by the Hellinger divergence between the corresponding LFD pairs:
\begin{align*}
\nsTV(\eps/2) &\asymp \frac{1}{\dhel^2(\ptv(\eps/2), \qtv(\eps/2))}, \\
\nsHub(\eps) &\asymp \frac{1}{\dhel^2(\phub(\eps), \qhub(\eps))}.
\end{align*}
Thus, it suffices to show that
\[
\dhel^2(\ptv(\eps/2), \qtv(\eps/2)) 
\gtrsim 
\dhel^2(\phub(\eps), \qhub(\eps)).
\]

Given the formulas for the LFDs, a natural approach is to explicitly evaluate the LFD pairs, compute the corresponding Hellinger divergences, and compare them. We follow this strategy, however, the final comparison step turns out to be nontrivial because the LFDs do not have analytical expressions.

To address this, we proceed in three steps:

\paragraph {Step I:} We first show that if the lower and upper clips $\cl$ and $\cu$ are identical for Huber and TV contamination, the Hellinger divergences between the TV-LFDs and the Huber-LFDs are within constant factors of each other. Observe that formulas for LFDs given by equations~\eqref{eq: lfd_hub} and \eqref{eq: lfd_tv} continue to make sense for any choice of thresholds $\cl < 1 < \cu$, but the resulting measures may not  integrate to 1 if the clips are not calibrated to $\eps$ according to equations~\eqref{eq: clips_hub} and~\eqref{eq: clips_tv}. For two non-necessarily probability measures, we continue to use the notation $\dhel^2(p,q) := \norm{\sqrt p - \sqrt q}^2$.

\begin{lemma}[Hellinger comparability]\label{lemma: hellinger-comparability}
Fix thresholds $\cl<1<\cu$ and set $(\ptv,\qtv)$ using equation~\eqref{eq: lfd_tv} and $(\phub,\qhub)$ using equation~\eqref{eq: lfd_hub}. Note that these may not be probability distributions as they may not integrate to 1.
If $\eps \le \frac{1}{2}$, then
\begin{align*}
\frac{1}{2}\dhel^2(\ptv,\qtv) \le \dhel^2(\phub,\qhub) \le 2\dhel^2(\ptv,\qtv).
\end{align*}
\end{lemma}
The proof is deferred to Appendix~\ref{app: hellinger-comparability}. 

\paragraph{Step II:} 
Next, we compare the TV-clips at $\eps/2$ with the Huber-clips at $\eps$. Specifically, we show that the upper TV-clip is larger than the upper Huber-clip, and the lower TV-clip is smaller than the lower Huber-clip.

\begin{lemma}[Ordering of TV and Huber clips]\label{lemma: threshold-order}
Fix $\eps\in [0,1/2]$ and let $\clHub(\eps),\,\cuHub(\eps)$ be the Huber thresholds for radius $\eps$,
and $\clTV(\eps/2),\,\cuTV(\eps/2)$ the TV thresholds for radius $\eps/2$.
Then
\[
\clTV(\eps/2)\ \le\ \clHub(\eps),
\qquad
\cuHub(\eps)\ \le\ \cuTV(\eps/2).
\]
\end{lemma}
The proof is deferred to Appendix~\ref{app: threshold-order}. 

\paragraph{Step III:} 
Our last step is showing that the Hellinger divergence under TV contamination is monotone in the clips. Intuitively, increasing the upper clip (and decreasing the lower clip) makes the LFDs closer to the uncontaminated $p$ and $q$, thereby increasing the Hellinger divergence.

\begin{lemma}[Monotonicity of TV Hellinger in the clips]\label{lemma: tv-hellinger-monotone}
Denote the Hellinger divergence between the TV-LFDs $(\ptv, \qtv)$ (not necessarily probability distributions) at clips $(\cl, \cu)$ 
\begin{align*}
\Hsq_{\TV}(\cl,\cu):=\dhel^2(\ptv,\qtv).
\end{align*}
Then the function $\Hsq$ is (i) nondecreasing in the upper clip $\cu$, and (ii) nonincreasing in the lower clip $\cl$. Equivalently: pushing the upper clip \emph{up} and/or the lower clip \emph{down} can only increase $\Hsq_{\TV}$.
\end{lemma}
The proof is deferred to Appendix~\ref{app: tv-hellinger-monotone}.

Combining these steps yields the result. Step (i) shows that the Hellinger divergences in the TV and Huber settings are comparable when evaluated at the same clips (namely, $\cl_{\Hub}(\eps)$ and $\cu_{\Hub}(\eps)$). Step (ii) shows that the TV-clips at $\eps/2$ move further away from 1 than these reference clips. Finally, step (iii) implies that the Hellinger divergence for TV at level $\eps/2$ is larger than the Hellinger divergence for Huber at level $\eps$. We explain this in more detail below.

Let $(\phub,\qhub)$ be the Huber LFDs at $(\clHub(\eps),\cuHub(\eps))$. Note that $(\phub,\qhub)$ are probability measures and we have the sample complexity relation $$\nsHub(\eps) \asymp \frac{1}{\dhel^2(\phub,\qhub)}.$$
Let $(\ptvpr,\qtvpr)$ denote the TV-LFDs evaluated at the \emph{same} thresholds $(\clHub(\eps),\cuHub(\eps))$. Note that these need not be probability distributions. Let $(\ptv(\eps/2),\qtv(\eps/2))$ be the TV-LFDs at the TV thresholds $(\clTV(\eps/2),\cuTV(\eps/2))$. Note that $(\ptv(\eps/2),\qtv(\eps/2))$ are probability measures and we have the sample complexity relation 
$$\nsTV(\eps/2) \asymp \frac{1}{\dhel^2(\ptv(\eps/2),\qtv(\eps/2))}.$$
By Lemma~\ref{lemma: threshold-order} and Lemma~\ref{lemma: tv-hellinger-monotone},
\begin{align*}
\dhel^2(\ptvpr,\qtvpr)\ \le\ \dhel^2\!\big(\ptv(\eps/2),\qtv(\eps/2)\big).
\end{align*}
By Lemma~\ref{lemma: hellinger-comparability} at the common thresholds $(\clHub(\eps),\cuHub(\eps))$,
\begin{align*}
\frac{1}{2}\,\dhel^2(\ptvpr,\qtvpr)\ \le\ \dhel^2(\phub,\qhub)\ \le\ 2\,\dhel^2(\ptvpr,\qtvpr).
\end{align*}
Combining, 
\begin{align*}
\tfrac{1}{2}\,\dhel^2(\phub,\qhub)\ \le\ \dhel^2(\ptv(\eps/2),\qtv(\eps/2)).
\end{align*}
Taking reciprocals yields
\begin{align*}
\nsHub(\eps) \gtrsim \nsTV(\eps/2).
\end{align*}

\end{proof}

\subsection{Sandwich bounds: Subtractive and TV}

Our main result in this section is the following:

\begin{theorem}[Comparing $\nsTV$ and $\nsSub$]\label{theorem: sub_tv}
Let $p$ and $q$ be distributions over a finite discrete space $\cX$ and let $\eps \leq \dtv(p,q)/4$. Let $\delta_0 > 0$. Then for all $\eps$ smaller than a constant that depends only on $\delta_0$ (and not on $p$ or $q$),
\begin{align*}
\nsSub(\eps) \lesssim \nsTV(\eps) \lesssim \nsSub((2+\delta_0)\eps).
\end{align*}
The constants implicit in $\lesssim$ may depend on $\delta_0$, but are independent of
$p$, $q$, and $\eps$. 
\end{theorem}

\begin{remark}\label{rem: delta_0} The lower bound on $\nsTV$ is immediate as TV-contamination is stronger than subtractive contamination. The upper bound expression is very similar to that in Theorem~\ref{theorem: tv_hub}. Specifically, if we ignore the role of $\delta_0$ (which may be fixed to any small constant), we have $\nsTV(\eps) \lesssim \nsSub(2\eps)$. We cannot prove the bound with the constant exactly 2 like in the Huber case. This can be seen by revisiting our example in Section~\ref{sec: example}. Suppose $0 < t < 1/2$, $\eps_1 = \eps - \eps^{1+t}$ and $\eps_2 = 2\eps/(1-2\eps) - \eps^{1+t}$. The calculations from Section~\ref{sec: example} show that
\begin{align*}
\nsTV(\eps_1) \asymp \frac{1}{\eps^{1+2t}} \gg \frac{1}{\eps^{1+t}} \asymp \nsSub(\eps_2).
\end{align*}
But observe that $\eps_2 > 2\eps_1$ for all small enough $\eps$, as
\begin{align*}
\eps_2 = 2\eps - \eps^{1+t} + O(\eps^2) >  2\eps - 2\eps^{1+t} = 2\eps_1.
\end{align*}
By monotonicity of $\nsSub$, we conclude
\begin{align*}
\nsTV(\eps_1) \gg \nsSub(2\eps_1).
\end{align*}
Thus, we cannot hope to prove Theorem~\ref{theorem: sub_tv} with a constant exactly 2 like in the Huber case, and the current statement is essentially optimal.
\end{remark}

\begin{remark}\label{rem: csub}
Unfortunately, the proof strategy used in comparing Huber and TV contamination in Theorem~\ref{theorem: tv_hub} fails when comparing TV and subtractive contamination. A one-sided version of Lemma~\ref{lemma: hellinger-comparability} continues to hold, that is, we can show that for the same thresholds $(\cl, \cu)$, 
$$\dhel^2(\ptv, \qtv) \gtrsim \dhel^2(\psub, \qsub).$$
However, the proof breaks down when trying to prove the analogue of Lemma~\ref{lemma: threshold-order} that shows that $\eps/2$-TV-clips are further outward from 1 compared to the $\eps$-Huber clips. In the subtractive setting it is easy to construct examples where this cannot hold even if we replace $\eps/2$ by $\eps/M$ for any large $M$. For example, consider an extreme case when there is a point $i$ where $p(i) = \eps$ and $q(i) = \delta \ll \eps$, with all other points having likelihood ratios bounded by, say, 2. For all small enough $\eps$, the upper clip only acts on the singleton set $\{i\}$.  Using the clip formulas~\eqref{eq: clips_sub} and ~\eqref{eq: clips_tv}, we can check that 
\begin{align*}
\cu_\Sub(\eps) &\asymp \frac{\eps}{\delta}, \quad \text{ whereas }\\
\cu_\TV (\eps/M) &\asymp M. 
\end{align*}
Thus, no matter how large an $M$ is picked, we may choose $\delta$ very small, say $\eps/M^2$ and ensure that $\cu_\Sub(\eps) > \cu_\TV(\eps/M)$. This is in contrast to the Huber case, where we could establish $\cuHub(\eps)\ \le\ \cuTV(\eps/2)$.

\end{remark}

\begin{proof}[Proof of Theorem~\ref{theorem: sub_tv}]
To make our exposition more readable, we shall prove the weaker result that for all $\eps \leq 1/4$
\begin{align*}
\nsTV(\eps) \lesssim \nsSub(12\eps). 
\end{align*}
The modifications needed to the proof to drive down the constant 12 to any $(2+\delta_0)$ are minor, and we shall point them out at the end of the proof. 

Remark~\ref{rem: csub} suggests that breakdown in establishing Theorem~\ref{theorem: sub_tv} along the lines of Theorem~\ref{theorem: tv_hub} occurs when the subtractive-clips tend to the extremes of 0 and $+\infty$. Thus we break our proof into two cases: When the clips are extreme, and when they are not. When the clips aren't extreme, a proof strategy similar to that of Theorem~\ref{theorem: tv_hub} works. When the clips are extreme, it turns out that both TV and subtractive LFDs alter probabilities only points with extreme likelihood ratios, and their resulting impact on the Hellinger divergence turn out to be similar.

To formalise this proof strategy, we first state a formula that decomposes the Hellinger divergence into contributions from points with extreme-likelihood ratios and moderate likelihood ratios. Such a decomposition is not new, appearing in~\cite{PenEtal23} as well.

\begin{lemma}[Formula for Hellinger divergence up to constant factors]\label{prop: approx_hell}
 Let $p$ and $q$ be probability distributions over a discrete space $\cX$. Define the following sets:
    \begin{align*}
        A_1 \coloneqq \{i : p(i)/q(i) \in [1,2)\}, \quad &\text{ and } \quad A_2 = \{i : p(i)/q(i) \in [2, \infty)\},\\
        B_1 \coloneqq \{i : p(i)/q(i) \in [1/2, 1)\}, \quad &\text{ and } \quad B_2 \coloneqq \{i : p(i)/q(i) \in [0, 1/2) \},\\
        A \coloneqq A_1 \cup A_2, \quad &\text{ and } \quad B \coloneqq B_1 \cup B_2.
    \end{align*}
    Then
    \begin{align}
        h^2_A \coloneqq \sum_{i \in A} \left(\sqrt{p(i)} - \sqrt{q(i)}\right)^2 &\asymp \tilde h^2_A \coloneqq  p(A_2) + \sum_{i \in A_1} \frac{(p(i) - q(i))^2}{p(i)} , \quad \text{and} \label{eq: approx_hell_A}\\
        h^2_B \coloneqq \sum_{i \in B} \left(\sqrt{p(i)} - \sqrt{q(i)}\right)^2 &\asymp \tilde h^2_B \coloneqq q(B_2) + \sum_{i \in B_1} \frac{(p(i) - q(i))^2}{q(i)}, \label{eq: approx_hell_B}
    \end{align}
    and consequently,
    \begin{align}
        \dhel^2(p,q) = h^2_A + h^2_B \asymp \tilde h^2_A + \tilde h^2_B. \label{eq: approx_hell}
    \end{align}
\end{lemma}
The proof is deferred to Appendix~\ref{app: approx_hell}.

Next, we show that for the subtractive adversary, the contribution to the Hellinger divergence from extreme- and moderate-likelihood ratio terms are monotonic in the contamination parameter $\eps$. %

\begin{lemma}\label{prop: monotonic_hell}
    Let $p$ and $q$ be distributions over $\cX$ and let $\eps \ge 0$. Let the sets $A, A_1, A_2, B, B_1,$ and $B_2$ be as in Lemma~\ref{prop: approx_hell}. Denote the LFDs for $\eps$-subtractive contamination by $(p^*_\eps, q^*_\eps)$. Let 
    \begin{align*}
        h^2_A(\eps) &\coloneqq \sum_{i \in A} \left( \sqrt{p^*_\eps(i)} - \sqrt{q^*_\eps(i)} \right)^2, \text{ and }\\
        h^2_B(\eps) &\coloneqq \sum_{i \in B} \left( \sqrt{p^*_\eps(i)} - \sqrt{q^*_\eps(i)} \right)^2.
    \end{align*}
    Then for any $\eps_1 \le \eps_2$, we have $h^2_A(\eps_1) \gtrsim  h^2_A(\eps_2)$ and $h^2_B(\eps_1) \gtrsim h^2_B(\eps_2)$. The universal constant on the right may be taken to be 2.
\end{lemma}
Note that as $\eps$ grows, the Hellinger divergence between the LFDs will go down monotonically, that is, $h^2_A(\eps_1)+h^2_B(\eps_1) \geq h^2_A(\eps_2)+h^2_B(\eps_2)$. The above result simply says that the contributions from $A$ and $B$ both decrease approximately monotonically (i.e., not just their sum). The proof is deferred to Appendix~\ref{app: monotonic_hell}.

Next we combine Lemmas~\ref{prop: approx_hell} and~\ref{prop: monotonic_hell} in the following simple corollary:
\begin{corollary}[Approximate monotonicity of $\tilde h_A(\eps)$ and $\tilde h_B(\eps)$]\label{corr: monotonic_approx_hell}
Consider the same setting as in Lemma~\ref{prop: monotonic_hell}. Then for $\eps_1 \le \eps_2$, we have
\begin{align*}
\tilde h_A(\eps_1) \gtrsim \tilde h_A(\eps_2), \quad \text{ and } \quad \tilde h_B(\eps_1) \gtrsim \tilde h_B(\eps_2).
\end{align*}
\end{corollary}
This is easily seen as 
\begin{align*}
\tilde h_A(\eps_1) \stackrel{(a)}\asymp h_A(\eps_1) \stackrel{(b)}\gtrsim  h_A(\eps_2) \stackrel{(c)}\asymp \tilde h_A(\eps_2),
\end{align*}
where $(a)$ and $(c)$ follow from Lemma~\ref{prop: approx_hell} and $(b)$ follows from Lemma~\ref{prop: monotonic_hell}.

We are now in a position to complete the proof of Theorem~\ref{theorem: sub_tv}. Let $A, A_1, A_2, B, B_1,$ and $B_2$ be as in Lemma~\ref{prop: approx_hell}. For a given $\eps$, the LFDs for TV-contamination be $\ptv(\eps)$ and $\qtv(\eps)$, and the LFDs for subtractive-contamination be $\psub(\eps)$ and $\qsub(\eps)$. We shall produce an $\eps'$ such that 
\begin{align}
\tilde h_A^2(\ptv(\eps), \qtv(\eps)) &\gtrsim \tilde h_A^2(\psub(\eps'), \qsub(\eps')), \quad \text{ and } \label{eq: hA_tv_sub}\\
\tilde h_B^2(\ptv(\eps), \qtv(\eps)) &\gtrsim \tilde h_B^2(\psub(\eps'), \qsub(\eps')). \label{eq: hB_tv_sub}
\end{align}
Adding these inequalities, using Lemma~\ref{prop: approx_hell} and taking the inverse will yield $\nsTV(\eps) \lesssim \nsSub(\eps')$. By checking  $\eps' \leq 12\eps$, the proof will be complete. We shall prove~\eqref{eq: hA_tv_sub}, as a similar proof will work for proving~\eqref{eq: hB_tv_sub}.

We consider two cases, either $\cu_\TV(\eps) \in [2, \infty)$ or $\cu_\TV(\eps) \in (1, 2)$.

\paragraph{First case when $\cu_\TV(\eps) \in [2, \infty)$:} We first argue that $\cu_\TV(\eps) \le \cu_\Sub(\eps)$. Recall that $\cu_\TV(\eps)$ is determined as per equations~\eqref{eq: clips_tv}; i.e.,
\begin{align}\label{eq: cupper_tv}
\frac{p(H) - \cu q(H)}{1+\cu} = \eps,
\end{align}
where $H$ is understood to be the set of points where $p_i/q_i \geq \cu_\TV(\eps)$.
In contrast, the value of $\cu_\Sub(\eps)$ is determined by the equations~\eqref{eq: clips_sub}; i.e.,
\begin{align}\label{eq: cupper_sup}
p(H) - \cu q(H) = \frac{\eps}{1+\eps},
\end{align}
when the solution exists, where $H$ is understood to be the set of points where $p_i/q_i \geq \cu_\Sub(\eps)$, otherwise we use the convention $\cu_\Sub(\eps) = \infty$. Observe that if we use the same $\cu_\TV$ in the formula for subtractive contamination above, the left hand side will evaluate to $(1+\cu_\TV)\eps \geq 3\eps > \eps/(1+\eps)$. Hence, the $\cu_\Sub(\eps)$ must be larger than $\cu_\TV(\eps)$ (the left hand side is a decreasing function of $\cu$, and as it needs to go down from $3\eps$ to $\eps/(1+\eps)$, we will have to increase $\cu$ beyond $\cu_\TV(\eps)$). In particular, we can say that 
\begin{align*}
\cu_\Sub(\eps) \ge \cu_\TV(\eps) \ge 2.
\end{align*}

In a sense, for both TV and subtractive contamination, ``most of the action'' (of reducing the mass of $p$) happens at high likelihood ratios. We can use this observation to calculate $\tilde h_A$ as follows:
\begin{align}
\tilde h^2_A(\ptv(\eps), \qtv(\eps)) &= \sum_{i\in A_1} \frac{(\ptv(i)-\qtv(i))^2}{\ptv(i)} + \ptv(A_2)\notag \\ 
&\stackrel{(a)}= \sum_{i\in A_1} \frac{(p_i-q_i)^2}{p_i} + p(A_2) - \eps, \label{eq: case1tv}
\end{align}
where in $(a)$ we used the formulas for the TV-LFDs and note that the $\ptv$ only differs from $p$ at some points in $A_2$ (those with likelihood ratios at least $\cu_\TV(\eps)$), and the cumulative decrease in mass of $\ptv$ on $A_2$ is exactly $\eps$. 
A similar calculation for subtractive contamination yields
 \begin{align}
\tilde h^2_A(\psub(\eps), \qsub(\eps)) &= \sum_{i\in A_1} \frac{(\psub(i)-\qsub(i))^2}{\psub(i)} + \psub(A_2)\notag \\
&\stackrel{(a)}= (1+\eps) \sum_{i\in A_1} \frac{(p_i-q_i)^2}{p_i} + (1+\eps)p(A_2) - \eps, \label{eq: case1sub}\\
& \asymp \sum_{i\in A_1} \frac{(p_i-q_i)^2}{p_i} + (p(A_2) - \eps/(1+\eps)). \label{eq: case1sub_2}
\end{align}
We justify the equality in $(a)$ as follows. First, consider the case when $\cu_\Sub(\eps) \neq \infty$. Using the formulas for the LFDs and the fact that $\cu_\Sub(\eps) \geq 2$ we see that for all points in $A_1$, $\psub(i)/p(i) = \qsub(i)/q(i) = 1+\eps$. Moreover, the only points in $A_2$ for which this relation does not hold are those whose likelihood ratios are at least $\cu_\Sub(\eps)$, and so we may write
\begin{align*}
\psub(A_2) &= (1+\eps)p(A_2) - \sum_{i: p_i/q_i \geq \cu_\Sub(\eps)} (1+\eps)p(i) - (1+\eps)\cu_\Sub(\eps)q(i) \\
&= (1+\eps) p(A_2) - \eps,
\end{align*}
where the final equality follows from the identity~\eqref{eq: cupper_sup}. When $\cu_\Sub(\eps) = \infty$, the equality in $(a)$ continues to hold because the ``$\eps$-reduction of mass'' only happens in the set $\bar H$, which lies in $A_2$. More precisely, we have the equality $\psub(\bar H) = (1+\eps)p(\bar H) - \eps$, and $\psub(A \setminus \bar H) = (1+\eps)p(A \setminus \bar H)$. Adding these two gives $(a)$. 
Now observe that the expressions \eqref{eq: case1tv} and \eqref{eq: case1sub_2} have the same first terms, and furthermore, we claim that 
\begin{align*}
p(A_2) - \eps \asymp p(A_2) - \frac{\eps}{1+\eps}. 
\end{align*}
This is because $p(A_2)$ is quite large, specifically,
\begin{align*}
p(A_2) \geq p(\{i: \frac{p_i}{q_i} \ge \cu_\TV(\eps)\}) \stackrel{(a)}\geq (1+\cu_\TV(\eps)) \eps \geq 3\eps,
\end{align*}
where $(a)$ uses identity \eqref{eq: cupper_tv}, and thus
\begin{align*}
p(A_2) - \eps \asymp p(A_2) \asymp p(A_2) - \eps/(1+\eps).
\end{align*}
This shows that for the case of $\cu_\TV(\eps) \geq 2$, we have 
\begin{align}
\tilde h^2_A(\psub(\eps), \qsub(\eps)) \asymp \tilde h^2_A(\ptv(\eps), \qtv(\eps)). \label{eq: case1_A}
\end{align}
A similar proof will show that if $\cl_\TV (\eps) \le 1/2$, then 
\begin{align}
\tilde h^2_B(\psub(\eps), \qsub(\eps)) \asymp \tilde h^2_B(\ptv(\eps), \qtv(\eps)). \label{eq: case1_B}
\end{align}

\paragraph{Second case when $\cu_\TV(\eps) \in (1, 2)$:} The identity \eqref{eq: cupper_tv} gives
\begin{align}
p(H) - \cu_\TV(\eps) q(H) = (1+\cu_\TV(\eps))\eps \le 3\eps, \tag{as $\cu < 2$}
\end{align}
Choose $\eps'$ such that 
\begin{align*}
\frac{\eps'}{1+\eps'} = (1+\cu_\TV(\eps))\eps,
\end{align*}
that is
\begin{align}\label{eq: 12eps}
\eps' = \frac{(1+\cu_\TV(\eps))\eps}{1 - (1+\cu_\TV(\eps))\eps} \le \frac{3\eps}{1-3\eps} \le 12\eps,
\end{align}
where the final expression assumes $\eps \le \dtv(p,q)/4 \le 1/4.$ By the identity~\eqref{eq: cupper_sup}, it is immediate that 
\begin{align*}
\cu_\Sub(\eps') = \cu_\TV(\eps).
\end{align*}
The important part of the proof so far is the observation that $\eps' \asymp \eps$. We now evaluate $\tilde h_A^2(\ptv(\eps), \qtv(\eps))$ and compare it to $\tilde h_A^2(\psub(\eps'), \qsub(\eps'))$, crucially aided by the fact that they both have the \emph{same} clips $\cu = \cu_\TV(\eps) = \cu_\Sub(\eps') \le 2$ and since likelihood ratios are bounded by $\cu$, there are no terms from $A_2$. Specifically,
\begin{align}
\tilde h_A^2(\ptv(\eps), \qtv(\eps)) &= \sum_{i \in A} \frac{(\ptv(i) - \qtv(i))^2}{\ptv(i)} \nonumber\\
&= \sum_{i: \frac{p_i}{q_i} \in [1, \cu)} \frac{(p_i-q_i)^2}{p_i} + \sum_{i: \frac{p_i}{q_i}  \in [\cu, \infty)} \frac{(\cu - 1)^2(p_i+q_i)}{\cu(1+\cu)} \nonumber\\
&\ge \sum_{i: \frac{p_i}{q_i} \in [1, \cu)} \frac{(p_i-q_i)^2}{p_i} + \sum_{i: \frac{p_i}{q_i}  \in [\cu, \infty)} \frac{(\cu - 1)^2q_i}{\cu} \label{eq: case2tv}
\end{align}
where we used $p_i \geq \cu q_i$ in the last inequality.

Now let's compute $\tilde h_A^2(\psub(\eps'), \qsub(\eps'))$. The same argument as above applies, and we only need to consider the first term as likelihood ratios are bounded by $\cu \le 2$. Hence,
\begin{align}
\tilde h_A^2(\psub(\eps'), \qsub(\eps')) &= \sum_{i \in A} \frac{(\psub(i) - \qsub(i))^2}{\psub(i)}\nonumber \\
&= (1+\eps') \sum_{i: \frac{p_i}{q_i} \in [1, \cu)} \frac{(p_i-q_i)^2}{p_i} + \sum_{i: \frac{p_i}{q_i}  \in [\cu, \infty)} (1+\eps')q_i \frac{(\cu-1)^2}{\cu} \nonumber \\ 
&= (1+\eps') \left[\sum_{i: \frac{p_i}{q_i} \in [1, \cu)} \frac{(p_i-q_i)^2}{p_i} + \sum_{i: \frac{p_i}{q_i}  \in [\cu, \infty)} q_i \frac{(\cu-1)^2}{\cu}\right] \label{eq: case2sub}.
\end{align}
Comparing expressions \eqref{eq: case2tv} and \eqref{eq: case2sub}, we see that
\begin{align*}
\tilde h_A^2(\psub(\eps'), \qsub(\eps')) &\lesssim \tilde h_A^2(\ptv(\eps), \qtv(\eps)).
\end{align*}
\paragraph{Putting it all together:} Finally, by the monotonicity of $\tilde h_A$, we conclude that no matter which of case 1 or case 2 holds, we must have
\begin{align}
\tilde h_A^2(\psub(12\eps), \qsub(12\eps)) \lesssim  \tilde h_A^2(\ptv(\eps), \qtv(\eps)).
\end{align}
The same result holds for $B$ as well, that is,
\begin{align}
\tilde h_B^2(\psub(12\eps), \qsub(12\eps)) \lesssim  \tilde h_B^2(\ptv(\eps), \qtv(\eps)).
\end{align}
Adding up, using Lemma~\ref{prop: approx_hell}, and taking the inverse, we finally arrive at the claimed sample complexity bound
\begin{align}
\nsSub(12\eps) \gtrsim \nsTV(\eps).
\end{align}
\emph{Improving the constant to $2+\delta_0$:} To improve the constant 12 to $2+\delta_0$ for any $\delta_0 > 0$ we need to do only two changes: Adjust Lemma~\ref{prop: approx_hell} to use the thresholds $1-\delta_0'$ and $1+\delta_0'$ for some $\delta_0' < \delta_0$ instead of $1/2$ and $2$. The same result continues to hold, with constants that now depend on $\delta_0'$. Now alter the case 1 and case 2 above depending on whether $\cu_\TV(\eps)$ lies beyond or below $1+\delta_0'$. Case 1 remains unchanged, but in case 2, the upper bound on $\eps'$ in inequality~\eqref{eq: 12eps} becomes
\begin{align*}
\eps' \leq \frac{(2+ \delta_0')\eps}{1 - (2+ \delta_0')\eps} = (2+ \delta_0')\eps + O(\eps^2) < (2+\delta_0) \eps.
\end{align*}
Thus, for all small enough $\eps$, we are able to show 
\begin{align}
\nsSub((2+\delta_0)\eps) \gtrsim \nsTV(\eps).
\end{align}
This completes the proof.
\end{proof}

\subsection{Sandwich bounds: Huber and subtractive}
Theorems~\ref{theorem: tv_hub} and~\ref{theorem: sub_tv} directly imply comparisons between $\nsHub$ and $\nsSub$. To be precise, we can deduce that
\begin{align*}
\nsSub(\eps/2) \lesssim \nsHub(\eps) \lesssim \nsSub((2+\delta_0) \eps).
\end{align*}
However, this is not entirely satisfactory because the constants $1/2$ and $2+\delta$ may not be tight. Moreover, our guiding example from Section~\ref{sec: example} suggests that the constants should be much closer to 1 when comparing Huber and subtractive contamination. The main goal of this section is to prove such a result. 

\begin{theorem}[Comparing $\nsSub$ and $\nsHub$]\label{theorem: hub_sub}
Let $p$ and $q$ be distributions over a finite discrete space $\cX$ and let $\eps \leq \dtv(p,q)/4$. Let $\delta_0 > 0$. Then for all $\eps$ smaller than a constant that depends only on $\delta_0$ (and not on $p$ or $q$),
\begin{align*}
\nsSub(\eps) \lesssim \nsHub(\eps) \lesssim \nsSub((1+\delta_0)\eps). 
\end{align*}
The constants implicit in $\lesssim$ may depend on $\delta_0$, but are independent of
$p$, $q$, and $\eps$. 
\end{theorem}

\begin{remark}
We cannot replace $1+\delta_0$ by 1 for similar reasons as outlined in Remark~\ref{rem: delta_0}. For the example in Section~\ref{sec: example}, we have that for $\eps_1 = 2\eps/(1+2\eps) - \eps^{1+t}$, the sample complexity under Huber contamination is $\nsHub(\eps_1) \asymp 1/\eps^{1+2t}$. This is much larger than the sample complexity with $\eps_2$-subtractive contamination where $\eps_2 = 2\eps/(1-2\eps) - \eps^{1+t}$, which is $\nsSub(\eps_2) \asymp 1/\eps^{1+t}$. However, $\eps_2 > \eps_1$, and hence
\begin{align*}
\nsHub(\eps_1) \gg \nsSub(\eps_1).
\end{align*}
Thus, the constant $1+\delta_0$ is essentially optimal.
\end{remark}

\begin{proof}[Proof of Theorem~\ref{theorem: hub_sub}]

The proof consists of two parts, one for the lower bound on $\nsHub(\eps)$ and one for the upper bound. Recall that proving $\nsSub(\eps) \lesssim \nsTV(\eps)$ was straightforward by noting the containment of uncertainty sets. However, the No-Simulation Lemma~\ref{lemma: containment} implies that we cannot argue $\nsSub(\eps) \lesssim \nsHub(\eps)$ in a similar manner. This makes both parts of the proof non-trivial.

\paragraph{Part I:} We first show the inequality $\nsSub(\eps) \lesssim \nsHub(\eps)$. We will argue that there exists some pair $(p'_\Hub, q'_\Hub)$ in the $\eps$-Huber uncertainty set such that the Hellinger divergence between this pair is at most the Hellinger divergence between the LFDs $(\psub, \qsub)$ for $\eps$-subtractive contamination.

Let $A = \{i \in \cX \,:\, p(i)/q(i) \ge 1\}$ and $B = \{i \in \cX \,:\, p(i)/q(i) < 1\}$. Recall that $A$ and $B$ are also the sets where $\psub(i) \ge \qsub(i)$ and  $\psub(i) < \qsub(i)$ as the likelihood ratios for LFDs are merely clipped versions of $p(i)/q(i)$ at some $\cu > 1 > \cl$. Suppose for each $i \in A$ and $j \in B$, we use the formulas~\eqref{eq:lfd_sub} to write
\begin{align*}
\psub(i) = (1+\eps)(p(i) - \eta(i)), \quad &\text{ and } \quad \qsub(i) = (1+\eps)q(i), \quad \text { and }\\
\psub(j) = (1+\eps)p(j), \quad &\text{ and } \quad \qsub(j) = (1+\eps)(q(j) - \eta(j)),
\end{align*}
for $\eta(i), \eta(j) \geq  0$ satisfying the following to ensure $\psub$ and $\qsub$ integrate to 1:
\begin{align*}
\sum_{i \in A} \eta(i) = \frac{\eps}{1+\eps}, \quad \text{ and } \quad \sum_{j \in B} \eta(j) = \frac{\eps}{1+\eps}
\end{align*}
 Observe that the contribution of $i \in A$ to $\dhel^2(\psub, \qsub)$ is
\begin{align}
\left( \sqrt{\psub(i)} - \sqrt{\qsub(i)} \right)^2 &\asymp \frac{(\psub(i) - \qsub(i))^2}{\psub(i)} \nonumber \\
&\asymp \frac{\left(p(i) - q(i) - \eta(i)\right)^2}{p(i) - \eta(i)}\nonumber \\
&\gtrsim \frac{\left(p(i) - q(i) - \eta(i)\right)^2}{p(i)}. \label{eq: psubA}
\end{align}
Similarly, the contribution of $j \in B$ to $\dhel^2(\psub, \qsub)$ is
\begin{align}
\left( \sqrt{\psub(j)} - \sqrt{\qsub(j)} \right)^2 &\asymp \frac{(\psub(j) - \qsub(j))^2}{\qsub(j)} \nonumber\\
&\asymp \frac{\left(p(j) - q(j) + \eta(j)\right)^2}{q(j) - \eta(j)}\nonumber \\
&\gtrsim \frac{\left(p(j) - q(j) + \eta(j)\right)^2}{q(j)}. \label{eq: psubB}
\end{align}
Now define the following Huber-contaminated distributions for $i \in A$ and $j \in B$:
\begin{align*}
p'_\Hub(i) = p(i)\frac{(1+ \eps)}{(1 + 2\eps)} \quad &\text{ and } \quad q'_\Hub(i) = \frac{(1+ \eps)}{(1 + 2\eps)}(q(i) + \eta(i)), \quad \text{ and }\\
p'_\Hub(j) = \frac{(1+ \eps)}{(1 + 2\eps)}(p(j) + \eta(j)) \quad &\text{ and } \quad q'_\Hub(j) = q(j)\frac{(1+ \eps)}{(1 + 2\eps)},
\end{align*}
Observe that these are valid distributions, since
\begin{align*}
\sum_{x \in \cX} p'_\Hub(x) &= \frac{(1+ \eps)}{(1 + 2\eps)} + \frac{(1+ \eps)}{(1 + 2\eps)}\sum_{j \in B} \eta_j\\
&= \frac{(1+ \eps)}{(1 + 2\eps)} + \frac{\eps}{(1+2\eps)}\\
&= 1. 
\end{align*}
A similar calculation may be done to verify $q'_\Hub$ also integrates to 1. Moreover, since $\eta(x) \geq 0$, we can immediately check that $p'_\Hub$ and $q'_\Hub$ lie in the $\frac{\eps}{1+2\eps}$-Huber uncertainty set, which in turn lies in the $\eps$-Huber uncertainty set.

Now the contribution of $i \in A$ to the Hellinger divergence to $\dhel^2(p'_\Hub, q'_\Hub)$ is
\begin{align}
\left( \sqrt{p'_\Hub(i)} - \sqrt{q'_\Hub(i)} \right)^2 &\asymp \frac{(p'_\Hub(i) - q'_\Hub(i))^2}{p'_\Hub(i)} \nonumber\\
&\asymp \frac{\left(p(i) - q(i) - \eta(i)\right)^2}{p(i)}. \label{eq: phubA}
\end{align}
Similarly, for $j \in B$,
\begin{align}
\left( \sqrt{p'_\Hub(j)} - \sqrt{q'_\Hub(j)} \right)^2 &\asymp \frac{(p'_\Hub(j) - q'_\Hub(j))^2}{q'_\Hub(j)} \nonumber\\
&\asymp \frac{\left(p(j) - q(j) + \eta(j)\right)^2}{q(j)}. \label{eq: phubB}
\end{align}

Comparing equations~\eqref{eq: psubA} with~\eqref{eq: phubA}, and~\eqref{eq: psubB} with~\eqref{eq: phubB}, we conclude
\begin{align*}
\sum_{i \in A} \left( \sqrt{\psub(i)} - \sqrt{\qsub(i)} \right)^2 &\gtrsim \sum_{i \in A} \left( \sqrt{p'_\Hub(i)} - \sqrt{q'_\Hub(i)} \right)^2, \quad \text{ and }\\
\sum_{j \in B} \left( \sqrt{\psub(j)} - \sqrt{\qsub(j)} \right)^2 &\gtrsim \sum_{j \in B} \left( \sqrt{p'_\Hub(j)} - \sqrt{q'_\Hub(j)} \right)^2.
\end{align*}
Adding these two, we conclude 
\begin{align*}
\dhel^2(\psub, \qsub) \gtrsim \dhel^2(p'_\Hub, q'_\Hub),
\end{align*}
which is what we wanted to prove.

\paragraph{Part II:} We now move to proving the upper bound. The proof is along similar lines to that of Theorem~\ref{theorem: sub_tv}, where we showed $\nsTV(\eps) \lesssim \nsSub((2+\delta_0) \eps)$. For readability, we shall prove $\nsHub(\eps) \lesssim \nsSub(8 \eps)$  and explain the minor modifications needed to drive the constant down to $1+\delta_0$. 

\paragraph{First case when $\cu_\Hub(\eps) \in [2,\infty)$:} From the formula for Huber clips in equation~\eqref{eq: clips_hub}, we have
\begin{align}\label{eq: cupper_hub2}
\frac{p(H) - \cu_\Hub(\eps)q(H)}{\cu_\Hub(\eps)} = \frac{\eps}{1-\eps},
\end{align}
where $H$ is the set of points where $p(i)/q(i) \geq \cu_\Hub(\eps)$. In contrast, the value of $\cu_\Sub(\eps)$ is determined by the equations~\eqref{eq: clips_sub}; i.e.,
\begin{align}\label{eq: cupper_sup2}
p(H) - \cu_\Sub(\eps) q(H) = \frac{\eps}{1+\eps},
\end{align}
when the solution exists, where $H$ is understood to be the set of points where $p_i/q_i \geq \cu_\Sub(\eps)$, otherwise we use the convention $\cu_\Sub(\eps) = \infty$. Observe that if we use the same $\cu_\Hub$ in the formula for subtractive contamination above, the left hand side will evaluate to $\cu_\Hub(\eps)\eps/(1-\eps) > \eps/(1+\eps)$.  Hence, the $\cu_\Sub(\eps)$ must be larger than $\cu_\Hub(\eps)$. In particular, we can say that 
\begin{align*}
\cu_\Sub(\eps) \ge \cu_\Hub(\eps) \ge 2.
\end{align*}
Thus for both Huber and subtractive contamination, ``most of the action'' (of reducing the mass of $p$) happens at high likelihood ratios. Consider the same notation as in Lemmas~\ref{prop: approx_hell} and~\ref{prop: monotonic_hell}. Computing $\tilde h_A^2$ for the Huber LFDs yields:
\begin{align}
\tilde h^2_A(\phub(\eps), \qhub(\eps)) &= \sum_{i\in A_1} \frac{(\phub(i)-\qhub(i))^2}{\phub(i)} + \phub(A_2)\notag \\ 
&\stackrel{(a)}= (1-\eps)\sum_{i\in A_1} \frac{(p_i-q_i)^2}{p_i} + (1-\eps)p(A_2) \label{eq: case1hub}\\
&\asymp \sum_{i\in A_1} \frac{(p_i-q_i)^2}{p_i} + p(A_2)\label{eq: case1hub2}
\end{align}
where in $(a)$ we used the equality $\phub(i) = (1-\eps)p(i)$ for all $i \in A_2$. Recall the calculation for subtractive contamination from equation~\eqref{eq: case1sub_2}, and the subsequent argument showing $p(A_2) \asymp p(A_2) - \eps/(1+\eps)$, which yielded 
 \begin{align}
\tilde h^2_A(\psub(\eps), \qsub(\eps)) 
& \asymp \sum_{i\in A_1} \frac{(p_i-q_i)^2}{p_i} + p(A_2). \label{eq: case1sub22}
\end{align}
It is clear that expressions \eqref{eq: case1hub2} and \eqref{eq: case1sub22} are within constants of each other. This shows that for the case of $\cu_\Hub(\eps) \geq 2$, we have 
\begin{align}
\tilde h^2_A(\psub(\eps), \qsub(\eps)) \asymp \tilde h^2_A(\phub(\eps), \qhub(\eps)). \label{eq: case1_A1}
\end{align}
A similar proof will show that if $\cl_\Hub (\eps) \le 1/2$, then 
\begin{align}
\tilde h^2_B(\psub(\eps), \qsub(\eps)) \asymp \tilde h^2_B(\phub(\eps), \qhub(\eps)). \label{eq: case1_B1}
\end{align}

\paragraph{Second case when $\cu_\Hub(\eps) \in (1, 2)$:} The identity \eqref{eq: cupper_hub2} gives
\begin{align}
p(H) - \cu_\Hub(\eps) q(H) = \cu_\Hub(\eps)\frac{\eps}{1-\eps} \le 3\eps. \tag{as $\cu < 2, \eps \leq 1/4$}
\end{align}
Choose $\eps'$ such that 
\begin{align*}
\frac{\eps'}{1+\eps'} = \cu_\Hub(\eps)\frac{\eps}{1-\eps},
\end{align*}
that is
\begin{align}\label{eq: 123eps}
\eps' = \frac{\cu_\Hub(\eps)\eps}{1 - (1+\cu_\Hub(\eps))\eps} \le \frac{2\eps}{1-3\eps} \le 8\eps,
\end{align}
where the final expression assumes $\eps \le \dtv(p,q)/4 \le 1/4.$ By the identity~\eqref{eq: cupper_sup2}, it is immediate that 
\begin{align*}
\cu_\Sub(\eps') = \cu_\Hub(\eps).
\end{align*}
The important part of the proof so far is the observation that $\eps' \asymp \eps$. We now evaluate $\tilde h_A^2(\phub(\eps), \qhub(\eps))$ and compare it to $\tilde h_A^2(\psub(\eps'), \qsub(\eps'))$, crucially aided by the fact that they both have the \emph{same} clips $\cu = \cu_\Hub(\eps) = \cu_\Sub(\eps') \le 2$ and since likelihood ratios are bounded by $\cu$, there are no terms from $A_2$. Specifically,
\begin{align}
\tilde h_A^2(\phub(\eps), \qhub(\eps)) &= \sum_{i \in A} \frac{(\phub(i) - \qhub(i))^2}{\phub(i)} \nonumber\\
&= (1-\eps)\sum_{i: \frac{p_i}{q_i} \in [1, \cu)} \frac{(p_i-q_i)^2}{p_i} + \sum_{i: \frac{p_i}{q_i}  \in [\cu, \infty)} \frac{(\cu - 1)^2p_i}{\cu^2} \nonumber\\
&\ge \sum_{i: \frac{p_i}{q_i} \in [1, \cu)} \frac{(p_i-q_i)^2}{p_i} + \sum_{i: \frac{p_i}{q_i}  \in [\cu, \infty)} \frac{(\cu - 1)^2q_i}{\cu} \label{eq: case2hub}
\end{align}
where we used $p_i \geq \cu q_i$ in the last inequality.

The calculation for $\tilde h_A^2(\psub(\eps'), \qsub(\eps'))$ is the same as in the proof of Theorem~\ref{theorem: sub_tv}, namely,
\begin{align}
\tilde h_A^2(\psub(\eps'), \qsub(\eps')) &= (1+\eps') \left[\sum_{i: \frac{p_i}{q_i} \in [1, \cu)} \frac{(p_i-q_i)^2}{p_i} + \sum_{i: \frac{p_i}{q_i}  \in [\cu, \infty)} q_i \frac{(\cu-1)^2}{\cu}\right] \label{eq: case2sub2}.
\end{align}
Comparing expressions \eqref{eq: case2hub} and \eqref{eq: case2sub2}, we see that
\begin{align*}
\tilde h_A^2(\psub(\eps'), \qsub(\eps')) &\lesssim \tilde h_A^2(\phub(\eps), \qhub(\eps)).
\end{align*}

\paragraph{Putting it all together:} Finally, by the monotonicity of $\tilde h_A$ from Lemma~\ref{prop: monotonic_hell}, we conclude that no matter which case holds, we must have
\begin{align}
\tilde h_A^2(\psub(8\eps), \qsub(8\eps)) \lesssim  \tilde h_A^2(\phub(\eps), \qhub(\eps)).
\end{align}
The same result holds for $B$ as well, that is,
\begin{align}
\tilde h_B^2(\psub(8\eps), \qsub(8\eps)) \lesssim  \tilde h_B^2(\phub(\eps), \qhub(\eps)).
\end{align}
Adding up, using Lemma~\ref{prop: approx_hell}, and taking the inverse, we finally arrive at the claimed sample complexity bound
\begin{align}
\nsSub(8\eps) \gtrsim \nsHub(\eps).
\end{align}
\emph{Improving the constant to $1+\delta_0$:} To improve the constant 8 to $1+\delta_0$ for any $\delta_0 > 0$ we need to do only two changes: Adjust Lemma~\ref{prop: approx_hell} to use the thresholds $1-\delta_0'$ and $1+\delta_0'$ for some $\delta_0' < \delta_0$, instead of $1/2$ and $2$. The same result continues to hold, with constants that now depend on $\delta_0'$. Now alter the case 1 and case 2 above depending on whether $\cu_\Hub(\eps)$ lies beyond or below $1+\delta_0'$. Case 1 remains unchanged, but in case 2, the upper bound on $\eps'$ in inequality~\eqref{eq: 123eps} becomes
\begin{align*}
\eps' \leq \frac{(1 + \delta_0')\eps}{1 - (1 + \delta_0')\eps} = (1 + \delta_0')\eps + O(\eps^2) \leq (1+\delta_0) \eps.
\end{align*}
Thus, for all small enough $\eps$, we are able to show 
\begin{align}
\nsSub((1 +\delta_0)\eps) \gtrsim \nsHub(\eps).
\end{align}
This completes the proof.
\end{proof}

\section{Adaptive contamination}\label{sec: adaptive}

In this section we consider adaptive variants of additive, subtractive, and general contamination models. Our goal is to show that adaptive models are essentially as strong as the oblivious models.  We first review adaptive contamination models.

\subsection{Adaptive contamination models}\label{subsection: adaptive_models}

The main difference between adaptive and oblivious contamination is that in the former, an adversary sees the entire dataset, and then either adds more samples, deletes existing samples, or does both. 

\paragraph{Adaptive-$\TV$ contamination:} Given a dataset $S = (X_1, \dots, X_n)$, an adaptive-$\TV$ adversary chooses any $\lfloor n\eps \rfloor$ samples to replace with any samples of their choosing, and then releases the corrupted dataset $S' = (X_1', \dots, X_n')$ after uniformly permuting it. The sample complexity of hypothesis testing with this adversary is denoted by $\nsATV(\eps)$.

\paragraph{Adaptive-$\Hub$ contamination:} Given a dataset $S = (X_1, \dots, X_{n-\lfloor n\eps \rfloor})$, an adaptive-$\Hub$ adversary adds extra $\lfloor n\eps \rfloor$ samples of their choosing to $S$, and releases the $n$-sample dataset $S'$ after uniformly permuting it. The sample complexity of hypothesis testing with this adversary is denoted by $\nsAHub(\eps)$.

\paragraph{Adaptive-$\Sub$ contamination:} Given a dataset $S = (X_1, \dots, X_n)$, an adaptive-$\Sub$  adversary replaces $\lfloor n\eps \rfloor$ samples of their choosing by $\perp$ and releases the dataset $S'$. The sample complexity of hypothesis testing with this adversary is denoted by $\nsASub(\eps)$.

Observe that the adaptive-TV adversary can simulate the adaptive-Huber and adaptive-subtractive adversaries. 

\subsection{Baseline and small-$\eps$ bounds}
The following baseline bounds can be easily stated and are contained in the proof of Proposition~\ref{prop:baseline} in Appendix~\ref{app: baseline}.

\begin{proposition}\label{prop:Abaseline}
Let $p$ and $q$ be probability distributions and let $\eps \le \frac{\dtv(p,q)}{4}$. Then  
\begin{align*}
\frac{1}{\dtv(p,q)} \lesssim \nsATV(\eps), \nsAHub(\eps), \nsASub(\eps) \lesssim \frac{1}{\dtv(p,q)^2}.
\end{align*}
\end{proposition}

Recall that in Proposition~\ref{prop:small_eps}, we showed that for $\eps \lesssim \dhel^2(p,q)$, the sample complexity with contamination is essentially the same as without contamination. Here, we prove a similar result for adaptive contamination. The proof of Proposition~\ref{prop:small_eps} relied on the existence of LFDs, so it does not work for the adversarial contamination models. However, it is still possible to prove a similar result by explicitly constructing a test for the small $\eps$ regime. 

\begin{proposition}[Adversarial sample complexity for small $\eps$]\label{prop:small_eps_adv}
    Fix $p$ and $q$ and let $\eps \le \frac{\dhel^2(p,q)}{8}$. Then $\nsATV(\eps), \nsAHub(\eps), \nsASub(\eps) \asymp \frac{1}{\dhel^2(p,q)}$.
\end{proposition}
\begin{proof}[Proof of Proposition~\ref{prop:small_eps_adv}]
The lower bound of $\Omega(1/\dhel^2(p,q))$ follows from the uncontaminated case, so we'll focus on proving a matching upper bound for $\nsATV(\eps)$, as it will automatically imply upper bounds for the other two settings.   

Consider $n$ i.i.d.\ samples $X_1,\dots,X_n$ drawn from either $p$ or $q$.
An adversary, after seeing the entire sample, may replace up to $\eps n$ of the samples arbitrarily. Define the (single-sample) statistic
\begin{align*}
\hstat(x) = \frac{\sqrt{p(x)}-\sqrt{q(x)}}{\sqrt{p(x)}+\sqrt{q(x)}} \in [-1,1],
\end{align*}
and let $\olh=\tfrac1n\sum_{i=1}^n \hstat(X_i)$.
Statistics similar to this have appeared in the literature~\citep{LeCam86, Sur21}. The test is: declare $\cP$  iff $\olh \ge \tau$ where $\tau=(\mu_p+\mu_q)/2$ and $\mu_p=\E_p[\hstat],\ \mu_q=\E_q[\hstat]$. We claim that this test achieves Type I and Type II errors of at most $1/20$ given $O(1/\dhel^2(p,q))$ samples provided $\eps \le \dhel^2(p,q)/8$. The following lemma is proved in Appendix~\ref{app: h-properties}.

\begin{lemma}[Properties of $\hstat$]
\label{lem:h-properties}
With $\hstat$ as above:
\begin{enumerate}
  \item $|\hstat(x)|\le 1$ for all $x$.
  \item $\mu_p-\mu_q = \dhel^2(p,q)$.
  \item $\Var_p(\hstat)\le \dhel^2(p,q)$ and $\Var_q(\hstat)\le \dhel^2(p,q)$.
\end{enumerate}
\end{lemma}

Our next lemma, proved in Appendix~\ref{app: adv-shift}, shows that the adaptive-TV adversary cannot change the test statistic significantly.
\begin{lemma}
\label{lem:adv-shift}
Suppose the uncorrupted single-sample values of the test statistic are $h_1,\dots,h_n\in[-1,1]$ and the overall test statistic be $\olh$. Let the $\olh^{\Adv}$ be the test statistic a dataset contaminated by an adaptive-TV adversary. Then
\begin{align*}
\big|\olh - \olh^{\Adv}\big| \le 2\eps.
\end{align*}
\end{lemma}

Under $p$, we empower the adversary to reduce the test statistic by $\dhel^2(p,q)/4$ (the maximum possible value of $2\eps$) each time and increase the test statistic by the same value under $q$. We can upper-bound the Type-I and Type-II errors for this (empowered) adversary as follows. Under $p$, our test will err only if $\olh \le \mu_p - \frac{\dhel^2}{4}$, because when this event happens, the adversary can bump it further lower by $\dhel^2(p,q)/4$ which is beyond the threshold $\tau$.  The probability of this event is bounded via Chebyshev's inequality and Lemma \ref{lem:h-properties},
\[
\prob_p\big(|\olh-\mu_p| \ge \tfrac{1}{4}\dhel^2(p,q)\big)
\le \frac{\Var_p(\hstat)/n}{(\dhel^2(p,q)/4)^2}
\le \frac{\dhel^2(p,q)/n}{(1/16)\dhel^4(p,q)}
= \frac{16}{n\dhel^2(p,q)}.
\]
The same bound holds for the Type-II error under $q$. Therefore to make both errors $\le 1/20$ it suffices to pick $n \geq \frac{320}{\dhel^2(p,q)}$. This shows the desired upper bound of $O(1/\dhel^2(p,q))$ on $\nsATV(\eps)$ for all $\eps \leq \dhel^2(p,q)/8$.
\end{proof}

\begin{remark}
The test used in the proof of Proposition~\ref{prop:small_eps_adv} also works for oblivious-TV contamination. To see this, we may replace the right hand side in Lemma~\ref{lem:adv-shift} by $4\eps$ and use the multiplicative-Chernoff bound to show that the claimed bound holds with high probability (say, 0.99) provided $n \gtrsim \frac{1}{\eps} \gtrsim \frac{1}{\dhel^2(p,q)}$. The rest of the proof remains largely the same: we empower the adversary to reduce the test statistic by $4\eps$, and adjust the values of the constants to ensure the errors of both types are at most $1/20$.
\end{remark}

\subsection{Comparing adaptive and oblivious contamination models}

Our goal in this section is to show that the sample complexities under adaptive contamination are comparable to those with oblivious contamination, after scaling the contamination parameter by $(1\pm \delta_0)$.

We first show that adaptive contamination can simulate oblivious contamination, thereby showing a lower bound on the sample complexity with adaptive contamination. Throughout this section, we shall consider $\eps > \dhel^2(p,q)/9$ to address the regime not covered already by Propositions~\ref{prop:small_eps} and~\ref{prop:small_eps_adv}.

\begin{theorem}[Adaptive adversaries are harder than oblivious adversaries]
\label{thm: adaptive_simulates}
Let $\eps > \dhel^2(p,q)/9$ and fix $\delta_0 > 0$. Then for all oblivious contamination models $\mathrm{O} \in \{\TV, \Hub, \Sub\}$ and their adaptive counterparts $\mathrm{A} \in \{\mathrm{A\!-\!TV}, \mathrm{A\!-\!Hub}, \mathrm{A\!-\!Sub}\}$, we have
\begin{align*}
\ns_{\mathrm{O}}(\eps(1-\delta_0)) \lesssim \ns_{\mathrm{A}}(\eps)
\end{align*}
where the implied constant depends only on \(\delta_0\), and not on \(p,q,\eps\).
\end{theorem}

\begin{proof}
We first prove the result for oblivious and adaptive TV contamination. We'll prove the claim with $\delta_0 = 1/2$, and remark on the minor modifications needed to adapt the proof to any $\delta_0 > 0$. 

\paragraph{TV contamination:} Fix distributions $p,q$. Let $T$ be any test that on $N \asymp \nsATV(\eps)$ samples such that both type-I and type-II error at most $1/20$ against an adaptive-TV adversary. First, observe that $N \gtrsim \frac{1}{\dhel^2(p,q)} \gtrsim 1/\eps$, since we are assuming $\eps \gtrsim \dhel^2(p,q)$. For easier calculations in what follows, assume $N \geq 6\log(20)/\eps$ in what follows.

Let $p^*$ satisfy $\dtv(p,p^*)\le\eps/2$. By the coupling characterisation of total variation there exists a coupling $\pi$ of $(X,Y)$ with
marginals $X\sim p$, $Y\sim p^*$ such that when $(X,Y) \sim \pi$, the disagreement probability satisfies the bound $\prob(X\neq Y)\le\eps/2$.
Draw $N$ independent coupled pairs $(X_i,Y_i)$ and define
\begin{align*}
K:= |\{i: Y_i\ne X_i\}| \sim \mathrm{Bin}(N,p_0),\qquad p_0\le\eps/2,
\end{align*}
so $\mu:=\E K=N p_0\le N\eps/2$.

Applying the multiplicative Chernoff bound, we have for any
$\delta\ge 1$,
\begin{align*}
\prob (K \ge (1+\delta)\mu)\le \exp\left(-\frac{\mu\delta}{3}\right).
\end{align*}
Set $(1+\delta)\mu=N\eps$ (this is feasible since
$\mu\le N\eps/2$, hence $\delta \ge 1$). Then
$\mu\delta = N\eps-\mu \ge N\eps/2$, so
\begin{align*}
\prob(K\ge N\eps)\le\exp \left(-\frac{N\eps}{6}\right).
\end{align*}
In particular, since $N \ge 6(\log 20)/\eps$, we have $\prob(K \ge N\eps)\le 1/20$.

Now consider an oblivious adversary which, given the clean sample $X_{1:N} = (X_1,\dots,X_N) \sim p^{\otimes N}$, samples the corresponding
$Y_i$ from the conditional distribution $\pi_{Y|X}$. This produces a dataset $Y_{1:N} = (Y_1, \dots, Y_N) \sim p^{* \otimes N}$. Crucially, with probability at least $1-1/20$, there are at most $N\eps$ positions of disagreement between $X_i$s and $Y_i$s. Consider an adaptive adversary that generates $Y'_{1:N} = (Y'_1,\dots, Y'_N)$ that is as similar to $(Y_1,\dots, Y_N)$ as possible. Specifically,  set $Y'_i = Y_i$ for the first $N\eps$ locations where $X_i \neq Y_i$, and set $Y'_i = X_i$ at all other locations. For this $\eps$-TV adaptive adversary, we know that the test $T$ yields an error of at most $1/20$. However, observe that since $(Y_1,\dots, Y_N)$ is identical to $(Y'_1, \dots, Y'_N)$ with probability at least 19/20 (because $K \leq N\eps$ with probability at least $19/20$) the same test $T$ will output the incorrect hypothesis on $(Y_1, \dots, Y_N)$ with probability at most $1/20 + 1/20 = 1/10$. To be more precise, the type-I error is bounded as
\begin{align*}
\prob(T(Y_{1:N}) = 1) &= \prob(Y_{1:N} = Y'_{1:N}, T(Y'_{1:N}) = 1) + \prob(Y_{1:N} \neq Y'_{1:N}, T(Y_{1:N}) = 1)\\
&\leq \prob(T(Y'_{1:N}) = 1) + \prob(Y_{1:N} \neq Y'_{1:N})\\
&\leq 1/20 + 1/20 = 1/10. 
\end{align*}
Repeating the symmetric argument for $q$ and $q^*$ shows $T$ distinguishes $p$ vs $q$ under TV-contamination radius $\eps/2$ with error at most $1/10$. This yields the bound
\begin{align*}
\nsTV(\eps/2) \lesssim \nsATV(\eps).
\end{align*}

\emph{Adapting the proof for any $\delta_0 > 0$:} When adapting the proof for a constant of $1-\delta_0$, we follow the same steps as above but make the following observations. The upper-bound on $p_0$ is $(1-\delta_0)\eps$, and so $\mu \leq N\eps(1-\delta_0)$. Choose $\delta$ such that $N\eps = (1+\delta)\mu$. Note that $\delta \geq \delta_0/(1-\delta_0)$. Now apply the multiplicative Chernoff bound to get
\begin{align*}
\prob(K \geq N\eps) \leq \exp\left(- \frac{\mu \delta^2}{2+\delta} \right) = \exp\left(-\frac{N\eps\delta^2}{(1+\delta)(2+\delta)} \right) \leq \exp\left(-\frac{N\eps\delta_0^2}{(2-\delta_0)} \right),
\end{align*}
where in the last line we substituted $\delta$ by $\delta_0/(1-\delta_0)$, used that $f(x) = x^2/((2+x)(1+x))$ is increasing for $x \geq 0$, and simplified the resulting expression. It is now clear that starting with $N \geq  (\log(20)(2-\delta_0)/(\eps\delta_0^2)) \asymp 1/\eps$, we get the desired error bounds.

\paragraph{Huber contamination:} Let $p^* = (1-\eps/2)p + (\eps/2)h$ be any distribution in the $\eps/2$-Huber uncertainty set. Consider i.i.d.\ $\Ber(\eps/2)$ random variables $B_1,\dots, B_N$, and set $S = \sum_{i=1}^N B_i$. We think of $B_i$ as the indicator of whether sample $X_i$ is drawn from $p$ or $h$. Observe that $Y_{1:N} \sim p^{*\otimes N}$ can be generated by first drawing $S \sim \mathrm{Bin}(N, \eps/2)$, generating $S$ i.i.d.\ samples from $h$ and $N-S$ i.i.d.\ samples from $p$, and randomly permuting them. Let us call the generated contaminated set (up to permutation) as $Y_{1:N} = (X_{1:N-S}, Z_{1:S})$. 

Using the multiplicative Chernoff bound as before, we argue that the number of contaminated points satisfies $S \leq N\eps$ with probability at least 19/20. Now consider an adaptive adversary that imitates $Z_{1:S}$ as much as possible, that is, it generates $Y'=(X_{1:N-N\eps},X'_{N-N\eps+1:N-S},Z_{1:S})$ where $X'_{N-N\eps+1},\dots,X'_{N-S}\sim p$ are fresh i.i.d.\ samples, and $Z'_i = Z_i$ for the first $N\eps$ indices. (When $S > N\eps$, the middle block of $X'_{N-N\eps+1:N-S}$ is dropped from $Y'_{1:N}$.) Note that when $S \leq N\eps$, the distribution of $Y_{1:N}$ is identical to that of $Y'_{1:N}$. Since $Y'_{1:N}$ corresponds to an $\eps$-adaptive Huber contamination, the test $T$ is assured to give a type-I error of at most $1/20$ for this adversary. Using a similar argument as above, this shows that the same test $T$ will yield a type-I error of at most $1/10$ over $Y_{1:N}$. Arguing in the same manner for the type-II error, we conclude 
\begin{align*}
\nsHub(\eps/2) \lesssim \nsAHub(\eps).
\end{align*}
The final argument where we tighten the bound using the multiplicative Chernoff bound is identical to the TV case, and we may conclude that for any $\delta_0 > 0$,
\begin{align*}
\nsHub((1-\delta_0)\eps) \lesssim \nsAHub(\eps).
\end{align*}

\paragraph{Subtractive contamination:} It is easier to argue the required bound by interpreting subtractive contamination that leads to a random-sized dataset. Recall that for characterising the sample complexity up to constant factors, there is no difference between the fixed-size and random-sized versions of subtractive contamination. 

Consider $p^*$ in the $\eps/2$-subtractive uncertainty set around $p$. Let $a(x) = p'(x)/(p(x)(1+\eps/2))$. Recall that subtractive contamination can be thought of as starting with the clean dataset $(X_1,\dots, X_N) \sim p^{\otimes N}$, generating the contaminated set $(Y_1, \dots, Y_N)$ by setting $Y_i = X_i$ with probability $a(X_i)$ and $Y_i \,=\, \perp$ with probability $1-a(X_i)$, independently over all $i$. Observe that $\prob(Y_i \,=\, \perp) = (\eps/2)/(1+(\eps/2)) < \eps/2$. Using the multiplicative Chernoff bound, the number of locations where $Y_i \,=\, \perp$ is at most $N\eps$ with probability at least $19/20$.

Now consider an $\eps$-adaptive subtractive adversary that tries to imitate $Y_{1:N}$ as much as possible. Specifically, this adversary sets $Y'_i \,=\, \perp$ for the first $N\eps$ locations, and keeps $Y'_i = X_i$ for the remaining locations. As before, we have $Y'_{1:N} = Y_{1:N}$ with probability at least 19/20, and a test $T$ that succeeds with type-I error of 1/20 for $\eps$-adaptive subtractive adversaries with $N$ samples will continue to succeed for the $\eps/2$-oblivious adversary with $\asymp N$ samples.  The $\eps/2$ can be replaced by any $(1-\delta_0)$ by tightening the analysis as outlined earlier, giving
\begin{align*}
\nsSub((1-\delta_0)\eps) \lesssim \nsASub(\eps).
\end{align*}
\end{proof}

We now prove the upper bound:
\begin{theorem}[Oblivious adversaries are harder than adaptive adversaries]
\label{thm: adaptive_coupling}
Let $\eps > \dhel^2(p,q)/9$ and fix $\delta_0 > 0$. Then for all oblivious contamination models $\mathrm{O} \in \{\TV, \Hub, \Sub\}$ and their adaptive counterparts $\mathrm{A} \in \{\mathrm{A\!-\!TV}, \mathrm{A\!-\!Hub}, \mathrm{A\!-\!Sub}\}$, we have
\begin{align*}
\ns_{\mathrm{A}}(\eps) \lesssim \ns_{\mathrm{O}}((1+\delta_0)(\eps))
\end{align*}
where the implied constant depends only on $\delta_0$ and not on $p,q,$ or $\eps$.
\end{theorem}

\begin{proof}
We prove the TV-contamination case first with $\delta_0 = 1$. The proof adaptations required to establish the result for any $\delta_0 > 0$ are minor, and we shall describe them at the end of the proof.

\paragraph{TV-contamination:} We'll show that the clipped-likelihood ratio test continues to work for adaptive-contamination. Fix distributions $p$ and $q$ and let $(p^*,q^*)$ denote the TV-LFDs for $(2\eps)$-contamination, and let $\cl, \cu$ be the corresponding clips with $\Delta:= \log \cu - \log \cl.$ Define the clipped log-likelihood
\begin{align*}
\psi(x) = \log \frac{p^*(x)}{q^*(x)} \in [\log \cl, \log \cu].
\end{align*}
Consider the test $T$ given by
\begin{align*}
T(x_1, \dots, x_N) = 
\begin{cases}
0 \quad &\text{ if } \quad \sum_{i=1}^N \psi(x_i) \geq 0,\\
1 \quad &\text{ otherwise.}
\end{cases}
\end{align*}
Let $N \asymp \nsTV(2\eps)$ such that test $T$ has   type-I and type-II errors of at most $1/20$ under $(2\eps)$-oblivious contamination. For easier calculations later, assume $N \geq 4\log(20)/\eps$. Note that this may be assumed as $N \gtrsim 1/\dhel^2(p,q) \gtrsim 1/\eps$, where we used the assumption $\eps \geq \dhel^2/9$.

Note that $\dtv(p, p^*) = 2\eps$, and there exists a coupling $\pi$ such that for $(X,Y) \sim \pi$, we have $X \sim p$, $Y \sim p^*$, and $\prob(X \neq Y) = 2\eps$. Consider the set $L$, $M$, and $H$ defined by thresholding the likelihood ratio $p(x)/q(x)$ at $\cl$ and $\cu$. The coupling $\pi$ may be explicitly written as 
\begin{align*}
\pi(x,y) = 
\begin{cases}
\min\{p(x), p^*(x)\} \quad &\text{ when } \quad x=y\\
\frac{(p(x)-p^*(x))(p^*(y)-p(y))}{2\eps}&\text{ when } \quad (x,y) \in H \times L \\
0 \quad &\text{ otherwise.}
\end{cases}
\end{align*}
The key point is to note that when $x \neq y$, the above coupling only assigns positive mass to pairs where $(x,y) \in H \times L$, where we have $\psi(x) - \psi(y) = \log \cu - \log \cl = \Delta$.

Now consider any $\eps$-adaptive adversary that acts on $X_{1:N}$ to produce $Y'_{1:N}$. Observe that 
\begin{align*}
\sum_{i=1}^N \psi(Y_i') \geq \sum_{i=1}^N \psi(X_i) - N\eps\Delta.
\end{align*}
Thus, the error probability of the test $T$ when used for the $\eps$-adaptive adversary may be upper-bounded as
\begin{align*}
\prob(T(Y'_{1:N}) = 1) &= \prob \left( \sum_{i=1}^N \psi(Y'_i) \leq 0 \right)\\
&\leq \prob \left( \sum_{i=1}^N \psi(X_i) \leq N\eps\Delta \right)\\
&= \prob \left( \sum_{i=1}^N \psi(Y_i) + \sum_{i=1}^N \psi(X_i) - \psi(Y_i) \leq N\eps\Delta \right)\\
&\leq \prob \left( \sum_{i=1}^N \psi(Y_i) \leq 0\right) + \prob \left( \sum_{i=1}^N \psi(X_i) - \psi(Y_i) \leq N\eps\Delta \right).
\end{align*}
Since $T$ yields a type-I error of at most $1/20$ on the contaminated dataset $Y_{1:N}$, the first term is at most $1/20$. To bound the second term, note that 
\begin{align*}
\psi(X_i) - \psi(Y_i) = 
\begin{cases}
0 &\quad \text{ with probability } \quad (1-2\eps),\\
\Delta &\quad \text{ with probability } \quad (2\eps).
\end{cases}
\end{align*}
Hence, we may think of the second term as a sum of $N$ i.i.d.\ $\Ber(2\eps)$ random variables $B_1, \dots, B_N$, scaled by $\Delta$. By the multiplicative Chernoff bound for the lower tail,
\begin{align*}
\prob\left(\sum_{i=1}^N \psi(X_i) - \psi(Y_i) \le N\eps\Delta \right)
&=\prob\left(\sum_{i=1}^N B_i \le N\eps \right)\\
&\le \exp\left(-\frac{N\eps}{4}\right)\\
&\leq \frac{1}{20},
\end{align*}
where the final line used $N \geq 4\log(20)/\eps$. Thus, we conclude that 
\begin{align*}
\prob(T(Y'_{1:N}) = 1) &\leq \frac{1}{10}.
\end{align*}
Using a similar argument for the type-II error, we conclude that the same test $T$ continues to succeed for $\eps$-adaptive contamination with $N$ samples, yielding 
\begin{align*}
\nsATV(\eps) \lesssim \nsTV(2\eps).
\end{align*}

\emph{Adapting the proof to any $\delta_0>0$:} In this case, we have $\dtv(p, p^*) = (1+\delta_0)\eps$. Following the same coupling-based strategy, we arrive at 
\begin{align*}
\prob(T(Y'_{1:N}) = 1) &\leq \prob \left( \sum_{i=1}^N \psi(Y_i) \leq 0\right) + \prob \left( \sum_{i=1}^N \psi(X_i) - \psi(Y_i) \leq N\eps\Delta \right).
\end{align*}
The first term is again bounded above by $1/20$, and for the second term we observe that 
\begin{align*}
\psi(X_i) - \psi(Y_i) = 
\begin{cases}
0 &\quad \text{ with probability } \quad (1-(1+\delta_0)\eps),\\
\Delta &\quad \text{ with probability } \quad ((1+\delta_0)\eps).
\end{cases}
\end{align*}
Now using the multiplicative Chernoff bound for $\sum_{i=1}^N B_i$ where $B_i$ are i.i.d.\ $\Ber(\eps(1+\delta_0))$ random variables yields
\begin{align*}
\prob \left( \sum_{i=1}^N \psi(X_i) - \psi(Y_i) \leq N\eps\Delta \right) = \prob \left( \sum_{i=1}^N B_i \leq N\eps \right) \leq e^{-\frac{\delta_0^2 N \eps}{2(1+\delta_0)}}.
\end{align*}
Thus, taking $N \geq 2(1+\delta_0)\log(20)/(\delta_0^2 \eps)$ gives the desired bound of $1/20$.

\paragraph{Huber contamination: } Just as the TV-case, we'll show that the clipped-likelihood ratio for $(2\eps)$-Huber contamination continues to work for $\eps$-adaptive Huber contamination. 

Fix distributions $p$ and $q$ and let $(p^*,q^*)$ denote the Huber-LFDs for $(2\eps)$-contamination, and let $\cl, \cu$ be the corresponding clips. Consider the set $L$, $M$, and $H$ defined by thresholding the likelihood ratio $p(x)/q(x)$ at $\cl$ and $\cu$. Define the clipped log-likelihood
\begin{align*}
\psi(x) = \log \frac{p^*(x)}{q^*(x)} \in [\log \cl, \log \cu].
\end{align*}
Consider the optimal test $T$ given by
\begin{align*}
T(x_1, \dots, x_N) = 
\begin{cases}
0 \quad &\text{ if } \quad \sum_{i=1}^N \psi(x_i) \geq 0,\\
1 \quad &\text{ otherwise.}
\end{cases}
\end{align*}
Let $N \asymp \nsHub(2\eps)$ such that test $T$ with $N$ samples with $(2\eps)$-oblivious Huber contamination has type-I and type-II errors of at most $1/20$. For easier calculations later, assume $N \geq 4\log(20)/\eps$. Note that this may be assumed as $N \gtrsim 1/\dhel^2(p,q) \gtrsim 1/\eps$, where we used the assumption $\eps \geq \dhel^2/9$.

Express $p^* = (1 - 2\eps)p + (2\eps)h$. Observe that, based on the LFD-formula~\eqref{eq: lfd_hub}, the contaminating distribution $h$ is supported only on $L$. The $(2\eps)$-contaminated dataset $Y_{1:N}$ is generated by setting $Y_i = X_i$ with probability $1-2\eps$, and drawing $Y_i \sim h$ otherwise, independently over all $i$. Let $S \sim \mathrm{Bin}(N, 2\eps)$ be the number of contaminated samples.  

Consider any $\eps$-adaptive Huber adversary that adds $N\eps$ new points to a clean dataset $X_{1:N-N\eps}$ to generate $Y'_{1:N}$. Observe that
\begin{align*}
\sum_{i=1}^N \psi(Y'_i) \geq \sum_{i=1}^{N-N\eps} \psi(X_i) + N\eps \log (\cl).
\end{align*}
For this adversary, the type-I error of the test $T$ may be upper-bounded as
\begin{align*}
\prob(T(Y'_{1:N}) = 1) &= \prob \left( \sum_{i=1}^N \psi(Y'_i) \leq 0 \right)\\
&\leq \prob \left( \sum_{i=1}^{N-N\eps} \psi(X_i) \leq N\eps\log(1/\cl) \right)\\
&\leq  \prob \left( \sum_{i=1}^{N-N\eps} \psi(X_i) + N\eps\log(\cl) \leq 0, S \geq N\eps \right) +  \prob \left( \sum_{i=1}^{N-N\eps} \psi(X_i) + N\eps\log(\cl) \leq 0, S < N\eps \right).
\end{align*}
For the first term, since $S \geq N\eps$, we have 
\begin{align*}
\sum_{i=1}^N \psi(Y_i) = \sum_{i=1}^{N-S} \psi(X_i) + S\log(\cl) \leq \sum_{i=1}^{N-N\eps} \psi(X_i) + N\eps\log(\cl),
\end{align*}
Here we used (i) the choice of coupling that places the clean
indices of $Y$ at positions $1,\dots,N-S$; (ii) $h$ is supported entirely on $L$ and so $\psi(Y_i) = \log \cl$ a.s. when $X_i$ is contaminated; and (iii) that
$\psi(X_i)\ge\log\cl$ for every $i$, so
$\sum_{i=N-S+1}^{N-N\eps}\psi(X_i)\ge(S-N\eps)\log\cl$. Plugging this back,
\begin{align*}
\prob(T(Y'_{1:N}) = 1) &\leq \prob \left( \sum_{i=1}^{N} \psi(Y_i) \leq 0\right) +  \prob \left(S < N\eps \right)\\
&\leq 1/20 + 1/20 = 1/10, 
\end{align*}
where in the final line we used the fact that $T$ has a type-I error of at most 1/20 under $(2\eps)$-Huber contamination, and a multiplicative Chernoff bound for the second term. Showing a similar bound for the type-II error, and adapting the argument suitable for any $\delta_0 > 0$, we conclude
\begin{align*}
\nsAHub(\eps) \lesssim \nsHub((1+\delta_0)\eps).
\end{align*}

\paragraph{Subtractive contamination: } We'll show that the clipped-likelihood ratio test corresponding to the LFDs continues to work for adaptive subtractive contamination as well. Fix distributions $p$ and $q$ and let $(p^*,q^*)$ denote the LFDs for $(2\eps)$-contamination, and let $\cl, \cu$ be the corresponding clips. We first bound the type-I error in the $\eps$-adaptive setting assuming that $\cu < \infty$. When $\cu = \infty$, the proof is more straightforward; we shall deal with it later. Define the clipped log-likelihood
\begin{align*}
\psi(x) = \log \frac{p^*(x)}{q^*(x)} \in [\log \cl, \log \cu].
\end{align*}
Consider the optimal test $T$ given by
\begin{align*}
T(x_1, \dots, x_N) = 
\begin{cases}
0 \quad &\text{ if } \quad \sum_{i=1}^N \psi(x_i) \geq 0,\\
1 \quad &\text{ otherwise.}
\end{cases}
\end{align*}
Let $N \asymp \nsSub(2\eps)$ such that test $T$ with $N$ samples with $(2\eps)$-oblivious subtractive contamination has type-I and type-II errors of at most $1/20$. As before, assume $N \gtrsim 1/\eps$. 

Consider any adaptive adversary that observes $X_{1:N} \sim p^{\otimes N}$ and changes it to $Y'_{1:N}$ such that $Y'_i \,=\, \perp$ for at most $N\eps$ indices. Observe that
\begin{align*}
\prob\left( T(Y'_{1:N}) = 1 \right) &= \prob\left(\sum_{i=1}^N \psi(Y_i') \leq 0 \right)\\ 
&\leq \prob\left(\sum_{i=1}^N \psi(X_i) \leq N\eps \log(\cu)\right).
\end{align*}
Above, we set $\psi(\perp) = 0$. Generate the oblivious contaminated dataset $Y_{1:N}$ by setting $Y_i = X_i$ with probability $a(X_i) = p^*(X_i)/((1+2\eps)p(X_i))$ and $\perp$ otherwise, independently over $i$. As in the Huber case, consider $S \sim \mathrm{Bin}(N, 2\eps/(1+2\eps))$, be the number $\perp$'s in $Y_{1:N}$. The crucial point to note is that $a(x) < 1$ only for those $x$ with $p(x)/q(x) \geq \cu$; i.e., only samples with high-likelihood ratios are deleted. We may upper-bound the desired error by
\begin{align*}
\prob\left( T(Y'_{1:N}) = 1 \right) &\le   \prob \left( \sum_{i=1}^{N} \psi(X_i) - N\eps\log(\cu) \leq 0, S \geq N\eps \right) +  \prob \left( \sum_{i=1}^{N} \psi(X_i) - N\eps\log(\cu) \leq 0, S < N\eps \right).
\end{align*}
When $S \geq N\eps$, we must have
\begin{align*}
\sum_{i=1}^N \psi(Y_i) = \sum_{i=1}^{N} \psi(X_i) - S\log(\cu) \leq \sum_{i=1}^{N} \psi(X_i) - N\eps\log(\cu),
\end{align*}
where we used the observation that only samples with high-likelihood ratios get deleted. Thus, we arrive at the upper-bound
\begin{align*}
\prob\left( T(Y'_{1:N}) = 1 \right) &\le   \prob \left( \sum_{i=1}^N \psi(Y_i) \leq 0, S \geq N\eps \right) +  \prob \left( \sum_{i=1}^{N} \psi(X_i) - N\eps\log(\cu) \leq 0, S < N\eps \right)\\
&\leq \prob \left( \sum_{i=1}^N \psi(Y_i) \leq 0  \right) + \prob \left(S < N\eps \right)\\
&\leq 1/20 + 1/20 = 1/10,
\end{align*}
where in the last line we used the fact that type-I error under $(2\eps)$-oblivious contamination is at most 1/20, and a multiplicative Chernoff bound for the second term. 

\paragraph{Subtractive contamination with $\cl = 0$ or $\cu = \infty$:} Suppose $\cu = +\infty$. Recall that for $(2\eps)$-subtractive contamination, this happens when $p(\bar H) \geq \frac{2\eps}{(1+2\eps)}$, where $\bar H = \{x : p(x) > 0 \text{ and } q(x) = 0\}$. Observe that the likelihood ratio test outputs $p$ whenever the dataset contains $x \in \bar H$. Given $N$ samples from $p$, the expected number of samples from $\bar H$ is at least $2N\eps/(1+2\eps)$. For $N \asymp \nsSub(2\eps) \gtrsim 1/\eps$, a simple application of the multiplicative Chernoff bound yields that with probability at least $9/10$, the number of samples from $\bar H$ is at least $N\eps$. This means that the $\eps$-adaptive subtractive adversary cannot delete all samples from $\bar H$---which comprise irrefutable evidence for $p$--- from the dataset with probability at least $9/10$, hence yielding a type-I error of at most $1/10$. 

This shows that for any $\cu$, finite or infinite, the type-I error under $\eps$-adaptive subtractive contamination is bounded above by 1/10. A similar proof works to bound the type-II error for any $\cl \ge 0$ as well. We may adapt the proof to any $\delta_0 > 0$ using similar ideas from the TV-contamination case, and thereby conclude
\begin{align*}
\nsASub(\eps) \lesssim \nsSub((1+\delta_0)\eps).
\end{align*}
\end{proof}

As a final observation, we note that $\nsAHub$, $\nsATV$, and $\nsASub$ may have polynomial jumps in the sample complexity for $O(\eps)$ perturbations in the contamination parameter $\eps$. To see this, consider the TV-contamination case. We have already shown that in some cases, we may have
\begin{align*}
\nsTV(\eps) \ll \ns(\eps(1+\delta_1)),
\end{align*}
for some arbitrarily small $\delta_1 > 0$. Using Theorems~\ref{thm: adaptive_simulates} and~\ref{thm: adaptive_coupling}, we have
\begin{align*}
\nsATV(\eps/(1+\delta_0)) \lesssim \nsTV(\eps) \ll \nsTV(\eps(1+\delta_1)) \lesssim \nsATV(\eps(1+\delta_0)(1+\delta_1))
\end{align*}
In particular, the adversarial sample jumps polynomially when the contamination parameter is scaled by $(1+\delta_0)^2(1+\delta_1)$, which may be arbitrarily close to 1. This is noted in the following corollary:

\begin{corollary}
For each $A \in \{\mathrm{A\!-\!Hub},\mathrm{A\!-\!Sub},\mathrm{A\!-\!TV} \}$, there exist instances for which the adaptive sample complexity $\ns_{A}(\eps)$ undergoes a polynomial jump when the contamination parameter is scaled by a factor arbitrarily close to $1$.
\end{corollary}

\section{Comparing private and robust hypothesis testing}\label{sec: privacy}

\def\priv{\mathrm{privacy}}
\def\rob{\mathrm{robust}}
\def\nstar{n^\star}
\def\privp{\gamma}
\def\estar{\privp^\star}
\def\npriv{\nstar_{\priv}}
\def\nrob{\nstar_{\rob}}
\def\NT{N_{\mathrm{Transformation}}}

In this section, we reconcile our results on the sample complexity curve as a function of the contamination fraction $\eps$ with the corresponding sample complexity curve as a function of the differential privacy parameter $\privp$. We refer the reader to \cite{DworkRoth13} for background on differential privacy.

\begin{definition}[Differential privacy] 
Let $\cX$ be the domain and let $\cY$ be the output space.
A randomized algorithm $\cA$: $\cX^n \to \cY$ is $\privp$-differentially private ($\privp$-DP for short) if, for all measurable events $E \subseteq \cY$ and all datasets $S, S' \in \cX^n$ that differ in at most one observation, 
 $\Pr[\cA(S) \in E] \leq e^{\privp}\Pr[\cA(S') \in E]$.
\end{definition}

\begin{definition}[Private sample complexity]
Let $\privp > 0$ and $p$ and $q$ be two distributions. The private sample complexity under $\privp$-central differential privacy, denoted by $\npriv(\privp):= \npriv(p,q,\privp)$, is defined as 
\begin{align*}
\npriv(\privp):= \npriv(p,q,\privp) := \min\left\{n: e^*_{n, \priv}(p,q,\privp) \leq 1/10  \right\}\,,
\end{align*}
where  $e_{n, \priv}^*(p,q,\privp) := \inf \{e_n(\phi;p,q) : \phi \text{ is $\privp$-DP}\}.$
\end{definition}
Using the sample-and-aggregate technique in differential privacy \cite{NissRS07}, one can show that for $\privp \gtrsim 1$, $\npriv(\privp) \asymp \frac{1}{\dhel^2(p,q)}$.
Furthermore, amplification by subsampling (see, for example, \cite{Steinke22}) shows that,  for any $\privp_1 < \privp_2 < O(1)$:
\begin{align*}
    \npriv(\privp_2) \leq \npriv(\privp_1) \lesssim \frac{\privp_2}{\privp_1} \npriv(\privp_2)\,.
\end{align*}
Thus, even without knowing the exact form of the curve $\npriv$, we can deduce that it is relatively stable in the following sense: changing the privacy parameter by a constant factor changes the corresponding sample complexity by at most a constant factor.
On the other hand, \cite{DwoLei09,AsiUZ23,HopKMN23} have shown that, for binary hypothesis testing, robustness and privacy are essentially equivalent in a certain sense (see also \cite{AsiEtal24,CanHLLN23}). 
An incorrect reading of these two observations might suggest that increasing the contamination fraction by a constant factor should lead only to a constant-factor increase in robust sample complexity, contradicting \Cref{thm:poly-jumps}.
In this section, we show that this apparent paradox arises from an incorrect translation of the results of \cite{DwoLei09,AsiUZ23,HopKMN23}. In fact, combining these works with an explicit example from \cite{CanEtal19} already yields an example in which changing the contamination fraction from $\eps$ to $C\eps$ for a sufficiently large constant $C$ leads to a polynomial blow-up in sample complexity.
However, the current transformations do not seem to be strong enough to deduce polynomial blowups for $o(\eps)$ perturbations to $\eps$, established in the previous sections.

\cite{CanEtal19} studied the sample-complexity curve $\npriv(\privp)$, along with the associated optimal private tests, and gave an expression for computing this curve. 
The resulting expression is closely related to least favorable distributions. 
In recent work, \cite{AsiEtal24} gave a simpler expression for the sample complexity, which we record below.
For two discrete distributions $p$ and $q$, let $D_\privp(p,q) := \sum_{i} q_i f_\gamma\left(\frac{p_i}{q_i}\right)$ be the (non-negative) expression, where $f_\gamma(t):= (t-1)\left[\log t\right]^{\privp}_{-\privp} = (t-1)\max( \min(\log t,\privp) , -\privp)$.

\begin{fact}[Private sample complexity~\cite{AsiEtal24}]
\label{fact:priv-sample-complexity}
Let $\privp > 0 $. Then 
       $\npriv(p,q,\privp) \asymp \frac{1}{\dhel^2(p,q)}  + \frac{1}{D_\privp(p,q)}$.
Furthermore, for $\privp = O(1)$, the private sample complexity satisfies that
\begin{align*}
  \frac{1}{\dhel^2(p,q)} + \frac{1}{\privp \dtv(p,q)}  \lesssim    \npriv(\privp) \lesssim \min\left(\frac{1}{\privp \dhel^2(p,q)}, \frac{1}{(\privp\dtv(p,q) - \privp^2)_+}\right)\,.
\end{align*}
In particular, if $\privp \ll \dtv(p,q)$, then $\npriv(\privp) \asymp \frac{1}{\privp \dtv(p,q)}$.
\end{fact}
However, as we will see, the connections between robustness and privacy place us in the regime $\privp \gtrsim \dtv(p,q)$, where the sample-complexity curve is more delicate. 
The relevant object is not the curve $\npriv(\privp)$ itself, but rather its inverse $\estar(n)$, defined below, and in particular the curve $n \cdot \estar(n)$.

\subsection{Connection between privacy and robustness}
{In this section, \(\eta\)-robustness refers to robustness under the strong contamination model, i.e. the adaptive-TV model from Section~\ref{sec: adaptive}. Accordingly, \(\nrob(\eta)\) denotes the sample complexity \(\nsATV(\eta)\).} We shall treat $p$ and $q$ as implicit throughout, but the underlying constants would be independent of $p$ and $q$. 

We begin by mentioning the transformations between privacy and robustness that have been established in \cite{DwoLei09,AsiUZ23,HopKMN23}.\footnote{In the following, the constant factor blow up in the sample complexity comes from boosting the success probability to accommodate the additional failure  probability of these transformations.}
\begin{fact}
\label{fact:priv-rob-transformations}
There exist constants $0 < C_2< C_1$, $\eps_0 \in (0,1/2)$, and $C_3 > 0$ such that the following holds.
Let $\cA_1$ be an $\eps$-robust algorithm for testing $p$ versus $q$ that uses $n_1$ samples, where $\eps < \eps_0$.
Let $\cA_2$ be an $\privp$-DP algorithm for testing between $p$ and $q$ that uses $n_2$ samples, where $\privp>0$.
Then:
\begin{enumerate}
    \item (Private to Robust: Group privacy \cite{DwoLei09}) There exists an algorithm $\cA_2'$ that is $\eps'$-robust with $\eps' = \frac{C_2}{n_2 \privp}$ and uses $O(n_2)$ samples.
    \item (Robust to Private: Inverse Sensitivity Mechanism~\cite{AsiUZ23,HopKMN23}) There exists an algorithm $\cA_1'$ that is $\privp'$-DP with $\privp' = \frac{C_1}{n_1\eps}$ and uses $C_3n_1$ samples.
\end{enumerate}
\end{fact}

To describe these results compactly, we introduce the function $\estar:\mathbb N \to \mathbb R$, an information-theoretic quantity that captures the best privacy budget achievable with $n$ samples:\begin{align*}
    \estar(n) := \min\left\{\privp: \npriv(\privp) \leq n\right\}\,.
\end{align*}
We set $\estar(n)=\infty$ if no such $\privp$ exists. Observe that $\estar$ is nonincreasing in $n$. As a simple consequence of \Cref{fact:priv-sample-complexity}, we obtain
\begin{align}
\label{eq:estar-condition}
\estar(n)
\text{ satisfies }
\begin{cases}
    \estar(n)=\infty,
    & \text{if } n\ll \frac{1}{\dhel^2(p,q)},\\[4pt]
    \frac{1}{n\dtv(p,q)}
    \lesssim \estar(n)
    \lesssim
    \frac{1}{n\dhel^2(p,q)},
    & \text{if } \frac{1}{\dhel^2(p,q)}
    \ll n \ll
    \frac{1}{\dtv^2(p,q)},\\[8pt]
    \estar(n)\asymp \frac{1}{n\dtv(p,q)},
    & \text{if } n\gg \frac{1}{\dtv^2(p,q)}.
\end{cases}
\end{align}

That is, outside the range of $\frac{1}{\dhel^2(p,q)}
    \ll n \ll
    \frac{1}{\dtv^2(p,q)}$, the function $\estar(n)$ has a simple description. Within this range, however, the behavior of $\estar$ can be much more complicated. 
In particular, 
\begin{align*}
    n \estar(n) \text{ satisfies } \begin{cases}
        \infty & \text{if } n \ll \frac{1}{\dhel^2(p,q)} \\
         \frac{1}{ \dtv(p,q)} \lesssim n \estar(n) \lesssim \frac{1}{ \dhel^2(p,q)}   & \text{if }  \frac{1}{\dhel^2(p,q)} \ll n \ll \frac{1}{\dtv^2(p,q)} \\
        \asymp \frac{1}{ \dtv(p,q)}  & \text{if } n \gg \frac{1}{\dtv^2(p,q)}
    \end{cases}\,.
\end{align*}

We now define a quantity $\NT: (0,1) \to \mathbb N$ that captures the sample complexity of the transformations described in \Cref{fact:priv-rob-transformations}. Define the nondecreasing function
\begin{align}
    \NT(\eta) := \min \left\{n  : n\cdot \estar(n) \leq \frac{1}{\eta}\right\},
\end{align}
with $\NT(\eta)=\infty$ when no such $n$ exists. As \Cref{eq:estar-condition} shows, for $\eta \gg \dtv(p,q)$, $\NT(\eta)$ is infinite.

The following simple corollary of \Cref{fact:priv-rob-transformations} shows that $\NT$ captures the robust sample complexity.
\begin{corollary}
\label{cor:robust-private-sample-complexity}
Let $\eta \ll \dtv(p,q)$. Then
\begin{align*}
\NT(\eta/C_1C_3) \lesssim    \nrob(\eta) \lesssim \NT(\eta/C_2)\,.
\end{align*}
\end{corollary}

However, the function $\NT(\eta)$ can exhibit large jumps under constant-factor changes in $\eta$. This, in turn, implies that $\nrob(\eta)$ can also exhibit large jumps under constant-factor changes in $\eta$. Recall that irregular behavior in $\estar(n)$ can occur in the range $n \in \left( \frac{1}{\dhel^2(p,q)}, \frac{1}{\dtv^2(p,q)}\right)$,
which is precisely the regime of interest for robustness. See \Cref{sec:private-example} for a concrete example.
\begin{proof}
    By part (i) of \Cref{fact:priv-rob-transformations}, any $\frac{C_2}{\eta n_2}$-private algorithm yields an $\eta$-robust algorithm using $O(n_2)$ samples. Equivalently, the robust sample complexity is $O(n)$ whenever
    $\estar(n) \leq \frac{C_2}{\eta n}$.
Therefore,
\begin{align*}
    \nrob(\eta) \lesssim \min\left\{n: \estar(n) \leq \frac{C_2}{\eta n} \right\} = \NT(\eta/C_2)\,.
\end{align*}

On the other hand, part (ii) of \Cref{fact:priv-rob-transformations} implies that an $\eta$-robust $n$-sample algorithm gives a $\frac{C_1}{\eta n}$-DP algorithm using $C_3n$ samples. In particular,
\begin{align}
    \nrob(\eta) \leq n
    \implies
    \estar(C_3 n) \leq \frac{C_1}{\eta n}
    =
    \frac{C_1 C_3}{\eta C_3 n}
    \implies
    \NT(\eta/C_1C_3) \leq C_3n.
\end{align}
Minimizing over $n$, we obtain
    $\nrob(\eta) \gtrsim \NT(\eta/C_1C_3)$.
\end{proof}

\subsection{Example: Polynomial jumps in robust sample complexity using privacy} 
\label{sec:private-example}
We now show how to deduce polynomial jumps in robust sample complexity from private sample-complexity curves, using an example from \cite{CanEtal19}. Section 1.3 of \cite{CanEtal19} gives an explicit example with $p = (0,0.5,0.5)$ and $q= (2 \alpha^{3/2}, 0.5 + \alpha - \alpha^{3/2}, 0.5 - \alpha - \alpha^{3/2}$). For this pair, we have $\dhel^2 = \alpha^{3/2}$ and $\dtv = \alpha$, and 
\begin{align}
    \npriv(\privp) \asymp \begin{cases}
        \frac{1}{ \privp \dtv} & \qquad \text{if }\privp \in (0,\dtv)\\
        \frac{1}{\dtv^2} & \qquad \text{if }\privp \in (\dtv, \dtv^2/\dhel^2)\\
        \frac{1}{\min(\privp,1)  \dhel^2} & \qquad \text{if }\privp \in (\dtv^2/\dhel^2,1)
    \end{cases}\,.
\end{align}
The middle flat region is what changes the behavior of $\estar(n)$. Inverting this function, we find that, for
 $n_0 \asymp \frac{1}{\dtv^2}$, 
 we have
\begin{align*}
    \estar(n)  \text { satisfies } \begin{cases}
     \asymp  \frac{1}{n \dhel^2} & \qquad \text{if }  \frac{1}{\dhel^2} \ll n \ll  n_0 \\
   \frac{1}{n\dtv} \lesssim \estar(n) \lesssim \frac{1}{n\dhel^2}  &   \qquad \text{if }   n \asymp n_0  \\
      \asymp \frac{1}{n \dtv}   & \qquad \text{if } n \gg n_0 
    \end{cases}\,.
\end{align*}
Thus, $\estar(n)$ has a sharp transition around $n_0$. 
In particular, 
\begin{align*}
    n \estar(n)  \text { satisfies } \begin{cases}
     \asymp  \frac{1}{\dhel^2} & \qquad \text{if }  \frac{1}{\dhel^2} \ll n \ll  n_0 \\
   \frac{1}{\dtv} \lesssim n \estar(n) \lesssim \frac{1}{\dhel^2}  &   \qquad \text{if }   n \asymp n_0  \\
      \asymp \frac{1}{ \dtv}   & \qquad \text{if } n \gg n_0 
    \end{cases}\,.
\end{align*}
We can use this to compute the following behavior of  $\NT(\eta)$:
\begin{align}
    \label{eq:robust-sample-complexity}
    \NT(\eta)  \text{ satisfies } \begin{cases}
        \asymp \frac{1}{\dhel^2} & \text{ if } \eta \lesssim \dhel^2 \\
       \asymp \frac{1}{\dtv^2}  & \text{ if }  \dhel^2 \ll \eta \lesssim \dtv \\
       \infty & \text{ if } \eta \gtrsim \dtv
    \end{cases}\,,
\end{align}
where we use that $n_0 \asymp \frac{1}{\dtv^2}$.
Since $\frac{1}{\dtv^2}$ is polynomially larger than $\frac{1}{\dhel^2}$ for this particular example, the function $\NT(\cdot)$ exhibits a polynomial jump at a point $\eta_0 \asymp \dhel^2$ when its argument $\eta$ is changed by a constant factor. By \Cref{cor:robust-private-sample-complexity}, there must exist an $\eta$ such that $\nrob(C\eta)$ is polynomially larger than $\nrob(\eta)$.
Furthermore, it is easy to check that the resulting expression in \Cref{eq:robust-sample-complexity} is tight by computing $\nrob(\eta)$: it can be seen that $\nrob(\eta) \asymp \frac{1}{\dhel^2}$ if $\eta \ll \dhel^2$ and $\nrob(\eta) \asymp \frac{1}{\dtv^2}$ for $\dhel^2 \lesssim \eta \lesssim \dtv$.

\section{Conclusion}

In this paper we studied the sample complexity of robust hypothesis testing under three natural models of contamination and their adaptive counterparts. We showed that, just as in Huber and TV contamination, the subtractive contamination setting also admits least favourable distributions, and we found explicit formulas for them. When least favourable distributions exist, the problem of analysing the sample complexity is equivalent to analysing the Hellinger divergence between the least favourable distributions. In all three models, we showed that this quantity is highly unstable in the contamination parameter $\eps$. Hence, small changes in $\eps$ can cause polynomial jumps in the sample complexity. Similarly, a small mismatch in the assumed and true contamination values may lead to polynomial jumps in the sample complexity or even a complete breakdown of the test. The instability of the sample complexity suggests that there is no easy formula for the sample complexity of robust hypothesis testing. It also rules out any robustness-amplification procedures that amplify robustness by taking at most a constant factor more samples. This is in contrast to privacy, where simple procedures like subsampling can boost the privacy parameter by taking a constant factor more samples. 

Despite the almost pathological behaviour of the sample complexity with respect to $\eps$ in each model, we show the surprising property that up to scaling of $\eps$ by universal constants, the sample complexities of the three models are comparable. We note that there are models in the literature for which such a comparisons do not hold. For instance, the uncertainty set for the \emph{$\eps$-realisable contamination model} from~\cite{MaEtal24} (when the probability of deletion is at most a constant) may be shown to be contained within the $C_0\eps$-Huber and $C_0\eps$-subtractive uncertainty sets for a constant $C_0$, that is the set
\begin{align*}
\cP_{\mathrm{HS}}(p, \eps) = \{p' \in \Delta(\cX) : (1-C_0\eps)p(x) \leq p'(x) \leq (1+C_0\eps)p(x) \text{ for all } x \in \cX\}.
\end{align*}
This model of contamination turns out to be too restrictive, and in particular, we cannot upper bound $\nsTV(\eps)$ by $n^*_{\mathrm{HS}}(C\eps)$ no matter how large a $C$ is chosen. To see this, consider the same example in Section~\ref{sec: example}. Here, the impact of $\{3\}$ simply cannot be cancelled by this type of contamination as $p'(3) \geq 2\eps(1-C\eps) \asymp \eps$, whereas $q'(3) = 0$. Thus, sample complexity is $O(1/\eps)$ no matter how large a $C$ is chosen. However, the sample complexity with $\eps$-TV contamination is $\Theta(1/\eps^2)$.

Our work leaves open several interesting research directions. We showed that in Huber and TV models, there may be sample complexity jumps between $\eps$ and $\eps + \Omega(\eps^{3/2})$, but not between $\eps$ and $\eps + O(\eps^2)$. It is unclear if $\eps^{3/2}$ is actually the threshold that determines whether or not sample complexity jumps occur. In a similar vein, we showed that underestimating the true contamination $\eps$ by $\Omega(\eps^{4/3})$ for TV contamination, $\Omega(\eps^{3/2})$ for Huber contamination, and $\Omega(\eps^2)$ for subtractive contamination may lead to a breakdown of the likelihood ratio test. It would be interesting to analyse if these thresholds could be improved or if they hold more generally. We did not analyse the sample complexity of robust testing when $p$ and $q$ have different levels of contamination, say $\eps_1$ for $p$ and $\eps_2$ for $q$. Part of the challenge here is that explicit formulas for LFDs are not known for TV contamination when $p$ and $q$ have different levels of contamination. However, the questions studied in this paper continue to be interesting and non-trivial for such settings as well. Our results on sample complexity comparisons across models raise the question whether there are more contamination models for which such comparisons hold beyond the three considered in this paper. Generalising further, it would also be interesting to explore whether the phenomenon of comparability of different contamination models after scaling $\eps$ holds for more general testing problems in robust statistics, and whether there are problems where the instability of the sample complexity is even more pronounced leading to super-quadratic jumps for small changes in contamination. 

\section*{Acknowledgements}

Part of this work was done when Shankar Vallinayagam was an intern as a part of the Summer Research in Maths (SRIM) program at the University of Cambridge. Shankar is grateful to Trinity College, Cambridge, for supporting his SRIM internship. 

Part of this work was completed while Ankit Pensia was supported by the Research Pod on Resilience in Brain, Natural, and Algorithmic Systems at the Simons Institute for the Theory of Computing, UC Berkeley. Ankit is also grateful to Tata Institute of Fundamental Research at Mumbai, where another part of this work was carried out during his visit.

Varun Jog thanks the Simons Institute for the Theory of Computing at Berkeley and the Simons--Laufer Mathematical Sciences Institute at Berkeley, where a part of this work was done during his visits in Fall 2024 and Spring 2025, respectively. Varun also gratefully acknowledges support from the Leverhulme Trust through the grant RPG-2025-226 titled ``Old Problems, New Perspectives: A Fresh Look at Classical Hypothesis Testing.'' 

We thank Alston Xu for his work on a version of the example in Section~\ref{sec: example} as a part of his SRIM internship. We also thank Aditya Dhawan for helpful discussions in the early part of this work.

\appendix

\section{Proofs for Section~\ref{sec: pseudo_lfd}}

\subsection{Proof of Lemma~\ref{lemma: exists_lfd}}\label{app: exists_lfd}

The theory of \citet{HubStra73} shows that LFDs exist when the uncertainty set can be associated with a \emph{two-alternating capacity}. A two-alternating capacity $v$ is a function on measurable sets that satisfies: (i) $v(\emptyset) = 0$, $v(\cX) = 1$, (ii) if $A \subseteq B$, then $v(A) \leq v(B)$, (iii) $A_n \uparrow A$ implies $v(A_n) \uparrow v(A)$, and (iv) for any measurable sets $A$ and $B$, 
\begin{align*}
v(A \cup B) + v(A \cap B) \leq v(A) + v(B).
\end{align*}
Given a capacity $v$, the uncertainty set is defined as $\cP_v = \{P : P(A) \leq v(A) \text{ for all measurable } A\}$. 

\citet{HubStra73} show that for Huber contamination, the associated capacity is $v(A) = (1-\eps)p(A) + \eps$, and for TV contamination it is $v(A) = \min\{1, p(A)+\eps\}$. For subtractive contamination, it is natural to consider 
\begin{align*}
v(A) = \sup_{P \in \cP_\Sub(p, \eps)} P(A) = \min\{(1+\eps)p(A), 1\}.
\end{align*}
It is easy to see that $\cP_v$ is exactly $\cP_\Sub(p, \eps)$. Moreover, $v$ satisfies the (i), (ii), and (iii) in a straightforward manner. If $v$ also satisfies (iv), then this would imply the existence of LFDs. 

We now argue that (iv) also holds. Suppose $p(A) = x$, $p(B) = y$, and $p(A\cap B) = z$. Let $g(t) = \min\{(1+\eps)t, 1\}$, which is a concave function. We need to show that
\begin{align*}
g(x+y-z) + g(z) \leq g(x) + g(y).
\end{align*}
Since $g$ is concave, we have for $u \leq v$ and $t \geq 0$,
\begin{align*}
g(v+t) - g(u+t) \leq g(v) - g(u).
\end{align*}
Substituting $u = z$, $v = x$, $t = y - z$, we have
\begin{align*}
g(x+y-z) - g(y) \leq g(x) - g(z),
\end{align*}
which, upon rearranging, gives the desired inequality.

\section{Section~\ref{sec: example} proofs}

\subsection{Proof of Proposition~\ref{prop:baseline}}\label{app: baseline}
    The lower bound is proved for the uncontaminated setting ($\eps = 0$). Note that the sum of the Type-I and Type-II errors with $n$ samples is simply $1 - \dtv(p^{\otimes n}, q^ {\otimes n})$. So we have
    \begin{align*}
        \frac{9}{10} \leq \dtv(p^{\otimes n}, q^{\otimes n}) \leq n\dtv(p, q),
     \end{align*}
     which immediately gives $n \gtrsim \frac{1}{\dtv(p,q)}.$
     To prove the upper bound, we show that Scheffe's test works for all models with a sample complexity of at most $\lesssim \frac{1}{\dtv^2(p,q)}$. Let $A \coloneqq \{i \,:\, p(i) \ge q(i)\}$. Scheffe's test computes the test statistic
     \begin{align*}
         S = \frac{1}{n} \sum_{i=1}^n \mathbbm{1}\{x_i \in A\} \, ,
     \end{align*}
     and declares $\cP$ if $S \ge \frac{p(A)+q(A)}{2}$ and $\cQ$ otherwise. Observe that in the oblivious TV, Huber, and subtractive contamination models, for any contaminated $p'$ and $q'$ we must have $\E_{p'} S = p'(A) \geq p(A) -\eps$ and $\E_{q'} S = q'(A) \leq q(A) + \eps$. In particular, the means are separated by at least $p(A)- q(A) - 2\eps \asymp \dtv(p,q)$, as $\eps \leq \dtv(p,q)/4$. A simple application of Hoeffding's inequality shows that the clean test statistics concentrates around its mean:
     \begin{align}
         \prob_{p'}\left(S - p'(A) \leq -\frac{\dtv(p,q)}{4}\right) \le e^{-\frac{n\dtv(p,q)^2}{8}},
     \end{align}
    and so choosing $n = \frac{8\log 20}{\dtv^2(p,q)}$ ensures that the Type-I error is at most $1/20$. A similar calculation may be done to bound the Type-II error by $1/20$ as well, showing that the sample complexity is upper bounded by $\lesssim \frac{1}{\dtv^2(p,q)}$.

In fact, the proof largely continues to work even with adaptive contamination. For adaptive-TV, observe that $S$ may change by at most $\eps $ after the adversary corrupts the dataset. To bound the impact of contamination on Scheffe's test, we may as well assume the worst case scenario where the clean test statistic is shifted by $\dtv(p,q)/4$. Under $\cP$, Scheffe's test will make an error if the clean test statistic lies more than $\frac{\dtv(p,q)}{4}$ below from its mean of $p(A)$, since the additional perturbation by $\dtv(p,q)/4$ will push it to the other side of the decision threshold. A similar Hoeffding bound shows that for $n \gtrsim 1/\dtv^2(p,q)$, the errors remain bounded by constants. Since adaptive-TV is stronger than adaptive-Huber or adaptive-Subtractive contamination, the upper bound extends to these cases as well.

\subsection{Proof of Proposition~\ref{prop:small_eps}}\label{app: small_eps}
We prove the upper bound $\nsTV(\eps) \lesssim \frac{1}{\dhel^2(p,q)}$. Lemma~\ref{lemma: containment} implies that $\max\{ \nsHub(\eps), \nsSub(\eps)\} \le \nsTV(\eps)$, so the same upper bound will hold for $\nsHub(\eps)$ and $\nsSub(\eps)$ as well. Note that the upper bound matches the lower bound obtained without any contamination, giving a tight characterisation of the sample complexity. 

The sample complexity with $\eps$-TV contamination is characterised by the Hellinger divergence between the least favourable distributions. Consider any $p'$ such that $\dtv(p',p) \le \eps$ and any $q'$ such that $\dtv(q',q) \le \eps$. We will show that when $\eps$ is small enough, the Hellinger divergence $\dhel^2(p', q')$ is $\dhel^2(p,q)$, up to constant factors. We note the inequality:
\[
\dhel^2(p,p') \le 2\dtv(p,p') \le \frac{2}{9}\dhel^2(p,q),
\]
hence \(\dhel(p,p')\le \frac{\sqrt 2}{3}\dhel(p,q)\), and similarly for \(q,q'\). Therefore
\[
\dhel(p',q')\ge \left(1-\frac{2\sqrt 2}{3}\right)\dhel(p,q).
\]
Since this inequality holds for any $p'$ and $q'$ in the respective TV-balls, it also holds for the least favourable distributions (LFD) in particular. Denoting the LFD-pair by $(\ptv(\eps), \qtv(\eps))$, we conclude 
\begin{align*}
\nsTV(\eps) \asymp \frac{1}{\dhel^2(\ptv(\eps), \qtv(\eps))}\lesssim \frac{1}{\dhel^2(p,q)}.
\end{align*}

\subsection{Proof of Lemma~\ref{lemma: tv_ball_at_p_eps}}\label{app: tv_ball_at_p_eps}

We explicitly construct $p_{\eps_0}$ in all three cases. For TV contamination, choose
\begin{align*}
p_{\eps_0} = \left(1 - \frac{\eps_0}{\eps}\right) p + \frac{\eps_0}{\eps} p_\eps.
\end{align*}
Observe that
\begin{align*}
\dtv(p, p_{\eps_0}) = \frac{\eps_0}{\eps}\dtv(p_\eps, p) \leq \eps_0,
\end{align*}
and so $p_{\eps_0}$ lies in $\cP_{\eps_0}$. It is also easy to check that
\begin{align*}
\dtv(p_\eps, p_{\eps_0}) = \left(1 - \frac{\eps_0}{\eps}\right)\dtv(p, p_\eps) \le \eps-\eps_0.
\end{align*}

For Huber contamination, suppose $p_\eps = (1-\eps)p + \eps h$. Consider $p_{\eps_0} = (1-\eps_0)p + \eps_0h$. Clearly, $p_{\eps_0} \in \cP_{\eps_0}$, and
\begin{align*}
\dtv(p_\eps, p_{\eps_0}) = (\eps-\eps_0) \dtv(p,h) \leq (\eps-\eps_0).
\end{align*}

For subtractive contamination, consider
\begin{align*}
p_{\eps_0} = p_\eps \frac{\eps_0}{\eps} + p\left( 1 - \frac{\eps_0}{\eps}\right).
\end{align*}
Observe that
\begin{align*}
\frac{p_{\eps_0}(i)}{p(i)} &\leq \frac{p_\eps(i)}{p(i)} \frac{\eps_0}{\eps} + \left( 1 - \frac{\eps_0}{\eps}\right)\\
&\leq (1+\eps) \frac{\eps_0}{\eps} + \left( 1 - \frac{\eps_0}{\eps}\right)\\
&= (1+\eps_0).
\end{align*}
Hence, $p_{\eps_0} \in \cP_{\eps_0}$. Furthermore,
\begin{align*}
\dtv(p_\eps, p_{\eps_0}) = \left( 1 - \frac{\eps_0}{\eps}\right)\dtv(p_\eps, p) \leq \eps - \eps_0.
\end{align*}
This concludes the proof.

\section{Section~\ref{sec: sandwich} proofs} \label{app: sandwich}

\subsection{No simulation lemma}\label{app: lemma: containment}

\begin{lemma}[No simulation lemma]\label{lemma: containment}
Let $p$ be a probability distribution and let $\eps > 0$. Let $\cP^{\TV}(\eps)$, $\cP^\Hub(\eps),$ and $\cP^\Sub(\eps)$ be the TV, Huber, and subtractive uncertainty sets at contamination $\eps$. Then the following results hold:
\begin{itemize}
\item[(i)] For any $C > 0$, there exist $p$ and $\eps>0$ such that $\cP^\TV(\eps) \not\subseteq \cP^{\Hub}(C\eps)$. Similarly, for any $C > 0$, there exist $p$ and $\eps>0$ such that $\cP^\TV(\eps) \not\subseteq \cP^{\Sub}(C\eps)$.
\item[(ii)] For any $c > 0$, there exist $p, q,$ and $\eps >0$ such that $\cP^\Hub(c\eps) \not\subseteq \cP^\Sub(\eps)$. For any $C > 0$, there exists a $p$ such that $\cP^\Sub(\eps) \not\subseteq \cP^\Hub(C\eps)$.
\item[(iii)]
For any $p$ and $\eps>0$, $\cP^\Hub(\eps) \subseteq \cP^\TV(\eps)$ and $\cP^\Sub(\eps) \subseteq \cP^\TV(\eps)$.
\item[(iv)] For any $C>0$, there exist $p, q,$ and $\eps>0$ such that LFD-pair $(p^*_\TV(\eps), q^*_\TV(\eps))$ under $\eps$-TV contamination satisfies $(p^*_\TV(\eps), q^*_\TV(\eps)) \notin \cP^\Hub(C\eps) \times \cP^\Hub(C\eps).$ Similarly, for any $C>1$, there exist $p, q,$ and $\eps>0$ such that $(p^*_\TV(\eps), q^*_\TV(\eps)) \notin \cP^\Sub(C\eps) \times \cP^\Sub(C\eps).$ 
\item[(v)] For any $C>0$, there exist $p, q,$ and $\eps>0$ such that LFD-pair $(p^*_\Hub(\eps), q^*_\Hub(\eps))$ under $\eps$-Huber contamination satisfies $(p^*_\Hub(\eps), q^*_\Hub(\eps)) \notin \cP^\Sub(C\eps) \times \cP^\Sub(C\eps).$ Similarly, for any $C>0$, there exist $p, q,$ and $\eps>0$ such that LFD-pair $(p^*_\Sub(\eps), q^*_\Sub(\eps))$ under $\eps$-subtractive contamination satisfies $(p^*_\Sub(\eps), q^*_\Sub(\eps)) \notin \cP^\Hub(C\eps) \times \cP^\Hub(C\eps).$
\end{itemize}
\end{lemma}

\begin{proof}
We prove each non-containment by an explicit two-point example. Consider the support $\{1, 2\}$.

\begin{itemize}
\item[(i)] Fix $C\ge 1$ without loss of generality. Choose \(0<\eps<1/C\), and let $p=(\eps,1-\eps)$ and $r=(0,1)$. 
Then $\dtv(p,r)=\eps,$ so \(r\in \cP^{\TV}(\eps)\). But \(r\notin \cP^{\Hub}(C\eps)\), because if it were, then we would have $r_1 \geq (1-C\eps)p_1 = (1-C\eps)\eps > 0$, which does not hold. Hence, 
$\cP^{\TV}(\eps)\not\subseteq \cP^{\Hub}(C\eps).$

For the subtractive setting,  let $p=(0,1)$ and $r=(\eps,1-\eps).$
Then $r\in \cP^{\TV}(\eps)$, but $r\notin \cP^{\Sub}(C\eps)$, since the subtractive adversary cannot assign non-zero probability to 
$\{1\}$.

\item[(ii)] Fix $c >0$, and assume without loss of generality that $c < 1$. Let $p=(1,0)$ and $r=(1-c\eps,c\eps)$. Then $r=(1-c\eps)p + c\eps(0,1)$ so $r\in \cP^{\Hub}(c\eps)$. But $r\notin \cP^{\Sub}(\eps)$ as $r_2 > 0$.  Hence $\cP^{\Hub}(c\eps)\not\subseteq \cP^{\Sub}(\eps).$

Finally, fix $C \geq 1$ without loss of generality. Choose \(0<\eps<1/C\), and let $p=(a,1-a)$ and $r=(0,1)$ where $a>0$ is small enough that $1 \le (1+\eps)(1-a)$, for instance any $a\le \eps/(1+\eps)$. Then $r\in \cP^{\Sub}(\eps)$, since
\begin{align*}
r_1 &=0\le (1+\eps)a,\\
r_2&=1 \le (1+\eps)(1-a).
\end{align*}
But $r\notin \cP^{\Hub}(C\eps)$, because any element $u\in \cP^{\Hub}(C\eps)$
must satisfy
\begin{align*}
u_1 \ge (1-C\eps)p_1 = (1-C\eps)a >0,
\end{align*}
again since $C\eps<1$, whereas $r_1=0$. Therefore $\cP^{\Sub}(\eps)\not\subseteq \cP^{\Hub}(C\eps).$

\item[(iii)] We prove the two containments separately.

First, let $q \in \cP^\Hub(\eps)$. Then by definition,
$q = (1-\eps)p + \eps r$ for some probability distribution $r$. Hence
\begin{align*}
\dtv(q,p)
&= \dtv\big((1-\eps)p+\eps r,\; p\big) \\
&= \dtv\big(\eps(r-p),0\big) \\
&= \eps \dtv(r,p) \\
&\le \eps.
\end{align*}
Thus $q \in \cP^\TV(\eps)$, so $\cP^\Hub(\eps)\subseteq \cP^\TV(\eps)$.

Next, let $q \in \cP^\Sub(\eps)$. Then $q(i)\le (1+\eps)p(i)$ for every $i$. Therefore
\begin{align*}
\dtv(q,p)
&= \sum_{i:\, q(i)\ge p(i)} \bigl(q(i)-p(i)\bigr) \\
&\le \sum_{i:\, q(i)\ge p(i)} \eps\, p(i) \\
&\le \eps \sum_i p(i) \\
&= \eps.
\end{align*}
Thus $q \in \cP^\TV(\eps)$, so $\cP^\Sub(\eps)\subseteq \cP^\TV(\eps)$.
\end{itemize}
This proves all the claimed non-containments.

\begin{itemize}
\item[(iv)]
Let $C >1$ without loss of generality. Consider $p = (0,1)$ and $q = (10\eps, 1-10\eps)$. The LFDs are easy to compute without resorting to the formulas, as the obvious choices within the uncertainty sets are those that increase $p_1$ as much as possible and decrease $q_1$ as much as possible to make the two Bernoullis more similar. In particular, it is straightforward to check that:
\begin{align*}
p^*_\TV(\eps) = (\eps, 1-\eps), \quad q^*_\TV(\eps) = (9\eps, 1-9\eps).
\end{align*}
However, for all distributions $u$ in the $C\eps$-Huber uncertainty set around $q$, we require 
\begin{align*}
u_1 \geq 10\eps(1-C\eps) = 10\eps - 10C\eps^2 \gg 9\eps.
\end{align*}
Hence, the $\eps$-TV-LFD pair cannot be simulated via $C\eps$-Huber contamination. It also cannot be simulated with $C\eps$-subtractive contamination, as for any $v$ in the $C\eps$-subtractive uncertainty set around $p$, we require $v_1 = 0$, which is not the case for the TV-LFD.

\item[(v)]
For the same example as in part (iv), we can check that the LFD-pair for $\eps$-Huber contamination is $p^*_\Hub(\eps) = (\eps, 1-\eps)$ and $q^*_\Hub(\eps) = (10\eps(1-\eps), 1-10\eps+10\eps^2)$. Observe that no matter how large $C$ is, any $u \in \cP^\Sub(C\eps)$ around $p$ will have $u_1=0$, which is not the case here. 

Continuing with the same example, LFDs for $\eps$-subtractive contamination may be checked $p^*_\Sub(\eps) = (0,1)$ and $q^*_\Sub(\eps) = (9\eps + 10\eps^2, (1-10\eps)(1+\eps))$. Observe that $C\eps$-Huber contamination cannot change $q$ by more than $O(\eps^2)$, since we require that  any $u$ in the $C\eps$-Huber uncertainty set around $q$ satisfy $u_1 \geq 10\eps(1-C\eps) \gg 9\eps + 10\eps^2$ for all small enough $\eps$. 
\end{itemize}
\end{proof}
\subsection{Proof of Lemma~\ref{lemma: hellinger-comparability}}\label{app: hellinger-comparability}

\begin{proof}[Proof of Lemma~\ref{lemma: hellinger-comparability}]
Let $L, M,$ and $H$ be a partition $\cX$ based on whether the likelihood lies in $(-\infty, \cl)$, $[\cl, \cu]$, or $(\cu, \infty)$, respectively. On $L$ both LFDs have ratio $\cl$, on $M$ the ratio is $r(i) = p(i)/q(i)$, and on $H$ the ratio is $\cu$. We'll use the expression that for any dummy measures $p$ and $q$,
\begin{align*}
\dhel^2(p,q) = \sum_{i \in \cX} \left(\sqrt{\frac{p(i)}{q(i)}} - 1\right)^2 q(i)
\end{align*}
replacing $(p,q)$ with $(\ptv, \qtv)$ and $(\phub, \qhub)$. Since the likelihood ratios satisfy 
\begin{align*}
\frac{\ptv(i)}{\qtv(i)} = \frac{\phub(i)}{\qhub(i)} \quad \text{ for all} \quad i \in \cX,
\end{align*}
it will be enough to prove $\qtv(i) \asymp \qhub(i)$ for all $i \in \cX$. We prove this separately for $i \in L$, $i \in M$, and $i \in H$.

Observe that for $i \in L$,
\begin{align*}
\frac{q(i)}{1+\cl} \le \frac{p(i)+q(i)}{1+\cl} = \qtv(i) \le q(i).
\end{align*}
where the last inequality follows since on $L$ we have $p(i) \leq \cl q(i)$. The above inequality, combined with $\qhub(i) = (1-\eps)q(i)$ gives
\begin{align*}
(1-\eps)\le \frac{\qhub(i)}{\qtv(i)} \le (1-\eps)(1+\cl)\in[\tfrac12,2] \tag{$\cl \in [0,1)$ and $\eps \leq 1/2$}
\end{align*}
Stated simply, on $L$ we have $\qhub(i) \asymp \qtv(i)$.

Moving on to $M$, observe that for $i \in M$,  $\qtv(i) = q(i)$ and $\qhub(i) = (1-\eps)q(i)$. Since $\eps \leq 1/2$, we immediately get $\qhub(i) \asymp \qtv(i)$, with the ratio lying in $[1/2,1]$.

Finally for $i \in H$, $\qtv(i) =(p(i) + q(i))/(1+\cu)$ and $\qhub(i) =(1-\eps)p(i)/\cu$. We have that
\begin{align*}
\frac{p(i)}{1+\cu}\le \frac{p(i) + q(i)}{1+\cu} = \qtv(i) \le \frac{p(i)}{\cu}.
\end{align*}
Hence,
\begin{align*}
(1-\eps)\le \frac{\qhub(i)}{\qtv(i)}\le (1-\eps)\frac{1+\cu}{\cu}\in [\tfrac12,2]. \tag{$\cu > 1$ and $\eps \leq 1/2$}
\end{align*}

Thus, we see that for all $i$, we have the relation $\qtv(i) \asymp \qhub(i)$. This concludes the proof.
\end{proof}
\subsection{Proof of Lemma~\ref{lemma: threshold-order}}\label{app: threshold-order} 

\begin{proof}[Proof of Lemma~\ref{lemma: threshold-order}]
For any clips $\cl < 1 < \cu$, define the following one-sided radii motivated by the clips formulas~\eqref{eq: clips_hub} and~\eqref{eq: clips_tv}:
\begin{align*}
\epsTVup{\cu} &= \frac{1}{1+\cu}\left(p(H)-\cu q(H)\right),\\
\epsTVlow{\cl}&= \frac{1}{1+\cl}\left(\cl q(L)-p(L)\right),\\
\epsHubup{\cu}&= 1-\frac{1}{1-q(H)+ p(H)/\cu},\\
\epsHublow{\cl}&= 1-\frac{1}{1-p(L)+\cl q(L)}.
\end{align*}
We can check the following identities by direct computation:
\begin{equation}\label{eq: tv-up-id}
\epsTVup{\cu}
= \frac{\cu}{1+\cu}\cdot \frac{\epsHubup{\cu}}{1-\epsHubup{\cu}},
\qquad
\epsTVlow{\cl}
= \frac{1}{1+\cl}\cdot \frac{\epsHublow{\cl}}{1-\epsHublow{\cl}}.
\end{equation}

At the clips $(\clHub(\eps),\cuHub(\eps))$, the Huber one-sided radii equal exactly $\eps$---indeed, that is how the clips are defined via \eqref{eq: clips_hub}. That is,
\begin{align*}
\eps = \epsHubup{\cuHub(\eps)}=\epsHublow{\clHub(\eps)}
\end{align*} 
Using the identity~\eqref{eq: tv-up-id} and the inequalities $\clHub(\eps) < 1 < \cuHub(\eps)$, we have
\begin{align}
\epsTVup{\cuHub(\eps)}
=\frac{\cuHub(\eps)}{1+\cuHub(\eps)}\cdot\frac{\eps}{1-\eps}
\ \ge\ \tfrac12\,\eps,\qquad
\epsTVlow{\clHub(\eps)}
=\frac{1}{1+\clHub(\eps)}\cdot\frac{\eps}{1-\eps}
\ \ge\ \tfrac12\,\eps. \label{eq: tv_radii}
\end{align}
Now observe that the one-sided TV radii are monotonic in the clips: $\epsTVup{\cu}$ reduces as $\cu$ increases, and $\epsTVlow{\cl}$ reduces as $\cl$ decreases. This is easily seen by rewriting
\begin{align*}
    \epsTVup{\cu} &= \frac{\sum_i \big(p(i) - \cu q(i)\big)_+}{1+\cu}, \quad \text{ and }\\
    \epsTVlow{\cl} &= \frac{\sum_i \Big(q(i) - \frac{p(i)}{\cl}\Big)_+}{1+ \frac{1}{\cl}}.
\end{align*}
At the clips $(\clTV(\eps/2), \cuTV(\eps/2))$, the TV-radii are exactly $\eps/2$. Inequalities~\eqref{eq: tv_radii} along with monotonicity of the TV-radii imply that we must have $\cuHub(\eps) \le \cuTV(\eps/2)$ and $\clTV(\eps/2)\le \clHub(\eps)$. This concludes the proof.
\end{proof}

\subsection{Proof of Lemma~\ref{lemma: tv-hellinger-monotone}}\label{app: tv-hellinger-monotone}
\begin{proof}
Write $r(i)=p(i)/q(i)$ and fix an index $i$. Using the LFD-formulas~\eqref{eq: lfd_tv}, the per--index contribution to $\dhel^2$ is
\begin{align*}
\phi_{i}(\cl,\cu)\;=\;
\begin{cases}
\dfrac{(\sqrt{\cl}-1)^{2}}{1+\cl}\,\big(p(i)+q(i)\big), & r(i)\le \cl,\\[6pt]
\big(\sqrt{r(i)}-1\big)^{2}\,q(i), & \cl<r(i)<\cu,\\[6pt]
\dfrac{(\sqrt{\cu}-1)^{2}}{1+\cu}\,\big(p(i)+q(i)\big), & r(i)\ge \cu,
\end{cases}
\end{align*}

Define the scalar function
\[
h(c):=\frac{(\sqrt{c}-1)^{2}}{1+c}=\frac{c-2\sqrt{c}+1}{1+c},\qquad c>0.
\]
Let $x=\sqrt{c}$. Then $h(c)=\dfrac{(x-1)^{2}}{1+x^{2}}$ and
\[
\frac{d}{dx}\,h(x)=\frac{2(x^{2}-1)}{(1+x^{2})^{2}},
\]
which is negative for $x<1$ (i.e.\ $c<1$), zero at $x=1$ ($c=1$), and positive for $x>1$ ($c>1$).
Hence $h$ is strictly decreasing on $(0,1]$ and strictly increasing on $[1,\infty)$.

\smallskip
\emph{Upper cutoff.}
Fix $i$ with $r(i)\ge 1$ and consider $\phi_i$ as a function of $\cu$ (with $\cl$ fixed).
If $\cu\le r(i)$, then $\phi_i(\cl,\cu)=h(\cu)\big(p(i)+q(i)\big)$, which is nondecreasing in $\cu$
because $h$ is increasing on $[1,\infty)$.
At $\cu=r(i)$, the ``clipped'' expression matches the mid expression:
\[
h(r(i))\big(p(i)+q(i)\big)=\frac{(\sqrt{r}-1)^{2}}{1+r}\,(r+1)\,q(i)=(\sqrt{r}-1)^{2}\,q(i),
\]
so $\phi_i$ is continuous when the index exits the clipped set.
For $\cu\ge r(i)$, $\phi_i$ remains constant (equal to the mid value).
Therefore $\phi_i$ is \emph{nondecreasing} in $\cu$ for every $i$.
Summing over $i$ gives that $\Hsq_{\TV}(\cl,\cu)$ is nondecreasing in $\cu$.

\smallskip
\emph{Lower cutoff.}
Fix $i$ with $r(i)\le 1$ and consider $\phi_i$ as a function of $\cl$ (with $\cu$ fixed).
If $\cl\ge r(i)$, then $\phi_i(\cl,\cu)=h(\cl)\big(p(i)+q(i)\big)$, which is nonincreasing in $\cl$
because $h$ is decreasing on $(0,1]$.
At $\cl=r(i)$, the clipped and mid expressions match (same calculation as above with $r\le 1$),
and for $\cl\le r(i)$ the contribution is constant (mid).
Thus $\phi_i$ is \emph{nonincreasing} in $\cl$ for every $i$.
Summing over $i$ gives that $\Hsq_{\TV}(\cl,\cu)$ is nonincreasing in $\cl$.

Combining the two sides proves the lemma.
\end{proof}

\subsection{Proof of Lemma~\ref{prop: approx_hell}} \label{app: approx_hell}

Clearly, we can arrive at \eqref{eq: approx_hell} by adding \eqref{eq: approx_hell_A} and \eqref{eq: approx_hell_B}. By symmetry, it is enough to prove \eqref{eq: approx_hell_A}. Observe that
    \begin{align*}
        (\sqrt{p_i} - \sqrt{q_i})^2 = \left(1 - \sqrt{\frac{q_i}{p_i}}\right)^2 p_i.
    \end{align*}
    Note that if $p_i/q_i \ge 2$, then $\left(1 - \sqrt{\frac{q_i}{p_i}}\right) \asymp 1$, and so
    \begin{align}
        \sum_{i \in A_2} (\sqrt{p_i} - \sqrt{q_i})^2 \asymp \sum_{i \in A_2} p_i = p(A_2). \label{eq: approx_hell_A_1}
    \end{align}
    On the other hand, if $p_i/q_i \in [1,2)$, then we have 
    \begin{align*}
        (\sqrt p_i - \sqrt q_i)^2 = \left(1 - \sqrt{1 - \frac{\delta_i}{p_i}}\right)^2 p_i,
        \end{align*}
        where $\delta_i = p_i-q_i$. Since $\delta_i/p_i \in [0,1]$, we may use the approximation $(\sqrt {1-x} - 1)^2 \asymp x^2$ for $x \in [0,1]$ to conclude
        \begin{align}
           \sum_{i \in A_1} (\sqrt{p_i} - \sqrt{q_i})^2 \asymp \sum_{i \in A_1} \frac{\delta_i^2}{p_i}. \label{eq: approx_hell_A_2}
        \end{align}
        Summing up \eqref{eq: approx_hell_A_1} and \eqref{eq: approx_hell_A_2} completes the proof of \eqref{eq: approx_hell_A}. The proof of \eqref{eq: approx_hell_B} follows similarly.

\subsection{Proof of Lemma~\ref{prop: monotonic_hell}}\label{app: monotonic_hell}

    By symmetry, we can prove the monotonicity of $h^2_A(\cdot)$ and a similar proof will work for $h^2_B(\cdot)$. First note that the set where $p^*_\eps$ dominates $q^*_\eps$ continues to be $A$, and similarly the set where $q^*_\eps$ dominates $p^*_\eps$  remains $B$. This is because the LFD construction \emph{never} changes a likelihood ratio above 1 to one below 1. 
    
    Now consider what happens when $\eps$ grows from 0. Suppose $i \in \bar H$ (the set of $i$ for which $p(i)/q(i) = \infty$). Observe that the contribution to $h_A^2$ from all terms in $\bar H$ is simply $p^*_\eps(\bar H)$, and this quantity monotonically decreases and becomes 0 when $\eps = p(\bar H)/(1-p(\bar H))$, and stays at 0 beyond this point. Now let us consider the contribution to $h_A$ from a fixed $i \in  A \cap \bar H^c$. Observe that the likelihood ratio $p^*_\eps(i)/q^*_\eps(i)$ starts at $p(i)/q(i)$ when $\eps = 0$, and remains constant until $\cu(\eps)$ reaches to  $p(i)/q(i)$ (which is finite since we assumed $i \notin \bar H$). After this point, the likelihood ratio becomes equal to $\cu(\eps)$, which we recall is a monotonically decreasing function of $\eps$. Thus, the likelihood ratio $p^*_\eps(i)/q^*_\eps(i)$ decreases monotonically with $\eps$, and equivalently, the term $\left(1- \sqrt{\frac{q_\eps^*(i)}{p^*_\eps(i)}}\right)^2$ also decreases monotonically. Note that the contribution from $i$ to $h_A^2$ is
    \begin{align*}
        c_{A,\eps}(i) \coloneqq \left(1- \sqrt{\frac{q_\eps^*(i)}{p^*_\eps(i)}}\right)^2 p^*_\eps(i).
    \end{align*}
This is not obviously non-decreasing, so let's look instead at $c_{A,\eps}(i)/(1+\eps)$. The ratio $p^*_\eps(i)/(1+\eps)$ is monotonically decreasing with $\eps$: It equals $p(i)$ until $\cu(\eps) \geq p(i)/q(i)$, and then it equals $q(i)\cu(\eps)$ which decreases with $\eps$. Thus, $c_{A, \eps}(i)/(1+\eps)$ is a monotonically decreasing function, and for any $\eps_1 \le \eps_2$ and any $i \in A$, we have
\begin{align*}
\frac{c_{A,\eps_1}(i)}{1+\eps_1} \geq \frac{c_{A,\eps_2}(i)}{1+\eps_2},
\end{align*}
which gives
\begin{align*}
c_{A,\eps_1}(i) \geq c_{A, \eps_2}(i)\left(\frac{1+\eps_1}{1+\eps_2} \right) \geq \frac{c_{A, \eps_2}(i)}{2}.
\end{align*}
Summing up over all $i$, this gives that for $\eps_1 \le \eps_2$, we must have
\begin{align*}
h_A^2(\eps_1) \gtrsim h_A^2(\eps_2).
\end{align*}
This completes the proof.

\section{Proofs for Section~\ref{sec: adaptive}}

\subsection{Proof of Lemma~\ref{lem:h-properties}}\label{app: h-properties}
Let $u(x)=\sqrt{p(x)}$ and $v(x)=\sqrt{q(x)}$. Then $\hstat=(u-v)/(u+v)$ and
$\hstat^2=(u-v)^2/(u+v)^2$. Claim (1) is immediate from $|u-v|\le u+v$. For
(2), we compute
\[
\mu_p-\mu_q
= \sum_x \big(p(x)-q(x)\big)\frac{u-v}{u+v}
= \sum_x (u^2-v^2)\frac{u-v}{u+v}
= \sum_x (u-v)^2
= \dhel^2(p,q).
\]
For (3), note that
\[
\E_p[\hstat^2]+\E_q[\hstat^2]
= \sum_x (u^2+v^2)\frac{(u-v)^2}{(u+v)^2}
\le \sum_x (u+v)^2\frac{(u-v)^2}{(u+v)^2}
= \sum_x (u-v)^2 = \dhel^2(p,q),
\]
so each expectation (hence each variance, since $\Var(h)\le \E[\hstat^2]$) is $\le \dhel^2(p,q)$.

\subsection{Proof of Lemma~\ref{lem:adv-shift}}\label{app: adv-shift}

Since $\hstat$ is bounded, the adversary can change the statistic by at most $2$ for each point changed. Thus, $\olh^{\Adv}$ can differ from $\olh$ by at most $2\eps$. More precisely, let \(I=\{i:x_i\neq x_i'\}\). Then \(|I|\le \lfloor n\eps\rfloor\), and
\[
\left|\olh-\olh^{\Adv}\right|
=
\left|
\frac1n\sum_{i\in I}\big(\hstat(x_i)-\hstat(x_i')\big)
\right|
\le
\frac1n\sum_{i\in I}2
\le
2\eps.
\]

\printbibliography

\end{document}